\documentclass[10pt]{article}
   \usepackage{amsmath,amsfonts,amsthm,amssymb,amscd,enumerate, url}
   \usepackage[pdftex]{graphicx,color}  
\theoremstyle{plain}
\newtheorem{thm}{Theorem}[section]
\newtheorem{prop}[thm]{Proposition}
\newtheorem{lemma}[thm]{Lemma}
\newtheorem{cor}[thm]{Corollary}

\theoremstyle{definition}
\newtheorem{defn}[thm]{Definition}
\newtheorem{ex}[thm]{Example}

\theoremstyle{remark}
\newtheorem{rmk}[thm]{Remark}

\newcommand{\vol}{\ensuremath{\mathsf{vol}}}
\newcommand{\hol}{\ensuremath{\mathrm{Hol}_g (M)}}
\newcommand{\real}{\ensuremath{\mathrm{Re}}}
\newcommand{\imag}{\ensuremath{\mathrm{Im}}}
\newcommand{\G}{\ensuremath{\mathrm{G}_2}}
\newcommand{\SP}{\ensuremath{\mathrm{Spin}(7)}}

\newcommand{\SUth}{\ensuremath{\mathrm{SU}(3)}}
\newcommand{\SOs}{\ensuremath{\mathrm{SO}(7)}}
\newcommand{\Gs}{\ensuremath{\mathrm{G}_2}{-structure}}

\newcommand{\R}{\ensuremath{\mathbb R}}
\newcommand{\C}{\ensuremath{\mathbb C}}

\newcommand{\PR}{\ensuremath{\mathbb P}}

\newcommand{\ph}{\ensuremath{\varphi}}
\newcommand{\ps}{\ensuremath{\psi}}

\newcommand{\wtws}{\ensuremath{\Omega^2_7}}
\newcommand{\wtwf}{\ensuremath{\Omega^2_{14}}}
\newcommand{\wtho}{\ensuremath{\Omega^3_1}}
\newcommand{\wths}{\ensuremath{\Omega^3_7}}
\newcommand{\wtht}{\ensuremath{\Omega^3_{27}}}

\newcommand{\wfot}{\ensuremath{\Omega^4_{27}}}

\newcommand{\st}{\ensuremath{\ast}}
\newcommand{\hk}{\mathbin{\! \hbox{\vrule height0.3pt width5pt 
depth 0.2pt \vrule height5pt width0.4pt depth 0.2pt}}}

\newcommand{\ddr}{\ensuremath{\frac{\del}{\del\:\!\! r}}}

\newcommand{\nab}[1]{\ensuremath{\nabla_{\! \! #1 \,}}}
\newcommand{\del}{\ensuremath{\partial}}

\newcommand{\dc}{\ensuremath{d_{{}_C}}}
\newcommand{\ds}{\ensuremath{d_{{}_{\Sigma}}}}
\newcommand{\dsc}{\ensuremath{d^*_{\!{}_C}}}
\newcommand{\dss}{\ensuremath{d^*_{\!{}_{\Sigma}}}}
\newcommand{\nabc}{\ensuremath{\nabla_{\!\!{}_C}}}
\newcommand{\nabs}{\ensuremath{\nabla_{\!\!{}_{\Sigma}}}}
\newcommand{\stc}{\ensuremath{\st_{{}_C}}}
\newcommand{\sts}{\ensuremath{\st_{{}_{\Sigma}}}}
\newcommand{\stsi}{\ensuremath{\st_{{}_{\Sigma_i}}}}
\newcommand{\lapc}{\ensuremath{\Delta_{{}_C}}}
\newcommand{\laps}{\ensuremath{\Delta_{{}_{\Sigma}}}}
\newcommand{\gc}{\ensuremath{g_{{}_C}}}
\newcommand{\gs}{\ensuremath{g_{{}_{\Sigma}}}}
\newcommand{\volc}{\ensuremath{\mathsf{vol}_{{}_C}}}
\newcommand{\vols}{\ensuremath{\mathsf{vol}_{{}_{\Sigma}}}}
\newcommand{\volsi}{\ensuremath{\mathsf{vol}_{{}_{\{r\} \times \Sigma_i}}}}
\newcommand{\z}{\ensuremath{\mathbf 0}}
\newcommand{\phc}{\ensuremath{\ph_{{}_{\!C}}}}
\newcommand{\psc}{\ensuremath{\ps_{{}_{\!C}}}}
\newcommand{\gci}{\ensuremath{g_{{}_{C_i}}}}
\newcommand{\phci}{\ensuremath{\ph_{{}_{\!{C_i}}}}}
\newcommand{\psci}{\ensuremath{\ps_{{}_{\!{C_i}}}}}
\newcommand{\stci}{\ensuremath{\st_{{}_{\!C_i}}}}
\newcommand{\nabci}{\ensuremath{\nabla_{\!\!{}_{C_i}}}}
\newcommand{\lapci}{\ensuremath{\Delta_{{}_{C_i}}}}
\newcommand{\dsci}{\ensuremath{d^*_{\!{}_{C_i}}}}
\newcommand{\gm}{\ensuremath{g_{{}_M}}}
\newcommand{\phm}{\ensuremath{\ph_{{}_{\!M}}}}
\newcommand{\psm}{\ensuremath{\ps_{{}_{\!M}}}}
\newcommand{\stm}{\ensuremath{\st_{{}_{\!M}}}}
\newcommand{\nabm}{\ensuremath{\nabla_{\!\!{}_{M}}}}
\newcommand{\volm}{\ensuremath{\mathsf{vol}_{{}_{M}}}}
\newcommand{\lapm}{\ensuremath{\Delta_{{}_M}}}
\newcommand{\dsm}{\ensuremath{d^*_{\!{}_M}}}
\newcommand{\gn}{\ensuremath{g_{{}_N}}}
\newcommand{\phn}{\ensuremath{\ph_{{}_{\!N}}}}
\newcommand{\psn}{\ensuremath{\ps_{{}_{\!N}}}}

\newcommand{\nabn}{\ensuremath{\nabla_{\!\!{}_{N}}}}
\newcommand{\voln}{\ensuremath{\mathsf{vol}_{{}_{N}}}}

\newcommand{\dsn}{\ensuremath{d^*_{\!{}_N}}}
\newcommand{\dshn}{\ensuremath{d^*_{\!{}_{h^*(N)}}}}
\newcommand{\gni}{\ensuremath{g_{{}_{N_i}}}}
\newcommand{\phni}{\ensuremath{\ph_{{}_{\!{N_i}}}}}
\newcommand{\psni}{\ensuremath{\ps_{{}_{\!{N_i}}}}}

\newcommand{\gnis}{\ensuremath{g_{{}_{N_{i,s}}}}}
\newcommand{\phnis}{\ensuremath{\ph_{{}_{\!{N_{i,s}}}}}}
\newcommand{\psnis}{\ensuremath{\ps_{{}_{\!{N_{i,s}}}}}}
\newcommand{\gp}{\ensuremath{g_{{}_{\!\ph}}}}
\newcommand{\stp}{\ensuremath{\st_{{}_{\!\ph}}}}
\newcommand{\Fp}{\ensuremath{F_{{}_{\!\ph}}}}
\newcommand{\Gp}{\ensuremath{G_{{}_{\!\ph}}}}
\newcommand{\Jp}{\ensuremath{J_{{}_{\!\ph}}}}
\newcommand{\Gpm}{\ensuremath{G_{\!\phm}}\!}
\newcommand{\Jpm}{\ensuremath{J_{\!\phm}}\!}
\newcommand{\Gpci}{\ensuremath{G_{\!\phci}}\!}
\newcommand{\Jpci}{\ensuremath{J_{\!\phci}}\!}
\newcommand{\g}{\ensuremath{\gamma}}
\newcommand{\e}{\ensuremath{\varepsilon}}
\newcommand{\bl}{\ensuremath{\boldsymbol{\lambda}}}
\newcommand{\im}{\ensuremath{\operatorname{im}}}
\newcommand{\spi}{\ensuremath{{\, \slash \! \! \! \! \mathcal S}}}

\numberwithin{equation}{section}
\numberwithin{table}{section}
\numberwithin{figure}{section}

\begin{document}

\title{Desingularization of $\G$~manifolds\\ with isolated conical singularities}

\author{Spiro Karigiannis\footnote{The author's research has been supported by a Marie Curie Fellowship of the European Commission under contract number MIF1-CT-2006-039113. The contents of this work reflect only the author's views and not the views of the European Commission.} \\
{\it Mathematical Institute, University of Oxford} \\ \tt{karigiannis@maths.ox.ac.uk} \\
{\it Department of Pure Mathematics, University of Waterloo} \\ \tt{karigiannis@math.uwaterloo.ca} }

\maketitle

\begin{abstract}
We present a method to desingularize a compact $\G$ manifold $M$ with isolated conical singularities by cutting out a neighbourhood of each singular point $x_i$ and glueing in an asymptotically conical $\G$~manifold $N_i$. Controlling the error on the overlap glueing region enables us to use a result of Joyce to conclude that the resulting compact smooth $7$-manifold $\tilde M$ admits a torsion-free $\G$ structure, with full $\G$~holonomy.
  
There are topological obstructions for this procedure to work, which  arise from the degree $3$ and degree $4$ cohomology of the asymptotically conical $\G$~manifolds $N_i$ which are glued in at each conical singularity. When a certain necessary topological condition on the manifold $M$ with isolated conical singularities is satisfied, we can introduce correction terms to the glueing procedure to ensure that it still works. In the case of degree $4$ obstructions, these correction terms are trivial to construct, but in the case of degree $3$ obstructions we need to solve an elliptic equation on a non-compact manifold. For this we use the Lockhart--McOwen theory of weighted Sobolev spaces on manifolds with ends. This theory is also used to obtain a good asymptotic expansion of the $\G$~structure on an asymptotically conical $\G$~manifold $N$ under an appropriate gauge-fixing condition, which is required to make the glueing procedure work.
\end{abstract}

\section{Introduction} \label{introsec}

This is the first of two papers in which we study conical singularities for $\G$~manifolds. The present paper deals with the \emph{desingularization} of $\G$~manifolds with conical singularities by glueing. The next paper~\cite{KCSmoduli} is about the \emph{deformation theory} (moduli spaces) of such manifolds, and of the closely related aymptotically conical $\G$~manifolds. We use the same notation and occasionally mention some results from~\cite{KCSmoduli}. However, this paper is completely self-contained. All the results we use from other sources are carefully stated.

The main theorem we prove in this paper is the following. The notation and terminology used in this theorem will all be defined in Sections~\ref{firstmainsec} and~\ref{desingsec}.

\noindent {\bf Main Theorem.} (Theorem~\ref{mainthm})
\emph{
Let $M$ be a compact $\G$~manifold with isolated conical singularities, with singularities $x_1, \ldots, x_n$, cones $C_1, \ldots, C_n$, and rates $\mu_1, \ldots, \mu_n$, respectively. Suppose that we have asymptotically conical $\G$~manifolds $N_1, \ldots, N_n$, with the same cones $C_1, \ldots, C_n$, and rates $\nu_1, \ldots, \nu_n$, respectively, with each $\nu_i \leq -3$. If the topological conditions~\eqref{obsconditionseq1} and~\eqref{obsconditionseq2} of Theorem~\ref{obsthm} are satisfied, then there exists a one-parameter family $\tilde M_s$ of \emph{smooth, compact $\G$~manifolds}, for $0 < s < \kappa$, (with holonomy exactly equal to $\G$), which desingularize $M$.}

The study of manifolds with isolated conical singularities (ICS) or of asymptotically conical (AC) manifolds in the context of special holonomy and calibrated geometry was initiated by Joyce in his series of papers~\cite{JSL1, JSL2, JSL3, JSL4, JSL5} about special Lagrangian submanifolds with ICS. Marshall~\cite{M} and Pacini~\cite{Pacini} studied AC special Lagrangian submanifolds of $\C^m$, and Lotay~\cite{Lth, L1, L2} studied coassociative AC and ICS submanifolds of $\R^7$. The results in the present paper most closely relate to the work of Chan~\cite{YMC1, YMC2, YMC3} on AC and ICS Calabi--Yau $3$-folds.

For those readers who are familiar with the papers~\cite{YMC1, YMC3}, the main differences between the $\G$~case of the present paper and the Calabi--Yau $3$-fold case studied by Chan are the following. There are two sources of topological obstructions, coming from the $3$-form $\ph$ and the $4$-form $\ps$ of a $\G$~structure, compared to just the one source coming from $\Omega$ in the Calabi--Yau case. The solution to $3$-form obstructions is very similar to the approach of Chan~\cite{YMC3}, except that we can use the full $d + d^*$ operator rather than having to restrict to a subcomplex, but we need to also obtain excluded ranges for the orders of homogeneity for harmonic $2$-forms, in addition to the $0$-forms and $1$-forms which are needed in the Calabi--Yau $3$-fold case. The solution to $4$-form obstructions cannot be obtained in the same way, for technical reasons, but there turns out to be a much easier way to deal with those. Also, Chan assumes that the asymptotically conical manifolds used in the glueing already admit a good asymptotic expansion near infinity (and he shows that this does hold in the examples he discusses.) We prove that under a natural gauge-fixing condition, a good asymptotic expansion which allows the glueing construction to succeed always exists.

We now discuss the organization of this paper. Section~\ref{reviewsec} is a review of the main facts about $\G$~manifolds which we require. More thorough discussions of $\G$~structures can be found in Bryant~\cite{Br1} and Joyce~\cite{J4}.

Section~\ref{firstmainsec} is about $\G$~manifolds modelled on cones. In Section~\ref{g2conessec} we treat $\G$~cones, and in Section~\ref{conesformssec} we present several results about differential forms on $\G$~cones which we will need later. Sections~\ref{CSsec} and~\ref{ACsec} discuss, respectively, $\G$~manifolds with isolated conical singularities (ICS), and asymptotically conical $\G$~manifolds (AC), including the known explicit examples in the AC case.

In Section~\ref{desingsec} we present the details of the desingularization procedure as three steps. Step one in Section~\ref{manifoldconstructionsec} is concerned with constructing a compact smooth manifold $\tilde M_s$ for all $s$ sufficiently small. Then in step two in Section~\ref{formsconstructionsec} we construct a family $\ph_s$ of closed $\G$~structures on $\tilde M_s$ with small torsion. It is at this stage where we require two analytic results, which are described here, but whose proof is postponed until later in the paper. The first result involves the existence of a good asymptotic expansion near infinity for the $\G$~structure of an asymptotically conical $\G$~manifold. The second result is concerned with topological obstructions to the glueing procedure that can arise, and their resolution. Finally in step three of the desingularization procedure in Section~\ref{torsionfreesec} we obtain estimates to show that the torsion of $\ph_s$ is small enough to invoke a theorem of Joyce that gives the existence of a torsion-free $\G$~structure $\tilde \ph_s$ on $\tilde M_s$ for $s$ sufficiently small.

In Section~\ref{analsec}, we briefly review and summarize the relevant results which we will need from the Lockhart--McOwen theory of weighted Sobolev spaces on non-compact manifolds with ends. Section~\ref{lockhartCSsec} is about manifolds with isolated conical singularities, while Section~\ref{lockhartACsec} is about asymptotically conical manifolds.

In Section~\ref{obstructionsec} we discuss how to overcome the topological obstructions to our glueing procedure of Section~\ref{formsconstructionsec} when a certain necessary condition is satisfied, by explaining how to explicity construct `correction forms' to be able to carry out the construction in the obstructed case. In Section~\ref{obssubsection} we deal with the case of $3$-form obstructions, and in Section~\ref{obssubsection2} we handle the $4$-form obstructions.

Finally, in Section~\ref{ACgaugefixsec} we prove that a good asymptotic expansion exists for an asymptotically conical $\G$~manifold which satisfies a natural gauge-fixing condition.

{\bf Conventions.} There are two sign conventions in $\G$~geometry. The convention we choose to use is the one used in Bryant--Salamon~\cite{BS} and in Harvey--Lawson~\cite{HL}, but differs from the convention used in Bryant~\cite{Br1} or Joyce~\cite{J4}. A detailed discussion of sign conventions and orientations in $\G$~geometry can be found in the author's note~\cite{K2}.

The letter $C$ is used in two ways: to denote a \emph{cone}, or to denote a \emph{positive constant}. The use will always be clear from the context. When we are estimating various quantities, the value of $C > 0$ will change from step to step in the calculations, but we will always use the same letter $C$ to denote it, to avoid the proliferation of too much notation. We use $f = O(r)$ to mean that $f \leq C r$ for some $C > 0$ as $r \to 0$ (when we are working near a singularity), or as $r \to \infty$ (when we are working on an asymptotic end at infinity.)

One fact we will use repeatedly is that for any form $\omega$, the exterior derivative $d\omega$ is obtained by the skew-symmetrization of the covariant derivative $\nab{} \omega$, so there is some $C > 0$ such that
\begin{equation} \label{desteq}
|d \omega| \leq C | \nab{} \omega| , \qquad \qquad |\nab{} d \omega| \leq C | \nab{}^2 \omega|.
\end{equation}
Also, we have
\begin{equation*}
|d^*\omega| = |\st d \st \omega| = |d \st \omega| \leq C | \nab{}  \st \omega| = C | \st \nab{} \omega| = C | \nab{} \omega|
\end{equation*}
using the fact that $d^* = \pm \st d \st$, and that $\st$ is an isometry and commutes with $\nab{}$. Thus we also have
\begin{equation} \label{dstesteq}
|d^* \omega| \leq C | \nab{} \omega|.
\end{equation}

{\bf Acknowledgements.} The author is extremely grateful to Dominic Joyce for his tireless patience in explaining the technical details about weighted Banach spaces and glueing constructions, for his many useful suggestions, and for providing the initial overall direction for this project.  Useful discussions were also had with Mark Haskins, Tommaso Pacini, and Yinan Song. The author also thanks the European Commission for its support through the Marie Curie Programme.

\subsection{Review of \G~manifolds} \label{reviewsec}

Recall that a $\G$~structure on a connected smooth $7$-manifold $M$ is described by a smooth $3$-form $\ph$ satisfying a certain ``non-degeneracy'' condition. When such a structure exists, there is an open subbundle $\Omega^3_+$ of the bundle $\Omega^3$ of $3$-forms consisting of non-degenerate $3$-forms, also called \emph{positive} or \emph{stable} $3$-forms.  A $\G$~structure $\ph$ determines a Riemannian metric $g_{\ph}$ and an orientation $\vol_{\ph}$ in a non-linear way. Thus $\ph$ determines a Hodge star operator $\st_{\ph}$, and $\ps = \st_{\ph} \ph$ is the dual $4$-form which is an element of the bundle $\Omega^4_+$ of positive $4$-forms.

\begin{defn} \label{g2manifolddefn}
A $\G$~manifold is a connected manifold with a $\G$~structure $(M, \ph)$ such that $\ph$ is \emph{parallel} with respect to the Levi-Civita connection $\nab{}$ determined by $g_{\ph}$. That is, $\nab{g_{\ph}}\ph = 0$. Such a $\G$~structure is also called \emph{torsion-free}. In this case the Riemannian
holonomy $\hol$ of $(M, g_{\ph})$ is contained in the group $\G \subseteq \SOs$. If the fundamental group $\pi_1(M)$ is finite, then the holonomy of $g_{\ph}$ is exactly $\G$. (See Joyce~\cite{J4} for more details.)
\end{defn}
\begin{rmk} \label{g2manifoldrmk}
A $\G$~manifold is always \emph{Ricci-flat}, and by Fern\'andez--Gray~\cite{FG}, a $\G$~structure $\ph$ is torsion-free if and only if it is both closed and coclosed: $d\ph = 0$ and $d \ps = 0$.
\end{rmk}

The space $\Omega^3$ of $3$-forms on a manifold with $\G$~structure $\ph$ decomposes as
\begin{equation} \label{lambda3eq}
\Omega^3 = \wtho \oplus \wths \oplus \wtht 
\end{equation}
into irreducible $\G$~representations. Similarly we have a decomposition of the space $\Omega^2$ as
\begin{equation} \label{lambda2eq}
\Omega^2 = \wtws \oplus \wtwf,
\end{equation}
as well as isomorphic splittings of $\Lambda^4$ and $\Lambda^5$ given by the Hodge star of the above decompositions: $\Lambda^k_l = \st_{\ph} (\Lambda^{7-k}_l)$. When the $\G$~structure is torsion-free, these decompositions are preserved by the Hodge Laplacian $\Delta = d d^* + d^* d$. The essential aspect of this fact that we will need is the following. Suppose $f$ is any function and $\omega$ is any $1$-form on a $\G$~manifold $M$. Then
\begin{align} \label{temp27eq1}
\Delta (f \ph) & = (\Delta f) \ph, & \Delta (f \ps) & = (\Delta f) \ps, & \\ \label{temp27eq2} \Delta (\omega \wedge \ph) & = (\Delta \omega) \wedge \ph, & \Delta (\omega \wedge \ps) & = (\Delta \omega) \wedge \ps.
\end{align}
The identities in~\eqref{temp27eq1} can be proved using just the fact that $\ph$ and $\ps$ are parallel, while the identities in~\eqref{temp27eq2} also require the fact that $\G$~manifolds have vanishing Ricci curvature.

Let $\Theta : \Lambda^3_+ \to \Lambda^4_+$ be the non-linear map which associates to any $\G$~structure $\ph$, the dual $4$-form $\ps = \Theta(\ph) = \st_{\ph} \ph$ with respect to the metric $g_{\ph}$ associated to $\ph$. One result which will be crucial is the following. This is Proposition 10.3.5 in Joyce~\cite{J4}, adapted to suit our present purposes.

\begin{lemma} \label{quadlemma}
Suppose that $\ph$ is a $\G$~structure with induced metric $\gp$, and dual $4$-form $\ps = \st_{\ph}  \ph$. Let $\xi$ be a $3$-form which has sufficiently small pointwise norm with respect to $\gp$, so that $\ph + \xi$ is still non-degenerate. Then we have
\begin{equation} \label{quadeq}
\Theta(\ph + \xi) = \ps + \stp \left( \frac{4}{3} \pi_1 (\xi) + \pi_7 (\xi) - \pi_{27} (\xi) \right) + \Fp(\xi),
\end{equation}
where $\pi_k$ is the projection onto the subspace $\Omega^3_k$ with respect to the $\G$~structure $\ph$. The nonlinear function $\Fp : \Omega^3 \to \Omega^4$ satisfies
\begin{equation} \label{quadeq2}
\Fp(0) = 0, \qquad \qquad | \Fp(\xi) | \leq C |\xi|^2, \qquad \qquad |\nab{} \Fp (\xi)| \leq C \left( |\xi|^2 |\nab{} \ph| + |\xi| |\nab{} \xi| \right),
\end{equation}
for some $C > 0$, where the norms and the covariant derivatives are taken with respect to $\gp$.
\end{lemma}
\begin{rmk} \label{quadrmk}
It is clear that the same sort of result holds for the map $\Theta^{-1} : \Omega^4_+ \to \Omega^3_+$.
That is, if $\eta$ is a $4$-form with sufficiently small pointwise norm, then
\begin{equation} \label{quadeq3}
\Theta^{-1}(\ps + \eta) = \ph + \stp \left( \frac{3}{4} \pi_1 (\eta) + \pi_7 (\eta) - \pi_{27} (\eta) \right) + \Gp(\eta),
\end{equation}
where $\pi_k$ is the projection onto $\Omega^4_k$ with respect to $\ph$, and $\Gp : \Omega^4 \to \Omega^3$ satisfies
\begin{equation} \label{quadeq4}
\Gp(0) = 0, \qquad \qquad | \Gp(\eta) | \leq C |\eta|^2, \qquad \qquad |\nab{} \Gp (\eta)| \leq C \left( |\eta|^2 |\nab{} \ph| + |\eta| |\nab{} \eta| \right),
\end{equation}
for some $C > 0$.
\end{rmk}

It will be convenient to define the linear operator $J: \Omega^4 \to \Omega^3$ given by
\begin{equation} \label{Jdefneq}
\Jp(\eta) =  \stp \left( \frac{3}{4} \pi_1 (\eta) + \pi_7 (\eta) - \pi_{27} (\eta) \right).
\end{equation}
Then equation~\eqref{quadeq3} becomes $\Theta^{-1}(\ps + \eta) = \ph + \Jp(\eta) + \Gp(\eta)$. Given explicit formulas for $\pi_1$ and $\pi_7$, as can be found for example in~\cite{K1}, it is not difficult to verify that
\begin{equation} \label{Jesteq}
|\Jp(\eta)| \leq C |\eta|,  \qquad \qquad |\nab{} \Jp (\eta)| \leq C \left( |\eta| |\nab{} \ph| + |\nab{} \eta| \right),
\end{equation}
which we will need in Section~\ref{torsionfreesec} to estimate the torsion of the $\G$~structure we construct by glueing in Section~\ref{formsconstructionsec}.

\section{$\G$~manifolds modeled on cones} \label{firstmainsec}

In this section we discuss $\G$~manifolds which are modeled on cones. Specifically these are: $\G$~cones, $\G$~manifolds with isolated conical singularities (ICS), and asymptotically conical $\G$~manifolds (AC).

\subsection{$\G$~cones} \label{g2conessec}

Let $\Sigma^6$ be a compact, connected, smooth $6$-manifold.  An \SUth~structure on $\Sigma$ is described by a Riemannian metric $\gs$, an almost complex structure $J$ which is
orthogonal with respect to $\gs$, the associated K\"ahler form $\omega(u,v) = \gs(Ju,v)$ which is real and of type $(1,1)$ with respect to $J$, and a non-vanishing complex $(3,0)$-form $\Omega$. The
two forms are related by the normalization condition
\begin{equation} \label{su3compatibilityeq}
\vols = \frac{1}{6} \omega^3 = \frac{i}{8} \Omega \wedge \bar \Omega
= \frac{1}{4} \real(\Omega) \wedge \imag(\Omega).
\end{equation}
Such a structure exists whenever $\Sigma$ admits an almost complex structure and $c_1(\Sigma) = 0$, so that the canonical bundle $\Lambda^{3,0}$ of $\Sigma$ is topologically trivial and a non-vanishing section $\Omega$ exists. Then we can always scale $\Omega$ by a non-zero complex valued function to ensure
that~\eqref{su3compatibilityeq} holds.
\begin{defn} \label{weaksu3defn}
The manifold $\Sigma^6$ with $\SUth$ structure $(\omega, \Omega)$ is called \emph{strictly nearly K\"ahler}, or a \emph{weak \SUth manifold}, if the following equations are satisfied:
\begin{equation} \label{weaksu3eq}
\ds \omega = - 3 \, \real(\Omega), \qquad \qquad \ds \imag(\Omega) = 4 \, 
\frac{\omega^2}{2}.
\end{equation}
\end{defn}

\begin{rmk} \label{weaksu3rmk}
All strictly nearly K\"ahler $6$-manifolds can be shown to be \emph{Einstein}, with \emph{positive} scalar
curvature. It then follows from the Weitzenb\"ock formula that in the compact case, the first Betti number vanishes: $b_1 (\Sigma) = 0$. Some references for strictly nearly K\"ahler manifolds are B\"ar~\cite{Bar}, Reyes-Carri\'on--Salamon~\cite{CaS}, and Gray~\cite{Gr1, Gr2}.
\end{rmk}

\begin{defn} \label{g2conedefn}
Let $\Sigma^6$ be strictly nearly K\"ahler. Then there exists a torsion-free $\G$~structure $(\phc, \psc, \gc)$ on $C' = (0, \infty) \times \Sigma$ with holonomy exactly equal to $\G$. This structure is defined by
\begin{align} \label{phceq}
\phc & = r^3 \real(\Omega) - r^2 dr \wedge \omega, \\ \label{psceq} \psc & = - r^3 dr \wedge \imag(\Omega) - r^ 4\frac{\omega^2}{2}, \\ \label{gceq} \gc & = dr^2 + r^2 \gs,
\end{align}
where $r$ is the coordinate on $(0, \infty)$. We say that $C = C' \cup \{\z\}$ is a $\G$ \emph{cone}.
\end{defn}
The space $C = C' \cup \{\z\}$ is a $\G$~\emph{cone}, $C'$ is the \emph{smooth part} of $C$, the point $\z$ is the \emph{singular point} of the cone, and $\Sigma$ is called the \emph{link} of the cone. Notice that $\z = \lim_{r \to 0^+} (r, \sigma)$ for any point $\sigma$ in $\Sigma$. Now $\det(\gc) = r^{12} \det(\gs)$, and we choose the orientation on $C'$ so that
\begin{equation} \label{volceq}
\volc = r^{6} dr \wedge \vols
\end{equation}
is the volume form on $C'$.

\begin{rmk} \label{coneholonomyrmk}
It is known that for a Riemannian cone $C = \Sigma \times(0, \infty)$ with holonomy contained in $\G$ that the holonomy is either trivial, in which case $\Sigma$ is the standard round sphere $S^6$ and $C$ is the Euclidean $\R^7$, or else the holonomy is exactly equal to $\G$, in which case the link $\Sigma$ is strictly nearly K\"ahler, but not equal to $S^6$. See B\"ar~\cite{Bar} for more details. We reiterate that for us, a $\G$~cone will always have holonomy exactly $\G$, so we exclude the case where the link is $S^6$.
\end{rmk}

For any $t > 0$, we have a \emph{dilation} map $\mathbf{t} : C \to C$ defined by
\begin{equation} \label{dilationeq}
\mathbf{t} (\z) = \z, \qquad \qquad \mathbf{t} (r, \sigma) = (t r, \sigma).
\end{equation}
If we pull back $\phc$, $\psc$, $\gc$, and $\volc$ with respect to $\mathbf{t}$, we see that
\begin{equation} \label{dilationequiveq}
\begin{aligned}
\mathbf{t}^* (\phc) & = t^3 \, \phc, \qquad & \mathbf{t}^* (\psc) & = t^4 \, \psc, & \\ \mathbf{t}^* (\gc) & = t^2 \, \gc, \qquad & \mathbf{t}^* (\volc) & = t^7 \, \volc, &
\end{aligned}
\end{equation}
and we say that the conical $\G$~structure is \emph{dilation-equivariant}. Similarly any contravariant tensor $\omega$ of degree $k$ such that $\mathbf{t}^* (\omega) = t^k \, \omega$ is called dilation-equivariant. A useful property of dilations which we will use frequently is the following. Since $\gci (r , \sigma) = t^{-2} \gci (tr, \sigma)$ for vector fields, we see that
\begin{equation} \label{scalemetriceq}
|\mathbf{t}^*(\omega)(r, \sigma)|_{\gci(r, \sigma)} = t^k | \omega(tr, \sigma) |_{\gci(tr, \sigma)}
\end{equation}
whenever $\omega$ is a contravariant tensor of degree $k$.

\begin{prop} \label{coneformsexactprop}
The forms $\phc$ and $\psc$ on a $\G$~cone $C$ are \emph{exact}, and hence the cohomology classes $[\phc]$ and $[\psc]$ are trivial in $H^3(C, \R) \cong H^3(\Sigma, \R)$ and $H^4(C, \R) \cong H^4(\Sigma, \R)$, respectively.
\end{prop}
\begin{proof}
By using the relations~\eqref{weaksu3eq} defining a strictly nearly K\"ahler structure on the link $\Sigma$, it is easy to check that~\eqref{phceq} and~\eqref{psceq} can be written as
\begin{equation*}
\phc = d \left( - \frac{r^3}{3} \omega \right), \qquad \qquad \psc = d \left( - \frac{r^4}{4} \imag (\Omega) \right),
\end{equation*}
which is exactly what we needed to show.
\end{proof}

There are three known compact strictly nearly K\"ahler manifolds (other than the round $S^6$), and hence three known $\G$ cones, which we now discuss. They are all obtained by taking the biinvariant metric on a compact Lie group $G$ and descending this to the normal metric on $G/H$ for an appropriate Lie subgroup $H$. In particular, all these examples are homogeneous spaces. See B\"ar~\cite{Bar} for more details.

\begin{ex} \label{cp3ex}
{\bf The complex projective space: $\C \PR^3$.}
We can view $\C \PR^3$ (diffeomorphically) as the homogeneous space $\mathrm{Sp}(2) / (\mathrm{Sp}(1) \times \mathrm{U}(1))$. There is a natural non-integrable complex structure on this space, different from the standard one, and the normal metric is not the Fubini--Study metric, so this $\C \PR^3$ is \emph{not} K\"ahler. This space is also the \emph{twistor space} of the standard round $S^4$, which is the unit sphere subbundle of the bundle $\Lambda^2_-(S^4)$ of anti-self dual $2$-forms on $S^4$.
\end{ex}

\begin{ex} \label{flagex}
{\bf The complex flag manifold: $F_{1,2} = \SUth/T^2$.} This is the homogeneous space $\mathrm{SU}(3)/ T^2$ where $T^2 = \mathrm{S} ( \mathrm{U}(1) \times \mathrm{U}(1) \times \mathrm{U}(1) ) \subset \mathrm{SU}(3)$ is the maximal torus. This is the space of pairs $(V_1, V_2)$ where $V_1$ and $V_2$ are complex linear subspaces of $\C^3$ with $V_1 \subset V_2$.  This space is also the \emph{twistor space} of the standard Fubini--Study $\C \PR^2$, which is the unit sphere subbundle of the bundle $\Lambda^2_-(\C \PR^2)$ of anti-self dual $2$-forms on $\C \PR^2$.
\end{ex}

\begin{ex} \label{s3s3ex}
{\bf The product of two $3$-spheres: $S^3 \times S^3$.} We view $S^3$ as the Lie group of unit quaternions, which is isomorphic to $\mathrm{SU}(2)$. Then let $S^3 \times S^3 = (S^3 \times S^3 \times S^3) / S^3$ where we embed $S^3$ into $S^3 \times S^3 \times S^3$ as the diagonal subgroup.  This space is also the unit sphere subbundle of the \emph{spinor bundle} $\spi (S^3)$ of the standard round $S^3$.
\end{ex}

\begin{rmk} \label{NKcohomologyrmk}
We note for later use that Examples~\ref{cp3ex} and~\ref{flagex} both have $H^3(\Sigma, \R) = 0$ and $H^4 (\Sigma, \R) \neq 0$, and that Example~\ref{s3s3ex} has $H^3(\Sigma, \R) \neq 0$ and $H^4 (\Sigma, \R) = 0$.
\end{rmk}

\begin{rmk} \label{onlyhomormk}
In Butruille~\cite{But} it is proved that the above examples are the only \emph{homogeneous} strictly nearly K\"ahler compact manifolds. It is expected that there should exist many non-homogeneous examples, but as far as the author is aware, none have been found yet.
\end{rmk}

\subsection{Differential forms on $\G$~cones} \label{conesformssec}

We will denote by $\sts$, $\nabs$, $\ds$, $\dss$, and $\laps$ the Hodge star, Levi-Civita connection, exterior derivative, coderivative, and Hodge Laplacian of $\Sigma$. Similarly $\stc$, $\nabc$, $\dc$, $\dsc$, and $\lapc$ will denote the analogous operators for the smooth part $C'$ of the cone $C$. We will often abuse notation and say ``form on the cone $C$'' when we really mean ``form on the smooth part $C'$ of $C$.''  The following proposition is a simple exercise. 
\begin{prop} \label{coneformulasprop}
Let $\omega$ be a smooth $k$-form on $C$. Then we can write
\begin{equation} \label{coneformeq}
\omega = dr \wedge \alpha + \beta
\end{equation}
where $\alpha$ is a $(k-1)$-form and $\beta$ is a $k$-form on $\Sigma$, both depending on $r$ as a parameter. We use $'$ to denote differentiation with respect to the parameter $r$. The following formulas hold:
\begin{align} \label{conedeq}
\dc (dr \wedge \alpha + \beta) & \, = \, dr \wedge \left( \beta' - \ds \alpha \right) + (\ds \beta), \\ \label{conestareq} \stc (dr \wedge \alpha + \beta) & \, = \, dr \wedge \left( (-1)^k r^{6 - 2k} \sts \beta \right)
+ \left( r^{8 - 2k} \sts \alpha \right), \\ \label{conedleq} \dsc (dr \wedge \alpha + \beta) & \, = \, dr \wedge \left( - \frac{1}{r^2} \dss \alpha \right) + \left( - \frac{(8 - 2k)}{r} \alpha - \alpha' + \frac{1}{r^2} \dss
\beta \right),
\end{align}
and finally
\begin{equation} \label{conelaplacianeq}
\begin{aligned}
\lapc (dr \wedge \alpha + \beta) & \, = \, dr \wedge \left( \frac{1}{r^2} \laps \alpha + \frac{(8 - 2k)}{r^2} \alpha - \frac{(8 - 2k)}{r} \alpha' - \alpha'' - \frac{2}{r^3} \dss \beta \right) \\ & \qquad + \left( \frac{1}{r^2} \laps \beta - \frac{(6 - 2k)}{r} \beta' - \beta'' - \frac{2}{r} \ds \alpha \right).
\end{aligned}
\end{equation}
\end{prop}

Suppose now that $\alpha$ is a $(k-1)$-form on $\Sigma$ and $\beta$ is a $k$-form on $\Sigma$. Then from~\eqref{gceq} it is easy to see that
\begin{equation} \label{formsconemetriceq2}
|dr \wedge \alpha|_{\gc} = r^{-(k-1)} \, |\alpha|_{\gs} \qquad \text{and} \qquad |\beta|_{\gc} = r^{-k} \, |\beta|_{\gs},
\end{equation}
from which it follows that
\begin{equation} \label{formsconemetriceq}
|r^{k-1}dr \wedge \alpha + r^k \beta|^2_{\gc} \, = \, |\alpha|^2_{\gs} + |\beta|^2_{\gs}.
\end{equation}
For this reason, we will always write a $k$-form on the cone $C'$ in the form $\omega = r^{k-1}dr \wedge \alpha + r^k \beta$ for some $\alpha$ and $\beta$, which are forms on $\Sigma$ possibly depending on the parameter $r$. Note that if $\alpha$ and $\beta$ were independent of $r$, then $\omega$ would be dilation-equivariant, as defined above.

The next lemma is an important result about closed differential forms on a cone $C$ with certain growth rates near $0$ or $\infty$.
\begin{lemma} \label{exactformslemma}
Let $\omega$ be a smooth \emph{closed} $k$-form on $C' = (0, \infty) \times
\Sigma$. Suppose that either
\begin{align*}
\text{i) \, } | \omega |_{\gc} & = O(r^{\lambda}) \text{ on } (0, \e) \times \Sigma, {\text \, for \, \, } \lambda > -k, \text{ or } \\ \text{ii) \, } | \omega |_{\gc} & = O(r^{\lambda}) \text{ on } (R, \infty) \times \Sigma, {\text \, for \, \, } \lambda < -k.
\end{align*}
for some small $\e$ or some large $R$. Then for each case respectively we have that
\begin{align*}
\text{i) \, } \omega & = d\Omega \text{ for some $(k-1)$-form } \Omega \text{ on } (0, \e) \times \Sigma, \text{ or } \\ \text{ii) \, } \omega & = d\Omega \text{ for some $(k-1)$-form } \Omega \text{ on } (R, \infty) \times \Sigma,
\end{align*}
where in each case, on its domain of definition, $\Omega$ satisfies $|\Omega|_{\gc} = O(r^{\lambda + 1})$, as long as $\lambda \neq -1$. Furthermore, now suppose that $0 < \lambda < 1$, or $\lambda < -1$. If we also know that $|\nabc^j \omega|_{\gc} = O(r^{\lambda - j})$, then for
this $\Omega$ we have $|\nabc^j \Omega|_{\gc} = O(r^{\lambda + 1 - j})$ for all $j \geq 0$. 
\end{lemma}
\begin{proof}
Write $\omega(r, \sigma) = dr \wedge \alpha (r, \sigma) + \beta (r, \sigma)$, where $\alpha$ and $\beta$ are a $(k-1)$-form and a $k$-form on $\Sigma$, respectively. Since $d\omega = 0$, by equation~\eqref{conedeq} we have
\begin{equation} \label{exacttempeq}
\ds \beta = 0, \qquad \qquad \ds \alpha = \beta'.
\end{equation}
For case i) let us define
\begin{equation*}
\Omega(r, \sigma) = \int_0^r  \alpha (t, \sigma)dt.
\end{equation*}
This integral converges because $|\Omega(r, \sigma)|_{\gs} \leq \int_0^r |\alpha(t, \sigma)|_{\gs} dt \leq C \int_0^r t^{k-1 + \lambda} dt < \infty$, where we have used~\eqref{formsconemetriceq2} and the facts that $|\alpha(r,\sigma)|_{\gc(r, \sigma)} \leq C r^{\lambda}$ and $\lambda + k > 0$ by hypothesis. Note that these hypotheses also show that
\begin{equation} \label{exacttempeq1b}
\lim_{r \to 0} \beta(r, \sigma) \, = \, 0.
\end{equation}
It is clear (since we assume $\lambda \neq -1$) that $|\Omega|_{\gc} = O(r^{\lambda + 1})$, and so in particular also
\begin{equation} \label{exacttempeq2}
|d\Omega|_{\gc} = O(r^{\lambda}).
\end{equation}
Now we compute that
\begin{align*}
d\Omega (r, \sigma) & = dr \wedge \alpha (r, \sigma) + \int_{r_0}^r \ds \alpha (t, \sigma) dt \\ & = dr \wedge \alpha (r, \sigma) + \int_{r_0}^r \beta' (t, \sigma) dt = dr \wedge \alpha (r, \sigma) + \beta (r, \sigma) - \lim_{r \to 0} \beta(r, \sigma) \\ & = \omega(r, \sigma)
\end{align*}
using~\eqref{exacttempeq} and~\eqref{exacttempeq1b}. The fact that $|\nabc^j \Omega|_{\gc} = O(r^{\lambda + 1- j})$ if we know that $|\nabc^j \omega|_{\gc} = O(r^{\lambda - j})$ is a simple exercise. The constraints on $\lambda$ are sufficient to ensure no logarithmic terms arise from having to integrate $t^{-1}$. The proof of case ii) is analogous, defining in this case $\Omega(r, \sigma) = \int_{\infty}^r  \alpha (t, \sigma)dt$ and using the fact that $\lambda + k < 0$ to ensure that the integral converges and that $\lim_{r \to \infty} \beta(r, \sigma) = 0$.
\end{proof}

\begin{defn} \label{homoforms}
We say that a smooth $k$ form $\omega$ on $C'$ is \emph{homogeneous of order} $\lambda$ if
\begin{equation} \label{homoformseq1}
\omega = r^{\lambda} \left( r^{k-1}dr \wedge \alpha + r^k \beta \right)
\end{equation}
where $\alpha$ and $\beta$ are forms on $\Sigma$, independent of $r$. Then
we see that
\begin{equation*}
| \mathbf{t}^* (\omega) (r, \sigma)|_{\gc(tr, \sigma)} = | t^{\lambda + k} \omega (r, \sigma) |_{\gc(tr, \sigma)} =  t^{\lambda + k} t^{-k} | \omega (r, \sigma) |_{\gc(r, \sigma)},
\end{equation*}
which we can write more concisely as
\begin{equation} \label{homoformseq2}
\mathbf{t}^* |\omega|_{\gc} = t^{\lambda} |\omega|_{\gc},
\end{equation}
so the function $|\omega|_{\gc}$ on $C'$ is homogeneous of order $\lambda$ in the variable $r$ in the usual sense.
\end{defn}
\begin{rmk} \label{homoformsfirstrmk}
It is easy to see that a homogeneous $k$-form $\omega$ of order $\lambda$ is dilation-equivariant ($\mathbf{t}^*(\omega) = t^k \omega)$ if $\lambda = 0$. Similarly it is \emph{dilation-invariant} ($\mathbf{t}^*(\omega) = \omega)$ if $\lambda = - k$.
\end{rmk}
\begin{rmk} \label{homoformsstarrmk}
It follows directly from~\eqref{conestareq} that if $\omega$ is homogeneous of order $\lambda$, then $\stc \omega$ is also homogenous of the same order $\lambda$.
\end{rmk}

Using equations~\eqref{conedeq},~\eqref{conedleq}, and~\eqref{conelaplacianeq}, for a homogenous $k$-form $\omega = r^{\lambda} \left( r^{k-1}dr \wedge \alpha + r^k \beta \right)$ of order $\lambda$, we find that:
\begin{align} \label{homodeq}
\dc \omega = \, & \, r^{\lambda + k - 1} dr \wedge ( (\lambda + k) \beta - \ds \alpha) + r^{\lambda + k} \ds \beta, \\ \label{homodleq} \dsc \omega = \, & \, r^{\lambda + k - 3} dr \wedge ( -\dss \alpha) + r^{\lambda + k - 2} ( -(\lambda - k + 7) \alpha + \dss \beta ),
\end{align}
\begin{equation} \label{homolapeq}
\begin{aligned} 
\lapc \omega & \, = \, r^{\lambda + k - 3} dr \wedge \left( \laps \alpha  - (\lambda + k - 2)(\lambda - k + 7) \alpha - 2 \dss \beta \right) \\ & \qquad {} + r^{\lambda + k - 2} \left( \laps \beta  - (\lambda + k)(\lambda - k + 5) \beta - 2 \ds \alpha \right).
\end{aligned}
\end{equation}

In Section~\ref{obssubsection},  we will need to consider the possible order $\lambda$ of a homogeneous $k$-form $\omega_k$ on a cone $C$ which is in the kernel of $\lapc$, or of a mixed degree form $\omega = \sum_{k=0}^7 \omega_k$ which is in the kernel of $\dc + \dsc$. The fact that we will need to use repeatedly in the rest of this section is that the Laplacian $\laps$ on the link $\Sigma$ has non-negative eigenvalues, so whenever we have an expression of the form $\laps \g = \mu \g$ for some $\mu < 0$ and $\g$ a form on $\Sigma$, we must have $\g = 0$.

\begin{prop} \label{homolapexcludeprop}
Let $\omega$ be a homogeneous $k$-form of order $\lambda$ which is harmonic on the cone: $\lapc \omega = 0$. Then we have:
\begin{align} \label{k07lapexcludeeq}
& \text{For } k = 0, 7, \qquad \qquad \omega = 0 \, \, \text{if } \, \lambda \in (-5,0), & \\ \label{k16lapexcludeeq} & \text{For } k = 1, 6, \qquad \qquad \omega = 0 \, \, \text{if } \, \lambda \in (-4,-1), & \\ \label{k25lapexcludeeq} & \text{For } k = 2, 5, \qquad \qquad \omega = 0 \, \, \text{if } \, \lambda \in (-3,-2). &
\end{align}
This gives an excluded range of orders of homogeneity for harmonic forms on the cone $C$.
\end{prop}
\begin{proof}
Since the star operator $\stc$ commutes with the Laplacian $\lapc$ and preserves the order of homogeneity by Remark~\ref{homoformsstarrmk}, it suffices to check the cases $k=0,1,2$. Let $k=0$ in equation~\eqref{homolapeq}, and noting that $\alpha_0 = 0$ and $\dss \beta_0 = 0$, we see that $\lapc \omega = 0$ implies
\begin{equation*}
\laps \beta_0 = \lambda(\lambda + 5) \beta_0,
\end{equation*}
and the result now follows for $\lambda \in (-5,0)$. We need to work a little harder for $k=1$ and $k=2$. 

For $k=1$, we get the following system of equations:
\begin{equation} \label{k16lapexcludetempeq}
\laps \alpha_1 = (\lambda-1) (\lambda+6) \alpha_1 + 2 \dss \beta_1, \qquad \qquad \laps \beta_1 = (\lambda+1) (\lambda+4) \beta_1 + 2 \ds \alpha_1.
\end{equation}
Now suppose that $\lambda \in (-4,-1)$. If we take $\ds$ of the second equation, since $\laps$ commutes with $\ds$, we get $\laps (\ds \beta_1) = (\lambda + 1) (\lambda + 4) \ds \beta_1$ and hence $\ds \beta_1 = 0$. Similarly taking $\dss$ of the first equation gives $\dss \alpha_1 = 0$. Also, taking $\ds$ of the first equation gives
\begin{equation*}
\laps (\ds \alpha_1) = (\lambda-1)(\lambda+6) \ds \alpha_1 + 2 \laps \beta_1 = (\lambda^2 + 5 \lambda - 2) \ds \alpha_1 + (2 \lambda^2 + 10\lambda + 8) \beta_1.
\end{equation*}
Now an easy calculation reveals that
\begin{equation*}
\laps ( \ds \alpha_1 + (\lambda + 4) \beta_1  ) = (\lambda + 1)(\lambda + 6) ( \ds \alpha_1 + (\lambda + 4) \beta_1),
\end{equation*}
and hence $\ds \alpha_1 + (\lambda + 4) \beta_1 = 0$ for $\lambda \in (-4, -1)$. This relation can now be substituted into each equation in~\eqref{k16lapexcludetempeq} to yield:
\begin{equation*}
\laps \alpha_1 = (\lambda-1) (\lambda+4) \alpha_1, \qquad \qquad \laps \beta_1 = (\lambda-1) (\lambda+4) \beta_1,
\end{equation*}
from which it follows that $\alpha_1 = 0$ and $\beta_1 = 0$.

The proof for $k=2$ is analogous, but we present it for completeness. We get the following system of equations:
\begin{equation} \label{k25lapexcludetempeq}
\laps \alpha_2 = \lambda (\lambda+5) \alpha_2 + 2 \dss \beta_2, \qquad \qquad \laps \beta_2 = (\lambda+2) (\lambda+3) \beta_2 + 2 \ds \alpha_2.
\end{equation}
Suppose that $\lambda \in (-3,-2)$. We get $\laps (\ds \beta_2) = (\lambda + 2) (\lambda + 3) \ds \beta_2$, so $\ds \beta_2 = 0$. Similarly taking $\dss$ of the first equation gives $\dss \alpha_2 = 0$. We also have
\begin{equation*}
\laps (\ds \alpha_2) = \lambda(\lambda+5) \ds \alpha_2 + 2 \laps \beta_2 = (\lambda^2 + 5 \lambda + 4) \ds \alpha_2 + (2 \lambda^2 + 10\lambda + 12) \beta_2.
\end{equation*}
This time one can check that
\begin{equation*}
\laps ( \ds \alpha_2 + (\lambda + 3) \beta_2  ) = (\lambda + 2)(\lambda + 5) ( \ds \alpha_2 + (\lambda + 3) \beta_2),
\end{equation*}
and hence $\ds \alpha_2 + (\lambda + 3) \beta_2 = 0$ for $\lambda \in (-3, -2)$. This relation can now be substituted into each equation in~\eqref{k25lapexcludetempeq} to yield:
\begin{equation*}
\laps \alpha_2 = \lambda (\lambda+3) \alpha_2, \qquad \qquad \laps \beta_2 = \lambda (\lambda+3) \beta_2,
\end{equation*}
from which it follows that $\alpha_2 = 0$ and $\beta_2 = 0$.

The reason we have gone through all three cases carefully is the following. The reader may now be tempted to try the same trick to get an excluded range of orders for homogeneous harmonic $3$-forms. It does not take long to verify that this procedure breaks down for $k=3$. See Corollary~\ref{homodiracexcludecor}, however, for what we can say about homogeneous $3$-forms.
\end{proof}

To study the operator $\dc + \dsc$ on a cone, we begin with the following lemma.
\begin{lemma} \label{homofirstlemma}
Let $\omega = \sum_{k=0}^7 \omega_k$, where $\omega_k = r^{\lambda + k - 1} dr \wedge \alpha_k
+ r^{\lambda + k} \beta_k$ is a homogeneous $k$-form of order $\lambda$. Then if $\omega$ is in the kernel of $\dc + \dsc$, the following equations hold:
\begin{align} \label{diracclosedeq1}
(\lambda + k - 1) \beta_{k - 1} & \, = \, \ds \alpha_{k-1} + \dss \alpha_{k+1}, \\ \label{diracclosedeq2} (\lambda - k + 6) \alpha_{k+1} & \, = \, \ds \beta_{k-1} + \dss \beta_{k+1},
\end{align}
for all $k= 0, \ldots, 7$.
\end{lemma}
\begin{proof}
By~\eqref{homodeq} and~\eqref{homodleq}, we can compute that
\begin{align*}
(\dc + \dsc) \omega \, = \, & \sum_{k=0}^7 \left( r^{\lambda + k - 1} dr \wedge ( (\lambda + k) \beta_k - \ds \alpha_k) + r^{\lambda + k} \ds \beta_k  \right) \\ \, & \, {} + \sum_{k=0}^7 \left( r^{\lambda + k - 3} dr \wedge ( -\dss \alpha_k) + r^{\lambda + k - 2} ( -(\lambda - k + 7) \alpha_k + \dss \beta_k)) \right) \\ 
\, = \, & \sum_{l=1}^8 \left( r^{\lambda + l - 2} dr \wedge ( (\lambda + l - 1) \beta_{l-1} - \ds \alpha_{l-1}) + r^{\lambda + l - 1} \ds \beta_{l-1}  \right) \\ \, & \, {} + \sum_{l=-1}^6 \left( r^{\lambda + l - 2} dr \wedge ( -\dss \alpha_{l+1}) + r^{\lambda + l - 1} ( -(\lambda - l + 6) \alpha_{l+1} + \dss \beta_{l+1})) \right).
\end{align*}
Using the fact that $\alpha_0 = 0$ and $\beta_7 = 0$, and $\ds \alpha_7 = 0$ and $\dss \beta_0 = 0$, both these sums can be taken from $l = 0$ to $7$. Relabelling the $l$ to a $k$ again, and combining terms, we find that
\begin{align*} 
(\dc + \dsc) \omega \, = \, & \sum_{k=0}^7 \left( r^{\lambda + k - 2} dr \wedge ( (\lambda + k - 1) \beta_{k - 1} - \ds \alpha_{k-1} - \dss \alpha_{k+1}) \right) \\ \, & \, {} + \sum_{k=0}^7 \left( r^{\lambda + k - 1} ( -(\lambda - k + 6) \alpha_{k+1} + \dss \beta_{k+1} + \ds \beta_{k-1}) \right)
\end{align*}
where it is understood that $\alpha_k$ and $\beta_k$ vanish when $k < 0$ or $k > 7$.
Setting this expression equal to zero yields equations~\eqref{diracclosedeq1} and~\eqref{diracclosedeq2}.
\end{proof}

\begin{cor} \label{homodiracexcludecor}
Suppose that $\omega$ is a homogeneous $k$-form of order $\lambda$ which is closed and coclosed: $\dc \omega = 0$, $\dsc \omega = 0$. Then we have:
\begin{align} \label{k07diracexcludeeq}
& \text{For } k = 0, 7, \qquad \qquad \omega = 0 \, \, \text{if } \, \lambda \in (-7,0), & \\ \label{k16diracexcludeeq} & \text{For } k = 1, 6, \qquad \qquad \omega = 0 \, \, \text{if } \, \lambda \in (-6,-1), & \\ \label{k25diracexcludeeq} & \text{For } k = 2, 5, \qquad \qquad \omega = 0 \, \, \text{if } \, \lambda \in (-5,-2), & \\ \label{k34diracexcludeeq} & \text{For } k = 3, 4, \qquad \qquad \omega = 0 \, \, \text{if } \, \lambda \in (-4,-3). &
\end{align}
This gives an excluded range of orders of homogeneity for closed and coclosed forms on the cone $C$.
\end{cor}
\begin{proof}
Let $\alpha_l=0$ and $\beta_l = 0$ for all $l \neq k$ in equations~\eqref{diracclosedeq1} and~\eqref{diracclosedeq2}. Thus we have that $\dc\omega = 0$ and $\dsc\omega = 0$ together imply the following equations:
\begin{equation} \label{homolaptempeq}
\ds \alpha_k = (\lambda + k) \beta_k, \qquad \dss \beta_k = (\lambda - k + 7) \alpha_k, \qquad \dss \alpha_k = 0, \qquad \ds \beta_k = 0.
\end{equation}
Since $\omega$ is closed and coclosed, it is also harmonic: $\lapc \omega = 0$. Therefore, equation~\eqref{homolapeq} shows that
\begin{align*}
\laps \alpha_k  & \, = \, (\lambda + k - 2)(\lambda - k + 7) \alpha_k + 2 \dss \beta_k, \\ \laps \beta_k & \, = \, (\lambda + k)(\lambda - k + 5) \beta_k + 2 \ds \alpha_k.
\end{align*}
Now substituting in the relations from~\eqref{homolaptempeq} and simplifying, we obtain
\begin{equation*}
\laps \alpha_k = (\lambda + k) (\lambda - k + 7) \alpha_k, \qquad \qquad \laps \beta_k = (\lambda + k) (\lambda - k + 7) \beta_k.
\end{equation*}
Since the Laplacian has non-negative eigenvalues, we see that both $\alpha_k$ and $\beta_k$ must vanish (and hence $\omega = 0$) if $(\lambda + k)(\lambda - k + 7) < 0$, which occurs exactly when $\lambda$ is between $-k$ and $k-7$.
\end{proof}

\begin{rmk} \label{diracexcludermk}
Corollary~\ref{homodiracexcludecor} should be compared to Proposition~\ref{homolapexcludeprop}. Closed and coclosed forms are harmonic, but not always conversely. The above result says that when the form is known to be closed and coclosed, we can get a bigger range of excluded orders of homogeneity. Also, Proposition~\ref{homolapexcludeprop} tells us nothing about $3$-forms.
\end{rmk}

\begin{cor} \label{homocccor}
Suppose that $\omega$ is a homogeneous $k$-form of order $-k$ which is closed and coclosed: $\dc \omega = 0$, $\dsc \omega = 0$. Then $\omega = \beta$, where $\beta$ is a harmonic $k$-form on the link $\Sigma$: $\laps \beta = 0$.
\end{cor}
\begin{proof}
We substitute $\lambda = - k$ in the relations~\eqref{homolaptempeq} to get
\begin{equation*}
\ds \alpha_k = 0, \qquad \dss \beta_k = (7- 2k) \alpha_k, \qquad \dss \alpha_k = 0, \qquad \ds \beta_k = 0.
\end{equation*}
Since $\Sigma$ is compact and oriented, we can use Hodge theory. The first and third equations above say that $\alpha_k$ is harmonic, but since $(7-2k)$ is never zero, the second equations says that $\alpha_k$ is also coexact. Thus $\alpha_k = 0$, and then the second and fourth equations say that $\beta_k$ is harmonic on $\Sigma$.
\end{proof}

The next three propositions are needed in Section~\ref{obssubsection} for the solution of the obstruction problem, and also in Section~\ref{ACgaugefixsec} for the asymptotic expansion of the $\G$~structure on an asymptotically conical $\G$~manifold.
\begin{prop} \label{dirachomokernelprop}
Let $\omega = \sum_{k=0}^3 \omega_{2k}$ be a mixed even-degree form on the cone $C' = (0, \infty) \times \Sigma$, which is homogeneous of order $\lambda = -3$. That is, each $\omega_{2k} = r^{2k-4} dr \wedge \alpha_{2k} + r^{2k-3} \beta_{2k}$, where $\alpha_{2k}$ and $\beta_{2k}$ are $(2k-1)$-forms and $2k$-forms on $\Sigma$, respectively, independent of $r$. If $(d + \dsc)(\omega) = 0$, then $\omega = dr \wedge \alpha_4$, where $\alpha_4$ is a harmonic $3$-form on $\Sigma$.
\end{prop}
\begin{proof}
Substituting $\lambda = -3$ into Lemma~\ref{homofirstlemma}, we obtain
\begin{equation}  \label{dirachomokerneltempeq}
(k-4) \beta_{k-1} = \ds \alpha_{k-1} + \dss \alpha_{k+1}, \qquad \qquad (3-k) \alpha_{k+1} = \ds \beta_{k-1} + \dss \beta_{k+1}.
\end{equation}
If we take $\ds$ and $\dss$ of these equations and relabel indices, we get
\begin{align*}
& \ds \dss \alpha_k = (k-5) \ds \beta_{k-2}, & \dss \ds \alpha_k = (k-3) \dss \beta_k, \\ & \ds \dss \beta_k = (4-k) \ds \alpha_k, & \dss \ds \beta_k = (2-k) \dss \alpha_{k+2}.
\end{align*}
These can now be combined to yield
\begin{equation*}
\laps \alpha_k = (k-3) \dss \beta_k + (k-5)\ds \beta_{k-2}, \qquad \qquad \laps \beta_k = (2-k)\dss \alpha_{k+2} + (4-k) \ds \alpha_k.
\end{equation*}
We can now use~\eqref{dirachomokerneltempeq} again to eliminate $\dss \beta_k$ and $\dss \alpha_{k+2}$ in the above expressions to finally obtain
\begin{equation} \label{dirachomokerneltempeq2}
\laps \alpha_k = -(k-3)(k-4) \alpha_k - 2 \ds \beta_{k-2}, \qquad \qquad 
\laps \beta_k = -(k-3)(k-2) \beta_k + 2 \ds \alpha_k.
\end{equation}
Now $\alpha_0 = 0$, so the second equation in~\eqref{dirachomokerneltempeq2} gives $\laps \beta_0 = - 6 \beta_0$, so $\beta_0 = 0$. Then the first equation in~\eqref{dirachomokerneltempeq2} gives $\laps \alpha_2 = -2 \alpha_2$, so $\alpha_2 = 0$. We continue in this way to alternate between the equations. Next we get $\laps \beta_2 = 0$, but from~\eqref{dirachomokerneltempeq} we see that $\beta_2$ is also coexact, so $\beta_2 = 0$. Then $\laps \alpha_4 = 0$, but this time~\eqref{dirachomokerneltempeq} gives no information about $\alpha_4$ so all we can say is that it is a harmonic $3$-form on $\Sigma$, hence closed and coclosed. Then $\laps \beta_4 = -2 \beta_4$, so $\beta_4 = 0$. Continuing we get $\laps \alpha_6 = -6 \alpha_6$, so $\alpha_6 = 0$, and finally $\laps \beta_6 = -12 \beta_6$, so $\beta_6 = 0$.
\end{proof}

\begin{prop} \label{dirachomokernellogsprop}
Let $\omega = \sum_{l=0}^m (\log r)^l \sum_{k=0}^3 \omega_{2k,l}$ be an even-degree mixed form on the cone $C' = (0, \infty) \times \Sigma$, satisfying $(d + \dsc)(\omega) = 0$, where each $\omega_{2k,l}$ is homogeneous of order $-3$, and $m \geq 0$. Then in fact necessarily $m = 0$ and $\omega$ is as given in Proposition~\ref{dirachomokernelprop}.
\end{prop}
\begin{proof}
We prove this by contradiction. Suppose that $m > 0$. Hence $\sum_{k=0}^3 \omega_{2k, m} \neq 0$. Each $\omega_{2k, l}$ is homogenous of order $-3$, so it can be written as
\begin{equation} \label{dirachomokernellogstempeq1}
\omega_{2k, l} \, = \, r^{2k - 4} dr \wedge \alpha_{2k, l} + r^{2k - 3} \beta_{2k,l}
\end{equation}
where for each $l$, $\alpha_{2k, l}$ and $\beta_{2k, l}$ are $(2k-1)$-forms and $2k$-forms on $\Sigma$, respectively, independent of $r$. It is easy to check that if $\omega_k$ is any $k$-form on the cone $C$, then
\begin{equation*}
(d + \dsc) ( (\log r)^l \omega_k) \, = \, (\log r)^l (d + \dsc) (\omega_k) + \frac{l}{r} (\log r)^{l-1} (dr \wedge \omega_k) - \frac{l}{r} (\log r)^{l-1} \left( \ddr \hk \omega_k \right).
\end{equation*}
Using this identity, one can now compute that
\begin{multline*}
(d + \dsc) \left( \sum_{l=0}^m (\log r)^l \sum_{k=0}^3 \omega_{2k,l} \right) \, = \, (\log r)^m (d + \dsc) \left( \sum_{k=0}^3 \omega_{2k, m} \right) \\ \, {} + \sum_{l=0}^{m-1} (\log r)^l \sum_{k=0}^3 \left( d (\omega_{2k, l}) + \dsc (\omega_{2k, l}) + \frac{(l+1)}{r} dr \wedge \omega_{2k, l+1} -  \frac{(l+1)}{r} \ddr \hk \omega_{2k, l+1} \right).
\end{multline*}
The above expression must vanish as a polynomial in $\log r$. Setting the coefficient of $(\log r)^m$ equal to zero, we see that $\sum_{k=0}^3 \omega_{2k, m}$ is in the kernel of $d + \dsc$ and homogeneous of order $-3$, so by Proposition~\ref{dirachomokernelprop} we have
\begin{equation} \label{dirachomokernellogstempeq2}
\sum_{k=0}^3 \omega_{2k, m} \, = \, dr \wedge \alpha_{4,m}
\end{equation}
where $\alpha_{4,m}$ is a non-zero harmonic $3$-form on the link $\Sigma$. Now consider the coefficient of the next highest term, $(\log r)^{m-1}$. Setting it equal to zero and using~\eqref{dirachomokernellogstempeq2} gives
\begin{equation*}
\sum_{k=0}^3 \left( d (\omega_{2k, m-1}) + \dsc (\omega_{2k, m-1}) \right) - \frac{m}{r} \alpha_{4,m} = 0.
\end{equation*}
Using~\eqref{homodeq} and~\eqref{homodleq} with $\lambda = - 3$, this becomes
\begin{multline*}
\sum_{k=0}^3 \left( r^{2k-4} dr \wedge ( (2k-3) \beta_{2k, m-1} - \ds \alpha_{2k, m-1} ) + r^{2k-3} \ds \beta_{2k, m-1} \right) \\ {} + \sum_{k=0}^3 \left( r^{2k-6} dr \wedge ( - \dss \alpha_{2k, m-1} ) + r^{2k-5} ( (2k-4) \alpha_{2k, m-1} + \dss \beta_{2k, m-1}) \right) - \frac{m}{r} \alpha_{4,m} = 0.
\end{multline*}
Taking the $3$-form component of the above equation and simplifying gives
\begin{equation*}
r^{-2} dr \wedge ( - \beta_{2,m-1} - \ds \alpha_{2,m-1} - \dss \alpha_{4,m-1} ) + r^{-1} ( \ds \beta_{1,m-1} + \dss \beta_{4,m-1} - m \alpha_{4,m} ) = 0,
\end{equation*}
which in turn says that
\begin{equation*}
\alpha_{4,m} = \ds \left( \frac{1}{m} \beta_{1,m-1} \right) + \dss \left( \frac{1}{m} \beta_{4,m-1} \right).
\end{equation*}
Since $\Sigma$ is compact, Hodge theory says that harmonic forms are orthogonal to the image of $\ds$ and to the image of $\dss$, so we must have $\alpha_{4,m} = 0$, giving us our contradiction. Therefore $m = 0$.
\end{proof}

\begin{prop} \label{dirachomokernellogsprop2}
Let $\omega = \sum_{l=0}^m (\log r)^l \omega_{4,l}$ be a pure $4$-form on the cone $C' = (0, \infty) \times \Sigma$, satisfying $(d + \dsc)(\omega) = 0$, where each $\omega_{4,l}$ is homogeneous of order $-4$, and $m \geq 0$. Then in fact necessarily $m = 0$ and $\omega$ is as given in Corollary~\ref{homocccor}.
\end{prop}
\begin{proof}
The proof of this proposition is exactly analogous to the proof of Proposition~\ref{dirachomokernellogsprop} and is omitted. However, we note here that it is essential that $\omega$ be a pure $4$-form for this result to be true. In a general even-degree mixed form in the kernel of $d + \dsc$ which is a polynomial in $\log(r)$ with coefficients being homogeneous forms of order $-4$, there \emph{can occur} $\log(r)$ terms.
\end{proof}

The next result is about the type decomposition of $3$-forms on a $\G$~cone, which are homogeneous of order $-3$.
\begin{prop} \label{cones27prop}
Let $(C, \phc)$ be a $\G$~cone. Let $\xi$ be a homogeneous $3$-form on $C$ of order $-3$ which is harmonic. Then $\xi$ is in $\Lambda^3_{27}$ with respect to $\phc$.
\end{prop}
\begin{proof}
Let $\xi = \xi_1 + \xi_7 + \xi_{27}$ be the decomposition of $\xi$ into components. Now $\xi_1 = f \phc$ for some function $f$ on $C$, and since, up to a constant, $f = \stc ( \xi \wedge \psc)$, we see that $f$ is homogenous of order $-3$. Similarly, we have $\xi_7 = \stc (\omega \wedge \phc)$ for some $1$-form $\omega$ on $C$, and, up to a constant, $\omega = \stc ( \xi \wedge \phc)$, so $\omega$ is also homogenous of order $-3$.  But $\xi$ is harmonic, and the Laplacian commutes with the projections. Therefore $\lapc (\xi_1) = 0$ and $\lapc(\xi_7) = 0$. By~\eqref{temp27eq1} and~\eqref{temp27eq2} we then see that $\lapc (f) = 0$ and $\lapc (\omega) = 0$. Finally, Proposition~\ref{homolapexcludeprop} says that both $f$ and $\omega$ are zero, since the order $\lambda = -3$ lies in the excluded range for both functions and $1$-forms. Thus $\xi = \xi_{27}$ as claimed.
\end{proof}

We close this section with an observation about representing cohomology classes
of $(a,b) \times \Sigma$, where $(a,b)$ is any open subinterval of $(0, \infty)$.
\begin{prop} \label{formsrepresentprop}
Suppose $B$ is a cohomology class in $H^k ( (a,b) \times \Sigma, \R)$. Then there exists a $k$-form $\beta$ on $(a,b) \times \Sigma$, harmonic with respect to the cone metric, and homogeneous of order $-k$ such that $[\beta] = B$.
\end{prop}
\begin{proof}
The projection $\pi : (a,b) \times \Sigma \to \Sigma$ induces an isomorphism $\pi^* : H^k( \Sigma, \R) \to H^k ( (a,b) \times \Sigma, \R)$ on cohomology by pullback. Then $(\pi^*)^{-1}(B)$ is a class in $H^k(\Sigma, \R)$. Since $\Sigma$ is compact and oriented, by Hodge theory there exists a unique harmonic $k$-form $\beta$ on $\Sigma$ such that $[\beta] = (\pi^*)^{-1}(B)$. Then $\pi^* (\beta) = \beta$, as a form on $(a, b) \times \Sigma$, and represents the class $B$. Now from equation~\eqref{conelaplacianeq}, with $\alpha = 0$ and $\beta' = 0$, we see that $\lapc \beta = \frac{1}{r^2} \laps \beta = 0$, so $\beta$ is a harmonic $k$-form on $(a,b) \times \Sigma$, homogeneous of order $-k$.
\end{proof}

\subsection{Compact \G~manifolds with isolated conical singularities} \label{CSsec}

Let $M$ be a compact, connected topological space, and let $x_1, \ldots, x_n$ be a finite set of isolated points in $M$. We assume that $M' = M \backslash \{x_1, \ldots, x_n \}$ is a smooth non-compact $7$-dimensional manifold which we will call the \emph{smooth part} of $M$ and $\{ x_1, \ldots, x_n\}$ will be called the \emph{singular points} of $M$.

\begin{defn} \label{CSdefn}
The space $M$ is called a \emph{$\G$~manifold with isolated conical singularities}, with cones $C_1, \ldots, C_n$ at $x_1, \ldots, x_n$ and \emph{rates} $\mu_1, \ldots, \mu_n$, where each $\mu_i > 0$, if all of the following holds:
\begin{itemize}
\item The smooth part $M'$ is a $\G$~manifold with torsion-free $\G$~structure $\phm$ and metric $\gm$.
\item There are $\G$~cones $(C_i, \phci, \gci)$ with links $\Sigma_i$ for all $i = 1, \ldots , n$.
\item There is a compact subset $K \subset M'$ such that $M' \backslash K$ is a union of open sets $S_1, \ldots, S_n$ whose closures $\bar S_1, \ldots, \bar S_n$ in $M$ are all \emph{disjoint} in $M$.
\item There is an $\e \in (0,1)$, and for each $i = 1, \ldots, n$, there is a smooth function $f_i : (0, \e) \times \Sigma_i \to M'$ that is a \emph{diffeomorphism} of $(0, \e) \times \Sigma_i$ onto $S_i$.
\item The pullback $f_i^* (\phm)$ is a torsion-free $\G$~structure on the subset $(0, \e) \times \Sigma_i$ of $C_i$. We require that this approach the torsion-free $\G$~structure $\phci$ in a $C^{\infty}$ sense, with rate $\mu_i > 0$. This means that
\begin{equation} \label{CSdefneq}
| \nabci^j \! ( f_i^* (\phm) - \phci ) |_{\gci} \, = \, O (r^{\mu_i - j})
\qquad \forall j \geq 0
\end{equation}
in $(0,\e) \times \Sigma_i$. Note that all norms and derivatives are computed using the cone metric $\gci$.
\end{itemize}
\end{defn}
The third condition ensures that the singular points are isolated in $M$. It is easy to see that the holonomy necessarily has to be exactly $\G$, because the holonomy of the asymptotic cones is exactly $\G$, and the holonomy of $M$ must be at least as big as the holonomy of its asymptotic cones, but it is contained in $\G$ since $\phm$ is a torsion-free $\G$~structure.

Since the metric $\gm$ and the $4$-form $\psm$ are pointwise smooth functions of $\phm$, by Taylor's theorem we also have
\begin{equation} \label{CSdefneq2}
\begin{aligned}
| \nabci^j \! ( f_i^* (\gm) - \gci ) |_{\gci} & \, = \, O (r^{\mu_i - j})
\qquad \forall j \geq 0, \\ | \nabci^j \! ( f_i^* (\psm) - \psci ) |_{\gci} & \, =
\, O (r^{\mu_i - j}) \qquad \forall j \geq 0.
\end{aligned}
\end{equation}
A $\G$~manifold with isolated conical singularities will sometimes be called a $\G$~manifold with ICS for brevity. Figure~\ref{CS1fig} shows a compact manifold with isolated conical singularities.

\begin{figure} [ht]
\centering
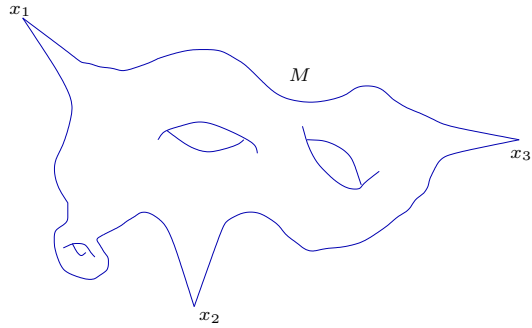
\caption{A compact manifold $M$ with isolated conical singularities}
\label{CS1fig}
\end{figure}

\begin{rmk} \label{CSregularityrmk}
It can be shown that if we assume that~\eqref{CSdefneq} holds only for $j=0$ and $j=1$, then there exists \emph{particular} diffeomorphisms $f_i$ that satisfy~\eqref{CSdefneq} for all $j \geq 0$. These special diffeomorphisms satisfy a \emph{gauge-fixing condition} which forces them to solve (in some sense) an elliptic equation, and this results in their improved regularity. This will be discussed in the sequel~\cite{KCSmoduli} to this paper. For our present purposes, it suffices to assume from the outset that~\eqref{CSdefneq} holds for all $j \geq 0$. See also Remark~\ref{gaugefixrmk}.
\end{rmk}

Next we discuss an important fact about $\G$~manifolds with ICS.
\begin{prop} \label{CSclassprop}
The closed $3$-forms $f_i^*(\phm)$ and $\phci$ on $(0, \e) \times \Sigma_i$ represent the same cohomology class in $H^3( (0, \e) \times \Sigma_i, \R) \cong H^3(\Sigma_i, \R)$. Similarly the closed $4$-forms $f_i^*(\psm)$ and $\psci$ on $(0, \e) \times \Sigma_i$ represent the same cohomology class in $H^4( (0, \e) \times \Sigma_i, \R) \cong H^4(\Sigma_i, \R)$. Therefore by Proposition~\ref{coneformsexactprop}, the cohomology classes $[f_i^*(\phm)]$ in $H^3(\Sigma_i, \R)$ and $[f_i^*(\psm)]$ in $H^4(\Sigma_i, \R)$ are always \emph{trivial}.
\end{prop}
\begin{proof}
Using $| ( f_i^* (\phm) - \phci ) |_{\gci} = O (r^{\mu_i})$ and $| ( f_i^* (\psm) - \psci ) |_{\gci} = O (r^{\mu_i})$, and the fact that $\mu_i > 0$, this follows immediately from Lemma~\ref{exactformslemma}.
\end{proof}
\begin{rmk} \label{CSclassrmk}
The analogous statement will be \emph{false} for asymptotically conical $\G$~manifolds, as discussed in Proposition~\ref{ACclassprop} and before Definition~\ref{ACinvariantsdefn}.
\end{rmk}

On several occasions we will need to compare $f_i^* (\stm \omega)$ with $\stci f_i^*(\omega)$ near $x_i$.
\begin{lemma} \label{starcomparelemma}
Let $\omega$ be a smooth $k$-form on $M'$, with $|f_i^*( \omega)|_{\gci} = O(r^{\lambda})$ near the $i^{\text{th}}$ singular point $x_i$. Then the $(7-k)$-forms $f_i^*(\stm \omega)$ and $\stci f_i^*(\omega)$ satisy
\begin{equation} \label{starcompareeq}
| f_i^*(\stm \omega) - \stci f_i^*(\omega) |_{\gci} = \, O(r^{\lambda + \mu_i}).
\end{equation}
\end{lemma}
\begin{proof}
We begin by computing that
\begin{equation*}
|f_i^*(\stm \omega) - \stci f_i^*(\omega) |_{\gci} = |\st_{f_i^*(\gm)} (f_i^*(\omega)) - \stci (f_i^* (\omega)) |_{\gci} =  |(\st_{f_i^*(\gm)} - \stci) f_i^*(\omega)|_{\gci}.
\end{equation*}
But by~\eqref{CSdefneq2}, we have $|f_i^*(\gm) - \gci|_{\gci} = O(r^{\mu_i})$, and thus
\begin{equation*}
|(\st_{f_i^*(\gm)} - \stci) f_i^*(\omega)|_{\gci} \leq C r^{\mu_i} |f_i^*(\omega)|_{\gci} \leq C r^{\lambda + \mu_i},
\end{equation*}
and the proof is complete.
\end{proof}

Finally, we close this section with a definition involving the algebraic topology of the non-compact smooth part $M'$ of a compact $\G$~manifold $M$ with isolated conical singularities, which will be crucial in understanding and solving the obstruction problem in Sections~\ref{obssubsection} and~\ref{obssubsection2}.

Fix some $r_0 \in (0, \e)$. The map $\iota_i : \Sigma_i \to (0, \e) \times \Sigma_i$ given by $\iota_i(\sigma) = (r_0, \sigma)$ embeds the link $\Sigma_i$ as a submanifold of $(0, \e) \times \Sigma_i$, and the pullback map $\iota_i^* : H^k ( (0, \e) \times \Sigma_i, \R) \to H^k (\Sigma_i, \R)$ is an isomorphism, which is just restriction on the level of forms. We also have the smooth maps $f_i :(0, \e) \times \Sigma_i \to M'$, and thus for each $i$ we get a map $(f_i \circ \iota_i)^* : H^k (M', \R) \to H^k (\Sigma_i, \R)$.
\begin{defn} \label{topdefn}
For each $k =0, \ldots, 7$, we define a map $\Upsilon^k : H^k (M', \R) \to \oplus_{i=1}^n H^k (\Sigma_i, \R)$ by
\begin{equation} \label{topeq}
\Upsilon^k( [\omega] ) \, = \, \oplus_{i=1}^n \, (f_i \circ \iota_i)^* ([\omega])
\end{equation}
for any cohomology class $[\omega] \in H^k(M', \R)$. Essentially, at each end we restrict the $k$-form $\omega$ to the link $\Sigma_i$ and take its cohomology in $H^k (\Sigma_i, \R)$.
\end{defn}
It is possible to fit the maps $\Upsilon^k$ into a long exact sequence in cohomology that is obtained by considering $M'$ as the interior of a compact manifold $\bar M'$ with boundary $\oplus_{i=1}^n \Sigma_i$ and taking the relative cohomology long exact sequence. However, for our present purposes we will not really need anything more than the above definition of $\Upsilon^k$.

\begin{rmk} \label{noICSrmk}
There are at present no known examples of compact $\G$~manifolds with isolated conical singularities, although they are expected to exist in abundance. Our main theorem in this paper can be interpreted as further evidence for the likelihood of their existence, and that they should arise as `boundary points' in the moduli space of smooth compact $\G$~manifolds. The author is currently working in collaboration with Dominic Joyce on a new construction of compact $\G$~manifolds which should also be able to produce the first examples of compact $\G$~manifolds with ICS.
\end{rmk}

\subsection{Asymptotically conical $\G$~manifolds} \label{ACsec}

In this section we define an \emph{asymptotically conical $\G$~manifold}, and discuss three explicit examples. Let $N$ be a non-compact, connected smooth $7$-dimensional manifold.
\begin{defn} \label{ACdefn}
The manifold $N$ is called an \emph{asymptotically conical} $\G$~manifold with cone $C$ and \emph{rate} $\nu < 0$ if all of the following holds:
\begin{itemize}
\item The manifold $N$ is a $\G$~manifold with torsion-free $\G$~structure $\phn$ and metric $\gn$.
\item There is a $\G$~cone $(C, \phc, \gc)$ with link $\Sigma$.
\item There is a compact subset $L \subset N$.
\item There is an $R > 1$, and a smooth function $h : (R, \infty) \times \Sigma \to N$ that is a \emph{diffeomorphism} of $(R, \infty) \times \Sigma$ onto $N \backslash L$.
\item The pullback $h^* (\phn)$ is a torsion-free $\G$~structure on the subset $(R, \infty) \times \Sigma$ of $C$. We require that this approach the torsion-free $\G$~structure $\phc$ in a $C^{\infty}$ sense, with
rate $\nu < 0$. This means that
\begin{equation} \label{ACdefneq}
| \nabc^j ( h^* (\phn) - \phc ) |_{\gc} \, = \, O (r^{\nu - j})
\qquad \forall j \geq 0
\end{equation}
in $(R,\infty) \times \Sigma$. Note that all norms and derivatives are computed using the cone metric $\gc$.
\end{itemize}
\end{defn}
This should be compared to Definition~\ref{CSdefn}. An asymptotically conical $\G$~manifold also has holonomy exactly equal to $\G$, by the same argument as in the ICS case. Also, an asymptotically conical $\G$~manifold always has only one asymptotic end. This follows from the \emph{Cheeger--Gromoll splitting theorem}, which says that a complete non-compact Ricci-flat manifold with more than one end isometrically splits into a Riemannian product, and hence if we had more than one end, the holonomy would be reducible. Hence the link $\Sigma$ of the asymptotic cone of $N$ must be connected. This is why we defined the link of a $\G$~cone to be connected in Section~\ref{g2conessec}. See Besse~\cite{Besse} for more details.

Since the metric $\gn$ and the $4$-form $\psn$ are pointwise smooth functions of $\phn$, by Taylor's theorem we also have
\begin{equation}  \label{ACdefneq2}
\begin{aligned}
| \nabc^j ( h^* (\gn) - \gc ) |_{\gc} &\, = \, O (r^{\nu - j}) \qquad \forall j \geq 0, \\ | \nabc^j ( h^* (\psn) - \psc ) |_{\gc} & \, = \, O (r^{\nu - j}) \qquad \forall j \geq 0.
\end{aligned}
\end{equation}
An asymptotically conical $\G$~manifold will sometimes be called an AC $\G$~manifold for brevity. Figure~\ref{AC1fig} shows an asymptotically conical manifold.

\begin{figure} [ht]
\centering
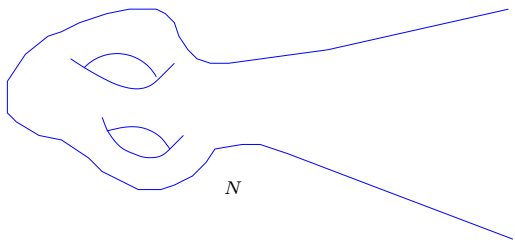
\caption{An asymptotically conical manifold $N$}
\label{AC1fig}
\end{figure}

\begin{rmk} \label{ACregularityrmk}
As in Remark~\ref{CSregularityrmk}, it can be shown that if we assume that~\eqref{ACdefneq} holds only for $j=0$ and $j=1$, then there exists a \emph{particular} diffeomorphism $h$ that satisfies~\eqref{ACdefneq} for all $j \geq 0$.
\end{rmk}

There are three known examples of asymptotically conical $\G$~manifolds, whose asymptotic cones have links given by the strictly nearly K\"ahler manifolds of Examples~\ref{cp3ex},~\ref{flagex}, and~\ref{s3s3ex}, respectively. They are all total spaces of vector bundles over a compact base. They were discovered by Bryant--Salamon~\cite{BS} and were the first examples of complete $\G$~manifolds.

\begin{ex} \label{asds4ex}
{\bf $\Lambda^2_-(S^4)$}, the bundle of anti-self-dual $2$-forms over the $4$-sphere. This is a non-trivial rank $3$ vector bundle over the standard round $S^4$. This AC $\G$~manifold is asymptotic to the cone over the non-K\"ahler $\C \PR^3$ of Example~\ref{cp3ex}, with rate $\nu = -4$.
\end{ex}

\begin{ex} \label{asdcp2ex}
{\bf $\Lambda^2_-(\C \PR^2)$}, the bundle of anti-self dual $2$-forms over the complex projective plane.
This is a non-trivial rank $3$ vector bundle over the standard Fubini-Study $\C \PR^2$. This AC $\G$~manifold is asymptotic to the cone over the complex flag manifold $F_{1,2}$ of Example~\ref{flagex}, also with rate $\nu = -4$.
\end{ex}

\begin{ex} \label{spins3ex}
{\bf $\spi (S^3)$}, the spinor bundle of the $3$-sphere. This is a trivial rank $4$ vector bundle over the standard round $S^3$, hence is topologically $S^3 \times \R^4$. This AC $\G$~manifold is asymptotic to the cone over the nearly K\"ahler $S^3 \times S^3$ of Example~\ref{s3s3ex}, with rate $\nu = -3$.
\end{ex}

\begin{rmk} \label{ACexamplesrmk}
Explicit formulas for the asymptotically conical $\G$~structures of Examples~\ref{asds4ex},~\ref{asdcp2ex}, and~\ref{spins3ex}, as well as the fact that their rates are $-4$, $-4$, and $-3$, respectively,  can be found in Bryant--Salamon~\cite{BS}, and also in Atiyah--Witten~\cite{AW}. We will not have need for these explicit formulas.
\end{rmk}

Next we discuss the AC analogue of Proposition~\ref{CSclassprop}, which is \emph{different} in a very important way which will lead to topological obstructions to our desingularization procedure.
\begin{prop} \label{ACclassprop}
The two closed $3$-forms $h^*(\phn)$ and $\phc$ on $(R, \infty) \times \Sigma$ represent the same cohomology class in $H^3( (R, \infty) \times \Sigma, \R) \cong H^3(\Sigma, \R)$ \emph{if the rate $\nu$ satisfies $\nu < -3$}. This also holds if $H^3(\Sigma, \R) = 0$. Similarly the two closed $4$-forms $h^*(\psn)$ and $\psc$ on $(R, \infty) \times \Sigma$ represent the same cohomology class in $H^4( (R, \infty) \times \Sigma, \R) \cong H^4(\Sigma, \R)$ \emph{if the rate $\nu$ satisfies $\nu < -4$}. This also holds if $H^4(\Sigma, \R) = 0$.
\end{prop}
\begin{proof}
This follows immediately from $| ( h^* (\phm) - \phc ) |_{\gc} = O (r^{\nu})$ and $| ( h^* (\psm) - \psc ) |_{\gc} = O (r^{\nu})$, using Lemma~\ref{exactformslemma}. The second part of each statement is automatic.
\end{proof}

Therefore, in contrast to the case of $\G$~manifolds with ICS as discussed in Proposition~\ref{CSclassprop}, we \emph{cannot} conclude that the cohomology classes $[h^*(\phn)]$ and $[h^*(\psn)]$ are trivial in $H^3(\Sigma, \R)$ and $H^4 (\Sigma, \R)$, respectively, and indeed in general when the rate $\nu$ is not sufficiently negative, they will not be. This will introduce some \emph{obstructions} to our glueing procedure, which are discussed in Section~\ref{formsconstructionsec}. These observations lead us to make the following definition.

\begin{defn} \label{ACinvariantsdefn}
For an AC $\G$~manifold $(N, \phn, \psn, \gn)$, we define the cohomological invariants $\Phi(N) \in H^3(\Sigma, \R)$ and $\Psi(N) \in H^4(\Sigma, \R)$ to be the cohomology classes $[h^*(\phn)]$ and $[h^*(\psn)]$, respectively. By Proposition~\ref{ACclassprop}, $\Phi(N) = 0$ if the rate $\nu < -3$ and $\Psi(N) = 0$ if $\nu < -4$.
\end{defn}

\begin{rmk} \label{BSratesrmk}
For the Bryant--Salamon examples of asymptotically conical $\G$~manifolds described in Examples~\ref{asds4ex},~\ref{asdcp2ex}, and~\ref{spins3ex}, \emph{only one} of these invariants can be (and is) non-zero in each case, by Remark~\ref{NKcohomologyrmk} and Definition~\ref{ACinvariantsdefn}. For the spaces $\Lambda^2_-(S^4)$ and $\Lambda^2_-(\C \PR^2)$, we have $\Phi(N) = 0$ and $\Psi(N) \neq 0$, whereas for the space $\spi (S^3)$ we have $\Phi(N) \neq 0$ and $\Psi(N) = 0$.
\end{rmk}

\section{Desingularization of compact $\G$~manifolds with ICS}
\label{desingsec}

Suppose $M$ is a compact $\G$~manifold with isolated conical singularities $x_1, \ldots, x_n$, cones $C_1, \ldots, C_n$, and rates $\mu_1, \ldots, \mu_n >0$. Assume also that we have asymptotically conical $\G$~manifolds $N_1, \ldots, N_n$, with the same cones $C_1, \ldots, C_n$, and rates $\nu_1, \ldots, \nu_n \leq -3$. (We will see later why we need to assume that each $\nu_i \leq -3$ rather than just $\nu_i < 0$.) We want to \emph{desingularize} $M$ to obtain a \emph{smooth} compact $\G$~manifold. The idea is to cut out a neighbourhood of each singularity $x_i$, and \emph{glue} in $N_i$. In this way we obtain a smooth compact manifold $\tilde M$, and then we need to show using analysis that $\tilde M$ admits a torsion-free $\G$~structure.

The main tool that we will require is the following theorem of Joyce, which says that if one can find a closed $\G$~structure $\ph$ on a compact manifold $M$ whose torsion is sufficiently small, then there exists a \emph{torsion-free} $\G$~structure $\tilde \ph$ on $M$ which is \emph{close to} $\ph$ in some sense. It was used by Joyce in~\cite{J1,J2} to construct the first compact examples of manifolds with $\G$~holonomy. (Another glueing construction of smooth compact $\G$~manifolds is due to Kovalev~\cite{Kov}.) The precise statement of Joyce's theorem is as follows.

\begin{thm}[Joyce~\cite{J4} Theorem 11.6.1] \label{joycethm}
Let $\kappa$, $D_1$, $D_2$, and $D_3$ be any positive constants. Then there exists $s_0 \in (0, 1]$ and $D_4 > 0$, such that whenever $0 < s \leq s_0$, the following holds:

Let $M$ be a smooth compact $7$-manifold, with $\G$~structure $\ph$ and associated metric $g$ satisfying $d \ph = 0$. Suppose there is a smooth $3$-form  $\chi$ on $M$ satisfying $d^*_g \chi = d^*_g \ph$ such that
\begin{enumerate}[i)]
\item $\, \, {|| \chi ||}_{C^0} \leq D_1 s^{\kappa}, \qquad {|| \chi ||}_{L^2} \leq D_1 s^{\frac{7}{2} + \kappa}, \qquad {|| d^*_g \chi ||}_{L^{14}} \leq D_1 s^{-\frac{1}{2} + \kappa}$.
\item the \emph{injectivity radius} $\mathcal{I}(g)$ satisfies $\mathcal{I}(g) \geq D_2 s$.
\item the \emph{Riemann curvature} $\mathcal{R}(g)$ satisfies ${|| \mathcal{R}(g) ||}_{C^0} \leq D_3 s^{-2}$.
\end{enumerate}
Then there exists a smooth, \emph{torsion-free} $\G$~structure $\tilde \ph$ with metric $\tilde g$ on $M$ and such that
\begin{itemize}
\item $\, {|| \ph - \tilde \ph ||}_{C^0} \leq D_4 s^{\kappa}$.
\item $\, [\ph] = [\tilde \ph]$ in $H^3(M, \mathbb R)$.
\end{itemize}
Here all norms are computed using the original metric $g$.
\end{thm}

\begin{rmk} \label{improvedconstantsrmk}
The exponents in i) of Theorem~\ref{joycethm} are the best possible for the theorem to be true, and are improvements to those presented in Theorem 11.6.1 of Joyce~\cite{J4}. See the discussion at the bottom of page 296 of Joyce~\cite{J4} for more details.
\end{rmk}

In Section~\ref{manifoldconstructionsec} we construct a one parameter family $\tilde M_s$ of \emph{compact, smooth} $7$-manifolds for small $s>0$. Then in Section~\ref{formsconstructionsec} we construct a closed $\G$~structure $\ph_s$ on $\tilde M_s$, and finally in Section~\ref{torsionfreesec} we show that for $s$ sufficiently small, there exists a torsion-free $\G$~structure $\tilde{\ph}_s$ on $\tilde M_s$ close to $\ph_s$.

\subsection{Construction of the smooth compact manifolds $\tilde M_s$}
\label{manifoldconstructionsec}

By letting $R = \max (R_1, \ldots, R_n)$, we can assume that the parameter $R$ in Definition~\ref{ACdefn} is the same for all the AC $\G$~manifolds $N_1, \ldots, N_n$. For now we will consider only those $s > 0$ which are small enough so that $2sR < \e$. Later, in Section~\ref{formsconstructionsec}, we will need to further restrict the values of $s$.

We first apply a \emph{homothety} (scaling) to each asymptotically conical $\G$~manifold $(N_i, \phni, \psni, \gni)$ as follows. We have a diffeomorphism $h_i : (R, \infty) \to N_i \backslash L_i$, such that
$h_i^* (\phni) = \phci + O(r^{\nu_i})$ as $r \to \infty$. For a fixed constant $s$, we define
\begin{equation} \label{scaledACdefneq}
N_{i,s} = N_i, \qquad \phnis = s^3 \phni, \qquad \psnis = s^4 \psni, \qquad \gnis = s^2 \gni.
\end{equation}
It is clear that $(N_{i,s}, \phnis, \psnis, \gnis)$ is again a $\G$~manifold, as we have simply scaled the $\G$~structure by a constant. We claim that $N_{i,s}$ is still asymptotically conical with the same asymptotic $\G$~cone $(C_i, \phci, \psci, \gci)$ and the same rate $\nu_i$. To see this, define $h_{i,s} : (sR, \infty) \times \Sigma_i \to N_i \backslash L_i$ by $h_{i,s} (r, \sigma) = h_i (s^{-1}r, \sigma)$. That is, $h_{i,s} = h_i \circ (\mathbf{s^{-1}})$, where $\mathbf{s^{-1}}$ is a dilation as defined in~\eqref{dilationeq}. Then we have
\begin{align} \nonumber
h_{i,s}^* ( \phnis) - \phci & = (\mathbf{s^{-1}})^* \circ h_i^* (s^3 \phni) - \phci = s^3 \, (\mathbf{s^{-1}})^* ( h_i^* (\phni) ) -  s^3 \, (\mathbf{s^{-1}})^* (\phci) \\ \label{scalingformeq} & = s^3 \, (\mathbf{s^{-1}})^*
(h_i^* (\phni) -  \phci)
\end{align}
where we have used the fact that $\phci$ is a dilation-equivariant $3$-form. Now using~\eqref{scalemetriceq}, we see that
\begin{equation*}
| (h_{i,s}^* (\phnis) - \phci ) |_{\gci}  = \, O (r^{\nu_i}),
\end{equation*}
and similarly
\begin{equation*}
| \nabc^j ( h_{i,s}^* (\phnis) - \phci) |_{\gci} = \, O (r^{\nu_i - j}) \qquad
\forall j \geq 0
\end{equation*}
as claimed. Essentially, all we have done here is to rescale each asymptotically conical $\G$~manifold $N_i$ so that the value $R$ where $N_i \backslash L_i$ resembles $(R, \infty) \times \Sigma_i$ is changed to $sR$. By taking $s$ small enough so that $sR < \e$, we will be able to identify the subset $(2sR, \e) \times \Sigma_i$ of $N_{i,s}$ with the subset $(2sR, \e) \times \Sigma_i$ of a neighbourhood of the singularity $x_i$ of $M$ in order to be able to glue the two manifolds together. This will all be made more precise now.

For each asymptotically conical $\G$~manifold $N_{i,s}$, we want to keep the compact part $L_i$ together with a little bit of the part which looks like a cone. Specifically, for each $i = 1, \ldots, n$, define
\begin{equation} \label{ACpiecedefneq}
P_{i,s} = L_i \, \cup \, h_{i,s} \left( (sR', \e) \times \Sigma_i \right),
\end{equation}
where $R' = 2R$. The reason for using $2R$ rather than $R$ will be apparent in Section~\ref{formsconstructionsec}. See Figure~\ref{GLUE1fig} for a picture of $P_{i,s}$.
\begin{figure} [ht]
\centering
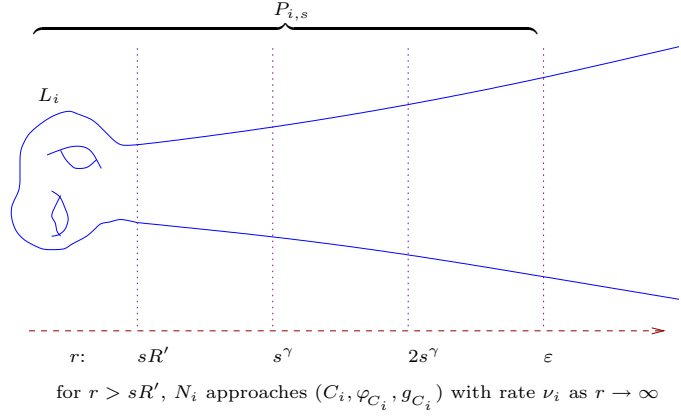
\caption{The region $P_{i,s}$ of the scaled AC manifold $N_{i,s}$}
\label{GLUE1fig}
\end{figure}
Similarly for the $\G$~manifold $M$ with isolated conical singularities, we want to keep the compact part $K = M' \backslash \sqcup_{i=1}^n S_i$ together with a little bit of each $S_i$ which looks like a cone. Specifically, we define
\begin{equation} \label{CSpiecedefneq}
Q_s = K \, \cup \, \sqcup_{i=1}^n f_i \left( (sR', \e) \times \Sigma_i \right).
\end{equation}
See Figure~\ref{GLUE2fig} for a picture of the part of $Q_s$ near $x_i$. The exponent $\gamma$ in both Figures~\ref{GLUE1fig} and~\ref{GLUE2fig} will be explained in Section~\ref{formsconstructionsec}.
\begin{figure} [ht]
\centering
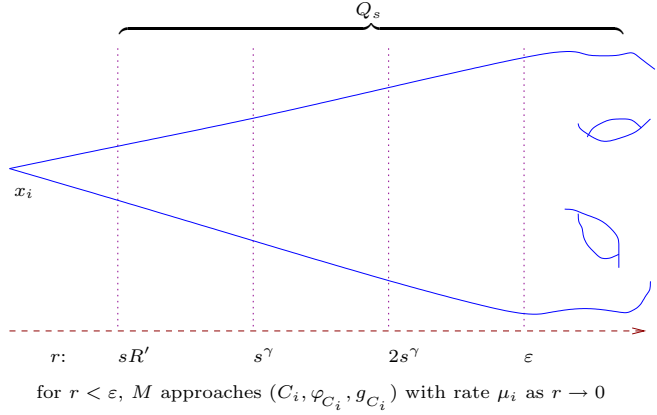
\caption{Part of the region $Q_s$ of the manifold $M$ with ICS}
\label{GLUE2fig}
\end{figure}
Taking the composition $h_{i,s} \circ f_i^{-1}$ of the two diffeomorphisms gives an identification which allows us to define the \emph{$i^{\text{th}}$ overlapping region} $U_{i,s}$ by
\begin{equation} \label{overlapdefneq}
U_{i,s} \, = \, (sR', \e) \times \Sigma_i \, \cong \, f_i \left( (sR', 2 \e) \times \Sigma_i \right) \, \cong \, h_{i,s} \left( (sR', 2 \e ) \times
\Sigma_i \right).
\end{equation}
We use this identification to perform our glueing, by defining the \emph{smooth, compact} $7$-manifold $\tilde M_s$ to be
\begin{equation} \label{Msdefneq}
\tilde M_s \, = \, \, \left( \sqcup_{i=1}^n  (P_{i,s} \backslash U_{i,s}) \right) \, \sqcup \, \left( Q_s \backslash \sqcup_{i=1}^n U_{i,s} \right) \, \sqcup \, \left( \sqcup_{i=1}^n U_{i,s} \right).
\end{equation}
This defines $\tilde M_s$ as a smooth manifold. See Figure~\ref{GLUE3fig} for a picture of $\tilde M_s$.
\begin{figure} [ht]
\centering
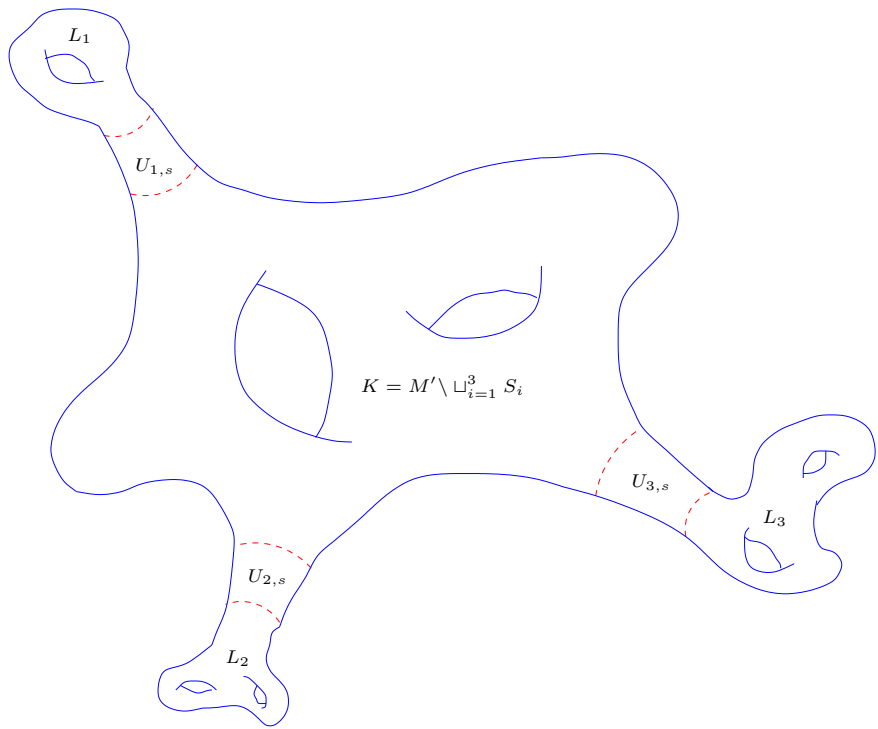
\caption{The smooth compact manifold $\tilde M_s$}
\label{GLUE3fig}
\end{figure}
It is clear that for all small $s$ the $\tilde M_s$'s are diffeomorphic, and that the `limit' as $s \to 0$ of $\tilde M_s$ is $M$, our original compact manifold with isolated singularities. 

Essentially, we have done the topology of the desingularization, and now we need to do the geometry. The next step is to construct a closed $\G$~structure $\ph_s$ on $\tilde M_s$ with small torsion. This is where the rates $\mu_i$ and $\nu_i$ will be important.

\subsection{Construction of the $\G$~structure $\ph_s$ on $\tilde M_s$}
\label{formsconstructionsec}

We will now construct a closed $\G$~structure $\ph_s$ on $\tilde M_s$, by patching together the torsion-free $\G$~structures on $M'$ and the $N_i$'s. Strictly speaking, we will actually construct a pair
$(\ph_s, \ps_s)$, which are non-degenerate closed $3$-forms and $4$-forms on $\tilde M_s$, respectively, but $\ps_s$ will not equal $\st_{g_{\ph_s}} \ph_s$, although it will be `close' in a sense to be made precise, when $s$ is sufficiently small. It is from this pair that we will be able to measure the torsion of $\ph_s$. We will see that sometimes the construction of the pair $(\ph_s, \ps_s)$ will only be possible if certain topological conditions are satisfied.  Finally in Section~\ref{torsionfreesec} we will show that for $s$ sufficiently small, $\ph_s$ has small enough torsion to apply Theorem~\ref{joycethm}.

In this section we will require two analytic results, Theorems~\ref{ACasymptoticexpansionthm} and~\ref{obsthm}, whose proofs will be postponed until Sections~\ref{ACgaugefixsec} and~\ref{obstructionsec}, respectively, as they require Fredholm theory on weighted Sobolev spaces and are best treated separately.

The basic idea behind the construction of the pair $(\ph_s, \ps_s)$ on $\tilde M_s$ is that we need to smoothly interpolate between the pair $(\phnis, \psnis)$ on $N_{i,s}$ and the pair $(\phm, \psm)$ for each $i = 1, \ldots, n.$ To be able to do this, we first need a good asymptotic expansion of the $\G$~structure near infinity for an asymptotically conical $\G$~manifold.

\begin{defn} \label{gaugefixdefn}
Let $N_i$ be an AC $\G$~manifold, with diffeomorphism $h_i : (R, \infty) \times \Sigma_i \to N_i \backslash L_i$. We say that $h_i$ satisfies the \emph{gauge-fixing condition} if $h_i^*(\phni) - \phci$ lies in $\wtht$ with respect to $\phci$.
\end{defn}
\begin{rmk} \label{gaugefixrmk}
The reason for making this definition is the following. One can show that if the gauge-fixing condition is satisfied, the forms $h_i^*(\phn) - \phci$ and $h_i^*(\psn) - \psci$ satisfy \emph{elliptic} equations with respect to weighted Sobolev spaces, and as such they possess improved regularity properties. This will become apparent in Section~\ref{ACgaugefixsec}. In the sequel~\cite{KCSmoduli} we show that one can always find a diffeomorphism $h_i$ for which the gauge-fixing condition is satisfied.
\end{rmk}
\begin{rmk} \label{gaugefixBSrmk}
For the Bryant--Salamon examples of asymptotically conical $\G$~manifolds, given by Examples~\ref{asds4ex},~\ref{asdcp2ex}, and~\ref{spins3ex}, it is trivial to verify that one can always find a diffeomorphism for which the gauge-fixing condition is satisfied. This is because these examples are \emph{cohomogeneity one}, which means the $\G$~structure is described by functions of one variable. The gauge-fixing condition is then equivalent to a change of variables (on the asymptotic end) satisfying a linear ordinary differential equation, which can be integrated exactly.
\end{rmk}

For our purposes, the importance of the gauge-fixing condition is the following theorem.
\begin{thm} \label{ACasymptoticexpansionthm}
Suppose that $N_i$ is an asymptotically conical $\G$~manifold with rate $\nu_i \leq -3$, and that $h_i$ satisfies the gauge-fixing condition given in Definition~\ref{gaugefixdefn}. Then on the subset $(2R, \infty) \times \Sigma_i$ of the cone $C_i$ we can write
\begin{align} \label{threeformasymptoticexpansioneq}
h_i^*(\phni) \, & = \, \phci + \xi_i + d\zeta_i,
\\ \label{fourformasymptoticexpansioneq}
h_i^*(\psni) \, & = \, \psci + \eta_i - \stci \xi_i + d\theta_i.
\end{align}
where $\xi_i$ is a harmonic $3$-form, homogeneous of order $-3$, and in $\wtht$ with respect to $\phci$, $\eta_i$ is a harmonic $4$-form, homogeneous of order $-4$, and $\zeta_i$ and $\theta_i$ are $2$-forms and $3$-forms on $(2R, \infty) \times \Sigma_i$, respectively, satisfying
\begin{equation} \label{ACinterpeq3new}
| \nabci^j \zeta_i |_{\gci} = \, \, O(r^{\nu'_i + 1 - j}), \qquad \quad | \nabci^j \theta_i |_{\gci} = \, \, O(r^{\nu'_i + 1 - j}), \qquad \forall j \geq 0,
\end{equation}
where $\nu'_i = - 4$. Furthermore, $[\xi_i] = \Phi(N_i)$ and $[\eta_i] = \Psi(N_i)$, where $\Phi(N_i)$ and $\Psi(N_i)$ are the cohomological invariants of the AC $\G$~manifold $N_i$ from Definition~\ref{ACinvariantsdefn}.
\end{thm}
Theorem~\ref{ACasymptoticexpansionthm} is proved in Section~\ref{ACgaugefixsec}. From now on we will assume that all of our AC $\G$~manifolds $N_1, \ldots, N_n$ have been gauge-fixed, so the asymptotic expansions in~\eqref{threeformasymptoticexpansioneq} and~\eqref{fourformasymptoticexpansioneq} apply. As noted in Remark~\ref{gaugefixBSrmk}, this can always be done for all the known examples of asymptotically conical $\G$~manifolds.

\begin{rmk} \label{minus3rmk}
It is because we must use Theorem~\ref{ACasymptoticexpansionthm} (see also Remark~\ref{nogaugefixrmk}) which requires each $\nu_i \leq -3$, that we need to make this restriction on the $\nu_i$'s. If some $\nu_{i_o} \in (-3,0)$, then the question of whether or not the glueing construction could be possible would depend on the spectrum of the Laplacian on the particular link $\Sigma_{i_0}$. Therefore, only if all $\nu_i \leq -3$ can we hope to be able to prove a general result like our Theorem~\ref{mainthm}.
\end{rmk}

Using the scaling in~\eqref{scaledACdefneq}, and the fact that $\xi_i$ and $\eta_i$ are dilation-invariant (see Remark~\ref{homoformsfirstrmk}), we can now argue just as in equation~\eqref{scalingformeq} to obtain
\begin{align} \label{scaledACinterpeq1}
h_{i,s}^*(\phnis) & \, = \, \phci + s^3 \xi_i + d \zeta_{i,s}, \\ \label{scaledACinterpeq2} h_{i,s}^*(\psnis) & \, = \, \psci + s^4 \eta_i - s^3 \stci \xi_i + d \theta_{i,s},
\end{align}
where $\zeta_{i,s} (r, \sigma) = s^3 \, \zeta_i (s^{-1}r, \sigma)$ and $\theta_{i,s} (r, \sigma) = s^4 \, \theta_i (s^{-1}r, \sigma)$. Now we can compute
\begin{align*}
| \nabci^j \zeta_{i,s} (r, \sigma) |_{\gci(r,\sigma)} & \, = \, s^3 | \nabci^j \zeta_i (s^{-1}r, \sigma) |_{\gci(r, \sigma)} \, = \, s^3 s^{-2-j} | \nabci^j \zeta_i (s^{-1}r, \sigma) |_{\gci(s^{-1}r, \sigma)} \\ & \, = \, s^{1-j} \, O( (s^{-1}r)^{\nu_i' + 1 - j} ) \, = \, s^{-\nu_i'} \, O(r^{\nu_i' + 1 - j}),
\end{align*}
where we have used~\eqref{scalemetriceq} and the fact that $\nabci^j \zeta_i \in (T^*)^j \otimes \Lambda^2 (T^*)$. We can compute similarly for $\theta_{i,s}$, and find that
\begin{equation} \label{scaledACinterpeq3}
| \nabci^j \zeta_{i,s} |_{\gci} = \, \, s^{-\nu'_i} \, O(r^{\nu'_i + 1 - j}), \qquad \quad | \nabci^j \theta_{i,s} |_{\gci} = \, \, s^{-\nu'_i} \, O(r^{\nu'_i + 1 - j}), \qquad \forall j \geq 0,
\end{equation}
on $(R', \infty) \times \Sigma_i$, where $R' = 2R$.

Now consider the $i^{\text{th}}$ conical singularity $x_i$ of $M$. By Proposition~\ref{CSclassprop}, on $(0, \e) \times \Sigma_i$ we can write
\begin{align} \label{CSinterpeq1}
f_i^*(\phm) & \, = \, \phci + d \alpha_i, \\ \label{CSinterpeq2} f_i^*(\psm) & \, = \, \psci + d \beta_i,
\end{align}
where $\alpha_i$ is a $2$-form and $\beta_i$ is a $3$-form on $(0, \e) \times \Sigma_i$, and by Lemma~\ref{exactformslemma} we know that
\begin{equation} \label{CSinterpeq3}
| \nabci^j \alpha_i |_{\gci} = \, \, O(r^{\mu_i + 1 - j}), \qquad \quad | \nabci^j \beta_i |_{\gci} = \, \, O(r^{\mu_i + 1 - j}), \qquad \forall j \geq 0.
\end{equation}

In order to be able to patch together the $\G$~structures on $M'$ and the $N_{i,s}$'s smoothly to obtain a $\G$~structure $\ph_s$ on $\tilde M_s$, we will arrange that in the $i^{\text{th}}$ overlap region
$U_{i,s}$, the forms will be in the same cohomology class, and hence differ by an exact piece.
Comparing equations~\eqref{scaledACinterpeq1} and~\eqref{scaledACinterpeq2} with equations~\eqref{CSinterpeq1} and~\eqref{CSinterpeq2}, we see that we need to replace the forms $\phm$ and $\psm$ on $M'$ with $\phm + s^3 \xi$ and $\psm + s^4 \eta - s^3 \stm \xi$ that are asymptotic, at the appropriate rate, to $\phci + s^3 \xi_i$ and $\psci + s^4 \eta_i - s^3 \stci \xi_i$ near $x_i$, respectively. In fact, for the construction to work, we need $\eta$ to be closed but $\xi$ to be both closed \emph{and} coclosed. There are topological obstructions that can prevent this from being possible. The resolution of this obstruction problem is given by the following theorem.

\begin{thm} \label{obsthm}
Let $M$ be a compact $\G$~manifold with isolated conical singularities. For each $i$, let $\xi_i$ and $\eta_i$ be $3$-forms and $4$-forms, homogeneous of orders $-3$ and $-4$, respectively, and closed and coclosed on each cone $C_i$, with respect to $\gci$. Suppose that
\begin{align} \label{obsconditionseq1}
& \oplus_{i=1}^n [\xi_i] \in \oplus_{i=1}^n H^3(\Sigma_i, \R) \text{ lies in the image of } \Upsilon^3 : H^3(M', \R) \to \oplus_{i=1}^n H^3(\Sigma_i, \R), \\ \label{obsconditionseq2} & \oplus_{i=1}^n [\eta_i] \in \oplus_{i=1}^n H^4(\Sigma_i, \R) \text{ lies in the image of } \Upsilon^4 : H^4(M', \R) \to \oplus_{i=1}^n H^4(\Sigma_i, \R).
\end{align}
where the maps $\Upsilon^k$ are given in Definition~\ref{topdefn}. Then for $\delta > 0$ sufficiently small, there exists a smooth $3$-form $\xi$ and a smooth $4$-form $\eta$ on $M'$ such that
\begin{align} \nonumber
& d \xi = 0, \qquad d^*_{\gm} \xi = 0, \qquad d \eta = 0, \\
\label{obseq1} & | \nabci^j \! ( f_i^* (\xi) - \xi_i ) |_{\gci} = \, \, O (r^{-3 + \delta - j})
\qquad \forall j \geq 0, \\ \label{obseq2} & | \nabci^j \! ( f_i^* (\eta) - \eta_i ) |_{\gci} = \, \, O (r^{-4 + \delta - j}) \qquad \forall j \geq 0.
\end{align}
Notice that we do not say that $\eta$ is coclosed. Furthermore, the $3$-form $\xi$ is in $\Lambda^3_{27}$ with respect to the $\G$~structure $\phm$.
\end{thm}
Theorem~\ref{obsthm} is proved in Section~\ref{obstructionsec}. Assuming this result for the moment, we now continue our construction of the pair $(\ph_s, \ps_s)$.

\begin{cor} \label{xietaestimatescor}
Near the $i^{\text{th}}$ singular point $x_i$, the $3$-form $f_i^*(\xi)$ and the $4$-form $f_i^*(\eta)$ satisfy
\begin{equation} \label{xietaestimateeq}
|f_i^*(\xi)|_{\gci} = O(r^{-3}), \qquad \qquad |f_i^*(\eta)|_{\gci} = O(r^{-4}).
\end{equation}
\end{cor}
\begin{proof}
By equation~\eqref{obseq1} we know that $|f_i^*(\xi)|_{\gci} \leq |\xi_i|_{\gci} + O(r^{-3 + \delta})$, and also $|\xi_i|_{\gci} = O(r^{-3})$, so $|f_i^*(\xi)|_{\gci} = O(r^{-3})$. The proof for $\eta$ is identical using~\eqref{obseq2}.
\end{proof}
\begin{cor} \label{Eestimatecor}
If $\delta >0$ is sufficiently small, then near the $i^{\text{th}}$ singular point $x_i$, the $4$-form $f_i^*(\stm \xi)$ and the $4$-form $\stci \xi_i$ satisy
\begin{equation} \label{Etempeq}
| f_i^*(\stm \xi) - \stci \xi_i |_{\gci} = \, O(r^{-3 + \delta}),
\end{equation}
and as a result we can say that
\begin{equation} \label{xistareq}
f_i^*(\stm \xi) = \stci \xi_i - dE_i,
\end{equation}
for some smooth $3$-form $E_i$ on $(0, \e) \times \Sigma_i$, where
\begin{equation} \label{Eesteq}
| \nabci^j E_i |_{\gci} = \, \, O(r^{-2 + \delta - j}).
\end{equation}
\end{cor}
\begin{proof}
We compute that
\begin{align*}
|f_i^*(\stm \xi) - \stci \xi_i |_{\gci} & \leq |f_i^*(\stm \xi) - \stci (f_i^* (\xi)) |_{\gci} + | \stci (f_i^*(\xi)) - \stci \xi_i |_{\gci} \\ & =  |f_i^*(\stm \xi) - \stci (f_i^* (\xi)) |_{\gci} + | f_i^*(\xi) - \xi_i |_{\gci} \\ & = O(r^{-3 + \mu_i})
+ O(r^{-3 + \delta})
\end{align*}
where we have used~\eqref{xietaestimateeq} and Lemma~\ref{starcomparelemma} on the first term, and~\eqref{obseq1} and the fact that $\stci$ is an isometry with respect to $\gci$ on the second term. This gives~\eqref{Etempeq} as long as we choose $\delta$ less than each $\mu_i$. Now the $3$-form 
$f_i^*(\stm \xi) - \stci \xi_i$ is closed by Theorem~\ref{obsthm}, and so equations~\eqref{xistareq} and~\eqref{Eesteq} now follow from Lemma~\ref{exactformslemma}.
\end{proof}
\begin{rmk} \label{Esignrmk}
The sign on $E_i$ in~\eqref{xistareq} is chosen for later convenience.
\end{rmk}

Recall that if each $\nu_i < -4$, then all the $\xi_i$'s and $\eta_i$'s vanish, so we can just take $\xi$ and $\eta$ to both be zero. This is the \emph{unobstructed case}. From now on, we will assume that the situation is either unobstructed, or that the conditions~\eqref{obsconditionseq1} and~\eqref{obsconditionseq2} are satisfied. By Lemma~\ref{exactformslemma} and equation~\eqref{obseq1}, since $\delta > 0$, we see that for each $i$ there exists a smooth $2$-form $A_i$ on $(0, \e) \times \Sigma_i$ such that
\begin{equation} \label{obsCSinterpeq1}
f_i^*(\phm + s^3 \xi) \, = \, \phci + s^3 \xi_i + d \alpha_i + s^3 dA_i,
\end{equation}
where we have also used equation~\eqref{CSinterpeq1}.
Similarly using equations~\eqref{obseq2},~\eqref{xistareq}, and~\eqref{CSinterpeq2}, we have that for each $i$ there exists  a smooth $3$-form $B_i$ on $(0, \e) \times \Sigma_i$ such that
\begin{equation} \label{obsCSinterpeq2}
f_i^*(\psm + s^4 \eta - s^3 \stm \xi) \, = \, \psci + s^4 \eta_i - s^3 \stci \xi_i + d \beta_i + s^4 dB_i + s^3 dE_i.
\end{equation}
Also, by Lemma~\ref{exactformslemma} and equations~\eqref{CSinterpeq3}, we have
\begin{align} \label{obsCSinterpeq3}
| \nabci^j \alpha_i |_{\gci} & = \, O(r^{\mu_i + 1 - j}), & | \nabci^j A_i |_{\gci} & = \, O(r^{-2 + \delta - j}), \qquad \forall j \geq 0, & \\ \label{obsCSinterpeq4} | \nabci^j \beta_i |_{\gci} & = \, O(r^{\mu_i + 1 - j}), &
| \nabci^j B_i |_{\gci} & = \, O(r^{-3 + \delta - j}), \qquad \forall j \geq 0. &
\end{align}

We are now ready to construct a smooth $3$-form $\ph_s$ and a smooth $4$-form $\ps_s$ on $\tilde M_s$. Let us take $\g \in (0,1)$, which will be chosen later. It will become clear later why we need to do this. From now on we will consider only those values of $s > 0$ which are small enough so that
\begin{equation} \label{sbounds}
s R' < s^{\g} < 2 s^{\g} < \e.
\end{equation}
Recall that $R' = 2R$. These inequalities hold for all $s < \mathrm{max}
( {\left( \frac{\e}{2} \right)}^{\frac{1}{\g}}, 
{\left( \frac{1}{R'} \right)}^{\frac{1}{1-\g}} )$. This enables us to split up each overlapping region $U_{i,s}$ into three disjoint pieces:
\begin{equation*}
U_{i,s} = (sR', s^{\g}) \times \Sigma_i \, \sqcup \, [s^{\g}, 2 s^{\g}] \times \Sigma_i \, \sqcup \, (2s^{\g}, \e) \times \Sigma_i.
\end{equation*}
The reader should refer again to Figures~\ref{GLUE1fig} and~\ref{GLUE2fig}.

We can write the smooth compact manifold $\tilde M_s$ as the (non-disjoint) union of the \emph{inner piece} $K \, \cup \, \sqcup_{i=1}^n f_i \left( (2s^{\g}, \e) \times \Sigma_i \right)$, together with an \emph{outer piece} $L_i \, \cup \, h_{i,s} \left( (sR', s^{\g}) \times \Sigma_i \right)$ and an overlapping region $U_{i,s} \, = \, (sR', \e) \times \Sigma_i$ for each $i = 1 , \ldots, n$. Thinking about it in this way, the inner piece intersects each $U_{i,s}$ in the set $f_i \left( (2s^{\g}, \e) \times \Sigma_i \right) \, \cong \, (2s^{\g}, \e) \times \Sigma_i$, and each inner piece intersects its associated $U_{i,s}$ in the set 
$h_{i,s} \left( (sR', s^{\g}) \times \Sigma_i \right) \, \cong \, (sR', s^{\g}) \times \Sigma_i$.

Let $u: (0, \infty) \to \R$ be any smooth increasing function such that
\begin{equation*}
u(r) = \begin{cases} 0 & \text{for } 0 < r \leq 1, \\ 1 & \text{for } 2 \leq r < \infty. \end{cases}
\end{equation*}
We will use $u(r)$ to \emph{interpolate} between the pair $(\phm + s^3 \xi, \psm + s^4 \eta - s^3 \stm \xi)$
on $M'$ and the pairs $(\phnis, \psnis)$ on the $N_{i,s}$'s.
Define $u_s(r) = u(s^{-\g} r)$. Then clearly
\begin{equation*}
u_s(r) = \begin{cases} 0 & \text{for } 0 < r \leq s^{\g}, \\ 1 & \text{for } 2s^{\g} \leq r < \infty. \end{cases}
\end{equation*}
Using $u_s$, we define a smooth $2$-form $\rho_i$ and a smooth $3$-form $\tau_i$ on each $(sR', \e) \times \Sigma_i$ as follows:
\begin{align} \label{rhodefneq}
\rho_i & = (\alpha_i + s^3 A_i)u_s  + \zeta_{i,s}(1 - u_s), \\ \label{taudefneq} \tau_i & = (\beta_i + s^4 B_i + s^3 E_i)u_s  + \theta_{i,s}(1 - u_s).
\end{align}

\begin{defn} \label{phspssdefn}
We define a smooth $3$-form $\ph_s$ on $\tilde M_s$ as follows:
\begin{equation} \label{phsdefneq}
\ph_s = \begin{cases} \phm + s^3 \xi & \text{on } K \, \cup \, \sqcup_{i=1}^n f_i \left( (2s^{\g}, \e) \times \Sigma_i \right), \\ \phci + s^3 \xi_i + d \rho_i & \text{on }  (sR', \e) \times \Sigma_i, \\ \phnis & \text{on } L_i \, \cup \, h_{i,s} \left( (sR', s^{\g}) \times \Sigma_i \right), \end{cases}
\end{equation}
and we define a smooth $4$-form $\ps_s$ on $\tilde M_s$ as follows:
\begin{equation} \label{pssdefneq}
\ps_s = \begin{cases} \psm + s^4 \eta - s^3 \stm \xi & \text{on } K \, \cup \, \sqcup_{i=1}^n f_i \left( (2s^{\g}, \e) \times \Sigma_i \right), \\ \psci + s^4 \eta_i - s^3 \stci \xi_i + d \tau_i & \text{on }  (sR', \e) \times \Sigma_i, \\ \psnis & \text{on } L_i \, \cup \, h_{i,s} \left( (sR', s^{\g}) \times \Sigma_i \right). \end{cases}
\end{equation}
\end{defn}

We need to check that $\ph_s$ is well defined, in the regions where we have given it two definitions. When $r \leq s^{\g}$, then $\rho_i = \zeta_{i,s}$ and therefore by~\eqref{scaledACinterpeq1} we have $\ph_s = h^*_{i,s} (\phnis)$ on $(sR', s^{\g}) \times \Sigma_i$ and $\ph_s = \phnis$ on $h_{i,s} \left( (sR', s^{\g}) \times \Sigma_i \right)$, and these are indeed identified with each other under the diffeomorphism $h_{i,s}$. Whereas when $r \geq 2 s^{\g}$, then $\rho_i = \alpha_i + s^3 A_i$ and therefore by~\eqref{obsCSinterpeq1} we have $\ph_s = f^*_i (\phm + s^3 \xi)$ on $(2s^{\g}, \e) \times \Sigma_i$ and $\ph_s = \phm + s^3 \xi$ on $f_i \left( (2s^{\g}, \e) \times \Sigma_i \right)$, and these are identified with each other under the diffeomorphism $f_i$. Similary one checks using~\eqref{scaledACinterpeq2} and~\eqref{obsCSinterpeq2} that $\ps_s$ is also well defined. It is clear that $\ph_s$ and $\ps_s$ are both \emph{closed} forms, for all $s > 0$.

\subsection{Existence of the torsion-free $\G$~structure $\tilde{\ph}_s$ on $\tilde M_s$}
\label{torsionfreesec}

In this section, we will show that from the pair $(\ph_s, \ps_s)$ constructed in Section~\ref{formsconstructionsec}, we can satisfy the hypotheses of Joyce's Theorem~\ref{joycethm} to conclude that $\tilde M_s$ admits a torsion-free $\G$~structure $\tilde \ph_s$.

Since non-degeneracy is an open condition, it will follow from the estimates in Proposition~\ref{glueestimatesprop} that for $s$ sufficiently small, $\ph_s$ and $\ps_s$ are non-degenerate $3$-forms and $4$-forms, respectively. However, $\ps_s$ is \emph{not} equal to $\st_{g_s} \ph_s$, where $g_s$ is the metric determined by $\ph_s$, so $\ph_s$ is not coclosed with respect to its induced metric, and hence not torsion-free. However, for $s$ sufficiently small, we will see that $\ps_s$ is close to $\st_{g_s} \ph_s$ in a suitable sense, and this difference, together with its derivative, is a measure of the torsion of $\ph_s$. 

\begin{defn} \label{chidefn}
For each $s$, we define a smooth $3$-form $\chi_s$ on $\tilde M_s$ by
\begin{equation} \label{chidefneq}
\chi_s \, = \, \ph_s - \st_{g_s} \ps_s.
\end{equation}
\end{defn}
\begin{rmk} \label{chirmk}
In the notation of Lemma~\ref{quadlemma} and Remark~\ref{quadrmk}, we have $\chi_s = \ph_s - \Theta^{-1} (\ps_s)$.
\end{rmk}

\begin{lemma} \label{chilemma}
The form $\chi_s$ satisfies $d^*_{g_s} \chi_s = d^*_{g_s} \phi_s$ for all $s$.
\end{lemma}
\begin{proof}
We compute
\begin{equation*}
d^*_{g_s} \chi_s = d^*_{g_s} ( \ph_s - \st_{g_s} \ps_s ) = d^*_{g_s} \ph_s - \st_{g_s} d \st_{g_s} \st_{g_s} \ps_s = d^*_{g_s} \ph_s - \st_{g_s} d \ps_s = d^*_{g_s} \ph_s
\end{equation*}
using the fact that $d^*_{g_s} = \st_{g_s} d \st_{g_s}$ on $4$-forms, that $\st_{g_s}^2 = 1$, and that $d \ps_s = 0$. 
\end{proof}

This $3$-form $\chi_s$ will play the role of $\chi$ in Theorem~\ref{joycethm}. To apply this result, we need to estimate various norms of $\chi_s$ in the different regions of $\tilde M_s$.
\begin{prop} \label{chiprop}
The $3$-form $\chi_s$ on $\tilde M_s$ is given by
\begin{equation} \label{chiregionseq}
\chi_s = \begin{cases} -\Jpm (s^4 \eta) - \Gpm (s^4 \eta - s^3 \stm \xi) & \text{on } K \, \cup \, \sqcup_{i=1}^n f_i \left( (2s^{\g}, \e) \times \Sigma_i \right), \\ \left[ \begin{array}{l} d\rho_i - \Jpci (s^4 \eta_i + d\tau_i) \\ \, {}- \Gpci (s^4 \eta_i - s^3 \stci \xi_i + d \tau_i) \end{array} \right] & \text{on }  [s^{\g}, 2s^{\g}] \times \Sigma_i, \\ \, \, 0 & \text{on } L_i \, \cup \, h_{i,s} \left( (sR', s^{\g}) \times \Sigma_i \right), \end{cases}
\end{equation}
in the different regions used in Definition~\ref{phspssdefn}. Here the maps  $J$ and $G$ are as defined in Remark~\ref{quadrmk} and equation~\eqref{Jdefneq}.
\end{prop}
\begin{proof}
On the region $L_i \, \cup \, h_{i,s} \left( (sR', s^{\g}) \times \Sigma_i \right)$, we have $\ps_s = \psnis$ and $\ph_s = \phnis$, so
\begin{equation*}
\chi_s = \phnis - \Theta^{-1}(\psnis) = 0.
\end{equation*}
In other words, the $\G$~structure $\ph_s$ is already torsion-free on these regions. Now on the inner region $K \, \cup \, \sqcup_{i=1}^n f_i \left( (2s^{\g}, \e) \times \Sigma_i \right)$, we have $\ps_s = \psm + s^4 \eta - s^3 \stm \xi$, and hence by~\eqref{quadeq3} we have
\begin{equation*}
\Theta^{-1} (\ps_s) = \phm + \Jpm (s^4 \eta - s^3 \stm \xi ) + \Gpm (s^4 \eta - s^3 \stm \xi ).
\end{equation*}
However, since by Theorem~\ref{obsthm} we know that $\xi$ is in $\wtht$ with respect to $\phm$, equation~\eqref{Jdefneq} tells us that $\Jpm (\stm \xi) = - \xi$. Therefore
\begin{equation*}
\Theta^{-1} (\ps_s) = \phm + s^3 \xi + \Jpm (s^4 \eta) + \Gpm (s^4 \eta - s^3 \stm \xi ).
\end{equation*}
Now since $\ph_s = \phm + s^3 \xi$, we obtain the expression for $\chi_s$ in the first line of~\eqref{chiregionseq}. Finally, for the overlap regions $[s^{\g}, 2s^{\g}] \times \Sigma_i$, the calculation is similar, this time using the fact that $\xi_i$ is in $\wtht$ with respect to $\phci$, which was shown in Proposition~\ref{cones27prop}.
\end{proof}
\begin{rmk} \label{disjointrmk}
Notice in~\eqref{chiregionseq} we are now writing the overlap regions as $[s^{\g}, 2s^{\g}] \times \Sigma_i$ as opposed to $(sR', \e) \times \Sigma_i$ as was used in Definition~\ref{phspssdefn}, so that now all the regions are disjoint. This will make is easier to estimate various norms in what follows, so we will decompose $\tilde M_s$  this way from now on.
\end{rmk}

Note that in the unobstructed case, we have $\xi = 0$ and $\eta = 0$. Thus $\chi_s$ also vanishes on the inner region, so $\ph_s$ only fails to be torsion-free in the annuli $[s^{\g}, 2 s^{\g}] \times \Sigma_i$ where we have interpolated between two torsion-free $\G$~structures. However in the obstructed case, the $\G$~structure $\ph_s$ also has torsion on the inner region coming from the corrections $\xi$ and $\eta$ to $\phm$ and $\psm$.

We are now ready to begin estimating various norms of $\chi_s$ in order to apply Theorem~\ref{joycethm}. We begin with estimates of $d\rho_i$ and $d\tau_i$.
\begin{prop} \label{glueestimatesprop}
Consider the $2$-forms $\rho_i$ and the $3$-forms $\tau_i$ given by equations~\eqref{rhodefneq} and~\eqref{taudefneq}. In the region $s^{\g} \leq r \leq 2 s^{\g}$, and with respect to the cone metric $\gci$, we have the following estimates:
\begin{align} \label{rhoestimateeq}
| d\rho_i |_{\gci} &  \leq  \, C ( s^{\g \mu_i} + s^{-\nu'_i(1 - \g)} + s^{3(1-\g) + \delta \g} ), \\ \label{rho2estimateeq} | \nabci d\rho_i |_{\gci} & \leq  \, C ( s^{\g \mu_i - \g} + s^{-\nu'_i(1 - \g) -\g} + s^{3(1-\g) + \delta \g - \g} ), \\ \label{tauestimateeq}
| d\tau_i |_{\gci} &  \leq  \, C ( s^{\g \mu_i} + s^{-\nu'_i(1 - \g)} + s^{3(1-\g) + \delta \g} ), \\ \label{tau2estimateeq} | \nabci d\tau_i |_{\gci} & \leq  \, C ( s^{\g \mu_i - \g} + s^{-\nu'_i(1 - \g) -\g} + s^{3(1-\g) + \delta \g - \g} ),
\end{align}
where each $C$ denotes some positive constant.
\end{prop}
\begin{proof}
Notice that $\frac{d}{dr} u_s(r) = s^{-\g} u_s'(r)$. Now from~\eqref{rhodefneq}, we see that
\begin{equation} \label{rhotempeq}
d\rho_i = s^{-\g} u_s' dr \wedge (\alpha_i + s^3 A_i) + u_s (d\alpha_i + s^3 dA_i) - s^{-\g} u_s' dr \wedge \zeta_{i.s} + (1 - u_s) d \zeta_{i,s}.
\end{equation}
Since $u_s$ and $u_s'$ are uniformly bounded functions of $r$ (independent of $s$), we have
\begin{equation*}
| d\rho_i |_{\gci} \, \leq \, C s^{-\g}  | \alpha_i + s^3 A_i |_{\gci} + C |d\alpha_i + s^3 dA_i |_{\gci} + C s^{-\g} |\zeta_{i.s} |_{\gci}  + C |d \zeta_{i,s}|_{\gci}.
\end{equation*}
Using equations~\eqref{desteq},~\eqref{obsCSinterpeq3}, and~\eqref{scaledACinterpeq3}, we compute \begin{equation*}
| d\rho_i |_{\gci} \, \leq \, C ( s^{-\g} r^{\mu_i + 1} + s^{3-\g} r^{-2 + \delta} + r^{\mu_i} + s^3 r^{-3 + \delta} + s^{-\g} s^{-\nu'_i} r^{\nu'_i + 1} + s^{-\nu'_i} r^{\nu'_i} ).
\end{equation*}
Since $s^{\g} \leq r \leq 2 s^{\g}$, we have $r^{a} \leq C s^{\g a}$ for any $a$. Therefore the above expression becomes
\begin{equation*}
| d\rho_i |_{\gci} \, \leq \, C ( s^{\g \mu_i} + s^{3(1-\g) + \delta \g} + s^{\g \mu_i} + s^{3(1-\g) + \delta \g} + s^{-\nu'_i(1 - \g)} + s^{-\nu'_i(1 - \g)} ),
\end{equation*}
which is~\eqref{rhoestimateeq}. To obtain~\eqref{rho2estimateeq}, take the covariant derivative of~\eqref{rhotempeq} and proceed as before.

For the estimates on $\tau_i$, equation~\eqref{taudefneq} shows that
\begin{equation} \label{tautempeq}
\begin{gathered}
d\tau_i \, = \, s^{-\g} u_s' dr \wedge (\beta_i + s^4 B_i + s^3 E_i) + u_s (d\beta_i + s^4 dB_i + s^3 dE_i) \\ {} - s^{-\g} u_s' dr \wedge \theta_{i.s} + (1 - u_s) d \theta_{i,s}.
\end{gathered}
\end{equation}
Now we use equations~\eqref{desteq},~\eqref{obsCSinterpeq4},~\eqref{Eesteq}, and~\eqref{scaledACinterpeq3}, and get
\begin{equation*}
| d\tau |_{\gci} \, \leq \, C( s^{\g \mu_i} + s^{4(1-\g) + \delta \g} + s^{3(1-\g) + \delta \g} + s^{-\nu'_i(1 - \g)} ),
\end{equation*}
where the $O( s^{4(1-\g) + \delta \g} )$ terms come from $B_i$ and the $O( s^{3(1-\g) + \delta \g} )$ terms come from $E_i$. But since $s < 1$, the former terms are smaller than, and can be absorbed by, the latter terms. This gives~\eqref{tauestimateeq}. The proof of~\eqref{tau2estimateeq} is similar.
\end{proof}
\begin{rmk} \label{gammarmk}
The estimates in Proposition~\ref{glueestimatesprop} explain why we needed to introduce the parameter $\g \in (0,1)$ to construct $\ph_s$ and $\ps_s$. If, instead of~\eqref{sbounds}, we had divided the interval $(sR', \e)$ into $sR' < R_1 < R_2 < \e$, which corresponds to $\g = 0$, or into $s R' < s R_1 < s R_2 <s \e$, which corresponds to $\g = 1$, then the estimates of various norms of $\chi_s$ would include terms of size $O(1)$, and they would not be small enough for us to apply Theorem~\ref{joycethm} later.
\end{rmk}

In Theorem~\ref{joycethm}, all the norms are measured with respect to the metric $g_s$ induced by $\ph_s$. However, it is easier for us to estimate various terms using the metrics $\gm$ from $\phm$ or $\gci$ from $\phci$. To translate these to estimates using $g_s$, we need the following observation.
\begin{lemma} \label{uniformequivlemma}
For $s$ sufficiently small, the metric $g_s$ is \emph{uniformly close} to the metric $\gm$ (in the $\gm$ metric) on the inner region $K \, \cup \, \sqcup_{i=1}^n f_i \left( (2s^{\g}, \e) \times \Sigma_i \right)$. This means that on this region
\begin{equation*}
\sup |\gm - g_s|_{\gm} \leq e, \qquad \qquad \sup |\gm^{-1} - g_s^{-1}|_{\gm} \leq e,
\end{equation*}
for some small $e < \frac{1}{2}$. As a result, it follows easily that for any tensor $\omega$ defined on the inner region,
\begin{equation} \label{unifeq1}
|\omega|_{g_s} \leq C |\omega|_{\gm}
\end{equation}
for some $C > 0$, independent of $s$ and $\omega$.

Similarly for $s$ sufficiently small, the metric $g_s$ is uniformly close to the metric $\gci$ (in the $\gci$ metric) on the overlap region $[s^{\g}, \e) \times \Sigma_i$, and hence
\begin{equation} \label{unifeq2}
|\omega|_{g_s} \leq C |\omega|_{\gci}.
\end{equation}
\end{lemma}
\begin{proof}
On the inner region $K \, \cup \, \sqcup_{i=1}^n f_i \left( (2s^{\g}, \e) \times \Sigma_i \right)$, from Definition~\ref{phspssdefn} we have $\ph_s - \phm = s^3 \xi$. On the subset $K$, which is a compact manifold with boundary, $|\xi|_{\gm}$ is bounded, so $s^3 |\xi|_{\gm} = O(s^{3})$. Meanwhile on the subset $\sqcup_{i=1}^n f_i \left( (2s^{\g}, \e) \times \Sigma_i \right)$, we claim that $s^3 |\xi|_{\gm} = O(s^{3(1-\g)})$. To see this, first we note that by~\eqref{CSdefneq2}, we have $|f_i^*(\gm) - \gci|_{\gci} = O(r^{\mu_i})$, which is uniformly bounded as $r \leq \e$. Therefore $|\xi|_{\gm} = |f_i^*(\xi)|_{f_i^*(\gm)} \leq C |f_i^*(\xi)|_{\gci}$ for some $C > 0$. But by~\eqref{obseq1} we know that $|f_i^*(\xi)|_{\gci} \leq |\xi_i|_{\gci} + O(r^{-3 + \delta})$, and also $|\xi_i|_{\gci} = O(r^{-3})$, so $s^3 |\xi|_{\gm} \leq C s^3 r^{-3} \leq C s^{3(1-\g)}$ as claimed, since $r^{-1} \leq s^{-\g}$. Therefore on the inner region we have
\begin{equation*}
\sup |\ph_s - \phm|_{\gm} \leq C( s^3 + s^{3(1-\g)})
\end{equation*}
which can be made as small as we want by taking $s$ sufficiently small. Therefore we can make $|g_s - \gm|_{\gm} < e$ for some $e < \frac{1}{2}$. The claim now follows.

On the region $[s^{\g}, \e) \times \Sigma_i$, we have from Definition~\ref{phspssdefn} that $\ph_s - \phci = s^3 \xi_i + d\rho_i$. Now $|\xi_i|_{\gci} = O(r^{-3})$, so $s^3 |\xi_i|_{\gci} \leq C s^{3(1-\g)} $ since $r^{-1} \leq C s^{-\g}$. Using~\eqref{rhoestimateeq}, we have
\begin{equation*}
\sup |\ph_s - \phci|_{\gci} \leq C( s^{3(1-\g)} + s^{\g \mu_i} + s^{-\nu'_i(1 - \g)} + s^{3(1-\g) + \delta \g} ),
\end{equation*}
and now the claim follows as in the first case.
\end{proof}
We will use~\eqref{unifeq1} and~\eqref{unifeq2} repeatedly in what follows, without explicit mention.
We can now estimate the pointwise norms of $\chi_s$ and $\nab{s} \chi_s$, where $\nab{s}$ is the covariant derivative with respect to the metric $g_s$ of $\ph_s$.

\begin{prop} \label{Czeroestimatesprop}
Let $\chi_s$ be as given in Proposition~\ref{chiprop}. Then we have the following pointwise estimates for $|\chi_s|_{g_s}$ and $|\nab{s} \chi_s|_{g_s}$, on the different regions that were used in~\eqref{chiregionseq}. On the subset $K$ of the inner region, we have
\begin{equation} \label{Czeroestimateseq1a}
|\chi_s|_{g_s} \leq \, C s^4, \qquad \qquad |\nab{s} \chi_s|_{g_s} \leq \, C s^4.
\end{equation}
On each subset $f_i \left( (2s^{\g}, \e) \times \Sigma_i \right)$ of the inner region, we have
\begin{equation} \label{Czeroestimateseq1b}
|\chi_s|_{g_s} \leq \, C s^4 r^{-4}, \qquad \qquad |\nab{s} \chi_s|_{g_s} \leq \, C s^4 r^{-4}.
\end{equation}
On the overlap regions $[s^{\g}, 2s^{\g}] \times \Sigma_i$ we have
\begin{equation} \label{Czeroestimateseq2}
\begin{aligned}
|\chi_s|_{g_s} & \leq \, C ( s^{\g \mu_i} + s^{-\nu'_i(1 - \g)} + s^{3(1-\g) + \delta \g} ), \\
|\nab{s} \chi_s|_{g_s} & \leq \, C ( s^{\g \mu_i - \g} + s^{-\nu'_i(1 - \g) -\g} + s^{3(1-\g) + \delta \g - \g} ),
\end{aligned}
\end{equation}
and finally on the outer regions $L_i \, \cup \, h_{i,s} \left( (sR', s^{\g}) \times \Sigma_i \right)$ we have
\begin{equation} \label{Czeroestimateseq3}
|\chi_s|_{g_s} = 0, \qquad \qquad |\nab{s} \chi_s|_{g_s} = 0.
\end{equation}
\end{prop}
\begin{proof}
We start with the inner region. Using~\eqref{chiregionseq} and~\eqref{Jesteq} and~\eqref{quadeq4}, we see that
\begin{align} \nonumber
|\chi_s|_{g_s} \leq C |\chi_s|_{\gm} & \leq C |\Jpm (s^4 \eta)|_{\gm} + C |\Gpm (s^4 \eta - s^3 \stm \xi)|_{\gm} \\ \nonumber & \leq C s^4 |\eta|_{\gm} + C |s^4 \eta - s^3 \stm \xi|^2_{\gm} \\ \label{innerCzerotempeq} & \leq  C s^4 |\eta|_{\gm} +  C {( s^4 |\eta|_{\gm} + s^3 |\xi|_{\gm} )}^2,
\end{align}
where we have used the fact that $\stm$ is an isometry. A similar computation yields
\begin{equation} \label{innerCzerotempeq2}
|\nab{s} \chi_s|_{g_s} \leq \, C s^4 |\nabm \eta|_{\gm} + C {( s^4 |\nabm \eta|_{\gm} + s^3 |\nabm \xi|_{\gm}) ( s^4 |\eta|_{\gm} + s^3 |\xi|_{\gm})}.
\end{equation}
To obtain the above estimate we need to use the fact that $\nabm$ commutes with $\stm$, and that $\phm$ is torsion-free, so that the $|\nabm \phm|_{\gm}$ terms in~\eqref{Jesteq} and~\eqref{quadeq4} are zero. Now the subset $K$ of the inner region is a compact manifold with boundary, so $|\xi|_{\gm}$, $|\eta|_{\gm}$, $|\nabm \xi|_{\gm}$, and $|\nabm \eta|_{\gm}$ are all bounded there. Hence by~\eqref{innerCzerotempeq} we have
\begin{equation*}
|\chi_s|_{g_s} \leq C ( s^4 + s^6 + s^7 + s^8) \leq C s^4
\end{equation*}
by absorbing smaller terms into the dominant term. Similarly by~\eqref{innerCzerotempeq2} we get $|\nab{s} \chi_s|_{g_s} \leq \, C s^4$. This proves~\eqref{Czeroestimateseq1a}. Now consider a subset $f_i ((2 s^{\g}, \e)  \times \Sigma_i)$ of the inner region. By~\eqref{CSdefneq2}, we have $|f_i^*(\gm) - \gci|_{\gci} = O(r^{\mu_i})$, which is uniformly bounded as $r \leq \e$. Therefore 
\begin{equation*}
|\xi|_{\gm} = |f_i^*(\xi)|_{f_i^*(\gm)} \leq C | f_i^*(\xi)|_{\gci} \leq C |\xi_i|_{\gci} + O(r^{-3 + \delta}) 
\leq C r^{-3},
\end{equation*}
using~\eqref{obseq1} and the fact that $\xi_i$ is homogenous of order $-3$. In exactly the same way, we can show that on $f_i ((2 s^{\g}, \e)  \times \Sigma_i)$, we have
\begin{equation} \label{mainestimatestempeq}
|\xi|_{\gm} \leq C r^{-3}, \qquad |\nabm \xi|_{\gm} \leq C r^{-4}, \qquad |\eta|_{\gm} \leq C r^{-4}, 
\qquad |\nabm \eta|_{\gm} \leq C r^{-5}.
\end{equation}
Hence~\eqref{innerCzerotempeq} and~\eqref{innerCzerotempeq2} then give
\begin{equation*}
|\chi_s|_{g_s} \leq \, C s^4 r^{-4}, \qquad \qquad |\nab{s} \chi_s|_{g_s} \leq \, C s^4 r^{-4},
\end{equation*}
which is~\eqref{Czeroestimateseq1b}. Here we have used the fact that the $s^k r^{-k}$ terms for $k > 4$ are dominated by $s^4 r^{-4}$ since $s r^{-1} < C s^{1-\g} < C$, because $s < 1$ and $\g \in (0,1)$.

We move on to the overlap regions. Again using~\eqref{chiregionseq} and~\eqref{Jesteq} and~\eqref{quadeq4}, we get
\begin{align*}
|\chi_s|_{g_s} \leq C |\chi_s|_{\gci} & \leq C |d\rho_i|_{\gci} +  C |\Jpci (s^4 \eta_i + d\tau_i)|_{\gci} + C |\Gpci (s^4 \eta_i - s^3 \stci \xi_i + d\tau_i)|_{\gci} \\ & \leq C |d\rho_i|_{\gci} +  C | s^4 \eta_i + d\tau_i |_{\gci} + C | s^4 \eta_i - s^3 \stci \xi_i + d\tau_i |^2_{\gci}.
\end{align*}
Now we use~\eqref{rhoestimateeq} and~\eqref{tauestimateeq} to obtain
\begin{equation} \label{Czerotempeq}
|\chi_s|_{g_s} \leq C ( s^{\g \mu_i} + s^{-\nu'_i(1 - \g)} + s^{3(1-\g) + \delta \g}) + C s^4 |\eta_i|_{\gci} +
C ( s^4 |\eta_i|_{\gci} + s^3 |\xi_i|_{\gci} + |d\tau_i|_{\gci})^2.
\end{equation}
We claim that the second and third terms above are smaller than, and can thus be absorbed by, the first term. To see this, recall that $|\xi_i|_{\gci} = O(r^{-3})$ and $|\eta_i|_{\gci} = O(r^{-4})$, and that $r^a \leq C r^{a \g}$ for $r \in [s^{\g}, 2s^{\g}]$, and thus
\begin{equation*}
s^4 |\eta_i|_{\gci} \leq C s^{4 (1-\g)}, \qquad \qquad s^3 |\xi_i|_{\gci} \leq C s^{3(1-\g)}. 
\end{equation*}
Since $\delta $ is close to zero, the $O(s^{4(1-\g)})$ term can be absorbed in $O(s^{3(1-\g) + \delta \g})$.
Now the final term in~\eqref{Czerotempeq} is
\begin{equation*}
( O(s^{3(1-\g)}) +  O(s^{\g \mu_i}) + O(s^{-\nu'_i(1 - \g)})^2.
\end{equation*}
Since $s < 1$, and $\delta$ is close to zero, every product of two terms in the above expression can be absorbed by one of the summands in the first term of~\eqref{Czerotempeq}. This proves the first estimate in~\eqref{Czeroestimateseq2}. The second estimate in~\eqref{Czeroestimateseq2} is obtained similarly to the estimate in~\eqref{innerCzerotempeq2}, this time using the fact that $\nabci$ commutes with $\stci$ and that $\phci$ is torsion-free, so the $|\nabci \phci|_{\gci}$ terms in~\eqref{Jesteq} and~\eqref{quadeq4} are zero.

Finally the equations in~\eqref{Czeroestimateseq3} are immediate from~\eqref{chiregionseq}.
\end{proof}

We can now use Proposition~\ref{Czeroestimatesprop} to estimate the norms needed in part i) of Theorem~\ref{joycethm}. 
\begin{thm} \label{mainestimatesthm}
Let $\chi_s$ be as given in Proposition~\ref{chiprop}. Then the following estimates hold for norms of $\chi_s$ on the smooth compact manifold $\tilde M_s$.
\begin{align} \label{Czerofinaleq}
{|| \chi_s ||}_{C^0} & \leq C ( s^{\g \mu_i} + s^{-\nu'_i(1 - \g)} + s^{3(1-\g) + \delta \g} ), \\ \label{L2finaleq} {|| \chi_s ||}_{L^2} & \leq C ( s^{\g \mu_i + \frac{7}{2}\g} + s^{-\nu'_i(1 - \g) + \frac{7}{2}\g} + s^{3(1-\g) + \delta \g + \frac{7}{2}\g} ) , \\ \label{L14finaleq} {|| d^*_{g_s} \chi_s ||}_{L^{14}} & \leq C ( s^{\g \mu_i - \frac{1}{2}\g} + s^{-\nu'_i(1 - \g) - \frac{1}{2}\g} + s^{3(1-\g) + \delta \g - \frac{1}{2}\g} ).
\end{align}
\end{thm}
\begin{proof}
We begin with the $C^0$ norms of $\chi_s$ and of $\nab{g_s} \chi_s$ with respect to $g_s$. Using the estimates in~\eqref{Czeroestimateseq1a},~\eqref{Czeroestimateseq1b},~\eqref{Czeroestimateseq2}, and~\eqref{Czeroestimateseq3}, we find
\begin{align} \label{Czerofinaleqtemp}
{|| \chi_s ||}_{C^0} & = \sup_{p \in \tilde M_s} |\chi_s (p)|_{g_s(p)} \leq C ( s^{\g \mu_i} + s^{-\nu'_i(1 - \g)} + s^{3(1-\g) + \delta \g} ), \\ \label{L14tempeq}
{|| \nab{g_s} \chi_s ||}_{C^0} & = \sup_{p \in \tilde M_s} | (\nab{g_s} \chi_s) (p)|_{g_s(p)} \leq C ( s^{\g \mu_i - \g} + s^{-\nu'_i(1 - \g) -\g} + s^{3(1-\g) + \delta \g - \g} ),
\end{align}
where in both cases the terms from~\eqref{Czeroestimateseq1a} and~\eqref{Czeroestimateseq1b} are absorbed by the last terms of~\eqref{Czeroestimateseq2}, using the fact that $r^{-4} \leq C s^{-4\g}$ in~\eqref{Czeroestimateseq1b}. The estimate~\eqref{Czerofinaleqtemp} is~\eqref{Czerofinaleq}.

We now move on to the $L^2$ norm of $\chi_s$. This is
\begin{equation*}
{|| \chi_s ||}_{L^2} = {\left( \int_{\tilde M_s} |\chi_s|^2_{g_s} \vol_{g_s} \right)}^{\frac{1}{2}}.
\end{equation*}
We will consider the contributions to the integral coming from the different regions in~\eqref{chiregionseq}. By~\eqref{Czeroestimateseq3}, there is no contribution from the outer regions. By Lemma~\ref{uniformequivlemma} and equation~\eqref{volceq}, in the overlap region $[s^{\g}, 2s^{\g}] \times \Sigma_i$ we have $\vol_{g_s} \leq C \volc = C r^6 dr \wedge \vols$. Now the estimate~\eqref{Czeroestimateseq2} gives
\begin{align*}
\int_{[s^{\g}, 2s^{\g}] \times \Sigma_i} |\chi_s|^2_{g_s} \vol_{g_s}  & \leq C \vol(\Sigma_i) \int_{s^{\g}}^{2s^{\g}} |\chi_s|^2_{g_s} r^6 dr \\ & \leq C ( s^{\g \mu_i} + s^{-\nu'_i(1 - \g)} + s^{3(1-\g) + \delta \g} )^2
\int_{s^{\g}}^{2s^{\g}} r^6 dr \\ & \leq C ( s^{\g \mu_i} + s^{-\nu'_i(1 - \g)} + s^{3(1-\g) + \delta \g} )^2 s^{7 \g},
\end{align*}
and hence the contribution to ${|| \chi_s ||}_{L^2}$ from the overlap region is bounded by
\begin{equation} \label{mainestimatestempeq2}
C ( s^{\g \mu_i + \frac{7}{2}\g} + s^{-\nu'_i(1 - \g) + \frac{7}{2}\g} + s^{3(1-\g) + \delta \g + \frac{7}{2}\g} ).
\end{equation}
On a subset $f_i ((2 s^{\g}, \e)  \times \Sigma_i)$ of the inner region, by Lemma~\ref{uniformequivlemma} we again have $\vol_{g_s} \leq C r^6 dr \wedge \vols$, and so by~\eqref{Czeroestimateseq1b} we see
\begin{align*}
\int_{(2 s^{\g}, \e)  \times \Sigma_i} |\chi_s|^2_{g_s} \vol_{g_s}  & \leq C \vol(\Sigma_i) \int_{2s^{\g}}^{\e} |\chi_s|^2_{g_s} r^6 dr \\ & \leq C s^8  \int_{2 s^{\g}}^{\e} r^{-8} r^6 dr \\ & \leq C ( s^8 + s^{8 - \g}) \leq C s^{8 - \g} ,
\end{align*}
and so the contribution to ${|| \chi_s ||}_{L^2}$ from this region is bounded by
\begin{equation} \label{mainestimatestempeq3}
C s^{4 - \frac{1}{2}\g} = C s^{4(1-\g) + \frac{7}{2}\g}.
\end{equation}
Finally the subset $K$ of the inner region is compact, so by~\eqref{Czeroestimateseq1a}, the contribution to ${|| \chi_s ||}_{L^2}$ on this region is bounded by $C s^4$, which is smaller than, and hence absorbed by,~\eqref{mainestimatestempeq3}. Putting~\eqref{mainestimatestempeq2} and~\eqref{mainestimatestempeq3} together, we find
\begin{equation*}
{|| \chi_s ||}_{L^2} \leq C ( s^{\g \mu_i + \frac{7}{2}\g} + s^{-\nu'_i(1 - \g) + \frac{7}{2}\g} + s^{3(1-\g) + \delta \g + \frac{7}{2}\g} +  s^{4(1-\g) + \frac{7}{2} \g} ).
\end{equation*}
It is easy to check that the fourth term above is smaller than, and absorbed by, the third term, exactly when
\begin{equation} \label{deltatempeq}
\delta < \frac{1 - \g}{\g}, 
\end{equation}
so by taking $\delta$ smaller if necessary, we ensure that~\eqref{L2finaleq} holds.

Lastly we need to consider the $L^{14}$ norm of $d^*_{g_s} \chi_s$. This is
\begin{equation*}
{|| d^*_{\g_s} \chi_s ||}_{L^{14}} = {\left( \int_{\tilde M_s} |d^*_{g_s} \chi_s|^{14}_{g_s} \vol_{g_s} \right)}^{\frac{1}{14}}.
\end{equation*}
By equation~\eqref{dstesteq}, we have ${| d^*_{g_s} \chi_s |}_{g_s} \leq C {| \nab{g_s}\chi_s |}_{g_s}$.
Now we proceed exactly as in the derivation of the $L^2$ estimates, using the pointwise estimates on $| \nab{g_s}\chi_s |$ given in~\eqref{Czeroestimateseq1a},~\eqref{Czeroestimateseq1b}, and~\eqref{Czeroestimateseq2}. We sketch the details. On the overlap regions, one obtains that the contribution to ${|| d^*_{\g_s} \chi_s ||}_{L^{14}}$ is bounded by
\begin{equation} \label{mainestimatestempeq4}
C ( s^{\g \mu_i - \g + \frac{1}{2}\g} + s^{-\nu'_i(1 - \g) - \g + \frac{1}{2}\g} + s^{3(1-\g) + \delta \g - \g + \frac{1}{2}\g} ),
\end{equation}
whereas on the inner region the contribution to ${|| d^*_{\g_s} \chi_s ||}_{L^{14}}$ is bounded by
\begin{equation*}
C s^{4 - \frac{7}{2}\g} = C s^{4(1-\g) + \frac{1}{2}\g}.
\end{equation*}
This term is absorbed by the third term in~\eqref{mainestimatestempeq4} whenever $\delta < \frac{1}{\g}$, which is automatic from~\eqref{deltatempeq}. Thus we obtain the estimate~\eqref{L14finaleq}.
\end{proof}

\begin{cor} \label{mainestimatescor}
There exists $\g \in (0,1)$ and $\kappa > 0$ such that, for $s$ sufficiently small, the closed $\G$~structure $\ph_s$ constructed in Definition~\ref{phspssdefn} satisfies part i) of Joyce's Theorem~\ref{joycethm}:
\begin{equation*}
\, \, {|| \chi_s ||}_{C^0} \leq D_1 s^{\kappa}, \qquad
{|| \chi_s ||}_{L^2} \leq D_1 s^{\frac{7}{2} + \kappa}, \qquad {|| d^*_{g_s} \chi_s ||}_{L^{14}} \leq D_1
s^{-\frac{1}{2} + \kappa}.
\end{equation*}
\end{cor}
\begin{proof}
By Theorem~\ref{mainestimatesthm}, the three estimates in part i) of Theorem~\ref{joycethm} will be satisfied if and only if
\begin{align} \label{ineq1}
\mu_i \g & \geq \kappa, & -\nu'_i (1 - \g ) & \geq \kappa, & 3( 1 - \g ) + \delta \g & \geq  \kappa, \\ \label{ineq2} \mu_i \g  + \frac{7}{2} \g & \geq \frac{7}{2} + \kappa, & -\nu'_i (1 - \g ) + \frac{7}{2} \g & \geq \frac{7}{2} + \kappa, & 3( 1 - \g ) + \delta \g + \frac{7}{2} \g & \geq \frac{7}{2} + \kappa, \\ \label{ineq3} \mu_i \g  - \frac{1}{2} \g & \geq -\frac{1}{2} + \kappa, & -\nu'_i (1 - \g ) - \frac{1}{2} \g & \geq -\frac{1}{2} + \kappa, & 3( 1 - \g ) + \delta \g - \frac{1}{2} \g & \geq -\frac{1}{2} + \kappa.
\end{align}
It is trivial to verify that for $\g < 1$, the inequalities in~\eqref{ineq2} automatically imply the inequalities in~\eqref{ineq1} and~\eqref{ineq3}. That is, the $L^2$ estimate implies the $C^0$ and $L^{14}$ estimates. The three inequalities in~\eqref{ineq2} can be rearranged to yield:
\begin{equation} \label{ineq4}
\g \geq \frac{\frac{7}{2} + \kappa}{\frac{7}{2} + \mu_i}, \qquad \qquad \g \geq 1 + \frac{\kappa}{\frac{7}{2} + \nu'_i}, \qquad \qquad \g \geq \frac{1 + 2\kappa}{1 + 2 \delta}.
\end{equation}
We want to ensure that there exists a $\kappa > 0$ such that if we define $\g$ by the three expressions above, then $\g$ will be in $(0,1)$. Recall that $\mu_i > 0$, $\delta > 0$, and $\nu'_i = -4$ (from Theorem~\ref{ACasymptoticexpansionthm}.) For the first inequality in~\eqref{ineq4}, we see that $\g > 0$ for any $\kappa > 0$, and $\g < 1$ if $\kappa < \mu_i$. For the third inequality in~\eqref{ineq4}, we have $\g > 0$ for any $\kappa > 0$, and $\g < 1$ if $\kappa < \delta$. Finally, for the middle inequality in~\eqref{ineq4}, we see that $\g < 1$ for any $\kappa > 0$, and $\g > 0$ if $\kappa < -(\frac{7}{2} + \nu'_i) = \frac{1}{2}$. Therefore it suffices to choose $\kappa$ satisfying
\begin{equation*}
\kappa < \delta, \qquad \kappa < \frac{1}{2}, \qquad \text{ and } \qquad \kappa < \mu_i, \qquad \text{for all } i = 1, \ldots, n. 
\end{equation*}
The intersection of the three inequalities in~\eqref{ineq4} with the regions $0 < \g <1$ and $\kappa > 0$ is represented by the shaded region in Figure~\ref{GRAPHfig}.
\begin{figure} [ht]
\centering
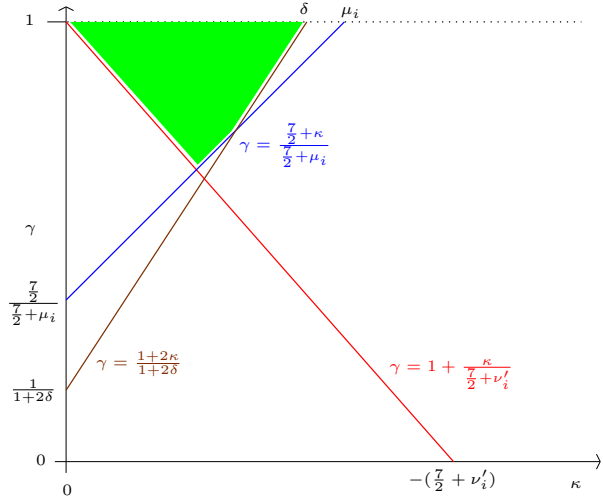
\caption{Intersection of the inequalities in the $\gamma$-$\kappa$ plane}
\label{GRAPHfig}
\end{figure}
\end{proof}
\begin{rmk} \label{criticalglueingratermk}
We have seen from~\eqref{ineq4} that there is a \emph{critical rate} $\nu_c = -\frac{7}{2}$ such that glueing in an AC $\G$~manifold with rate $\nu < \nu_c$ works, while for rate $\nu > \nu_c$ it \emph{does not} work. The determining factor is the size of the $L^2$ norm of the error caused by smoothing off the AC decay, which is contained in $\chi_s$. We were able to use Theorem~\ref{ACasymptoticexpansionthm} to show that for any initial AC rate $\nu \leq -3$, we can extract cohomological parts, in equations~\eqref{threeformasymptoticexpansioneq} and~\eqref{fourformasymptoticexpansioneq}, and then the left over parts have rate $\nu' = -4$.
\end{rmk}

Next we need to estimate the Riemann curvature and the injectivity radius of $g_s$.
\begin{prop} \label{RIRprop}
For $s$ sufficiently small, the metric $g_s$ induced by the closed $\G$~structure $\ph_s$ constructed in Definition~\ref{phspssdefn} satisfies parts ii) and iii) of Joyce's Theorem~\ref{joycethm}:
\begin{equation*}
\, \, \mathcal{I}(g_s) \geq D_2 s, \qquad
{|| \mathcal{R}(g_s) ||}_{C^0} \leq D_3 s^{-2}.
\end{equation*}
\end{prop}
\begin{proof}
We begin by noting that under a conformal scaling of metric $\tilde g = s^2 g$, the Riemann curvature tensor $\mathcal{R}$ and the injectivity radius $\mathcal{I}$ change by ${|| \mathcal{R}(\tilde g) ||}_{C^0(\tilde g)} = s^{-2} {|| \mathcal{R}(g) ||}_{C^0(g)}$ and $\mathcal{I}(\tilde g) = s \, \mathcal{I}(g)$.
We will consider the three regions of $\tilde M_s$ as given in Definition~\ref{phspssdefn}.

On the inner region $K \, \cup \, \sqcup_{i=1}^n f_i \left( (2s^{\g}, \e) \times \Sigma_i \right)$, we have $| g_s - \gm|_{\gm} \leq e$, for some $e < \frac{1}{2}$ by Lemma~\ref{uniformequivlemma}. Hence by Taylor's theorem we have $|\mathcal{R}(g_s) - \mathcal{R}(\gm)|_{\gm} \leq C e$. Therefore
\begin{equation*}
|\mathcal{R}(g_s)|_{g_s} \leq C |\mathcal{R}(g_s)|_{\gm} \leq C |\mathcal{R}(\gm)|_{\gm} + C|\mathcal{R}(g_s) - \mathcal{R}(\gm)|_{\gm} \leq C,
\end{equation*}
using~\eqref{unifeq1} and the fact that $\mathcal{R}(\gm)$ is smooth and hence bounded on any compact subset of $M$ containing the inner region. On the overlap regions $[s^{\g}, 2 s^{\g}] \times \Sigma_i$ we can again use the fact that $g_s$ is uniformly close to $\gci$ by Lemma~\ref{uniformequivlemma}, so as above we can say
\begin{equation*}
|\mathcal{R}(g_s)|_{g_s} \leq C |\mathcal{R}(\gci)|_{\gci} + C|\mathcal{R}(g_s) - \mathcal{R}(\gci)|_{\gci} \leq C |\mathcal{R}(\gci)|_{\gci} + C.
\end{equation*}
But since $\gci(tr, \sigma) = t^2 \gci(t, \sigma)$, we have $|\mathcal{R}(\gci)(tr, \sigma)|_{\gci(tr, \sigma)} = t^{-2} |\mathcal{R}(\gci)(r, \sigma)|_{\gci(r, \sigma)}$ on the cone $C_i$. Therefore since $r \in [s^{\g}, 2 s^{\g}]$ on these regions, we have $|\mathcal{R}(\gci)|_{\gci} \leq C s^{-2\g}$. Finally we consider the outer regions $L_i \, \cup \, h_{i,s} \left( (sR, s^{\g}) \times \Sigma_i \right)$. Here we have $g_s = s^2 \gni$ exactly, so $|\mathcal{R}(g_s)|_{g_s} = s^{-2} |\mathcal{R}(\gni)|_{\gni}| \leq C s^{-2}$ as $\mathcal{R}(\gni)$ is smooth and thus bounded on any compact subset of $N_i$ containing that outer region. Putting all three estimates together gives iii) of Theorem~\ref{joycethm}.

The estimate on the injectivity radius in ii) of Theorem~\ref{joycethm} is proved in essentially the same way. On the inner region, we get $\mathcal{I}(g_s) \geq C$, on the overlap regions we get $\mathcal{I}(g_s) \geq C s^{\g}$, and on the outer regions we get $\mathcal{I}(g_s) \geq C s$. Thus when $s$ is sufficiently small, putting these all together gives ii) of Theorem~\ref{joycethm}.

Note that in both cases, the dominant contribution comes from the conformally scaled metrics $g_s = s^2 \gni$ on the outer regions which were glued into $M$ to obtain $\tilde M_s$.
\end{proof}

Finally we can prove our main theorem.
\begin{thm} \label{mainthm}
Let $M$ be a compact $\G$~manifold with isolated conical singularities, with singularities $x_1, \ldots, x_n$, cones $C_1, \ldots, C_n$, and rates $\mu_1, \ldots, \mu_n$, respectively. Suppose that we have asymptotically conical $\G$~manifolds $N_1, \ldots, N_n$, with the same cones $C_1, \ldots, C_n$, and rates $\nu_1, \ldots, \nu_n$, respectively, with each $\nu_i \leq -3$. If the conditions~\eqref{obsconditionseq1} and~\eqref{obsconditionseq2} of Theorem~\ref{obsthm} are satisfied, then there exists a one-parameter family $\tilde M_s$ of \emph{smooth, compact $\G$~manifolds}, for $0 < s < \kappa$, (with holonomy exactly equal to $\G$), which desingularize $M$.
\end{thm}
\begin{proof}
From Corollary~\ref{mainestimatescor} and Proposition~\ref{RIRprop}, we see that the hypotheses of Theorem~\ref{joycethm} are satisfied with $\ph = \ph_s$ and $\chi = \chi_s$ on the smooth compact manifold $\tilde M_s$. Hence there exists a torsion-free $\G$~structure $\tilde \ph_s$ on $\tilde M_s$, for $0 < s < \kappa$. The Riemannian holonomy of the metric $\tilde g_s$ induced by $\tilde \ph_s$ must be exactly $\G$, since the holonomy of $M'$ and of the $N_i$'s was exactly $\G$, and the holonomy of a manifold obtained by glueing must be at least as large as the holonomy of its constituent pieces. 
\end{proof}

In fact we can say that  ``$\lim_{s\to0} \tilde M_s = M$'' in some sense, since as $s \to 0$, the asymptotically conical $\G$~manifolds $N_{i,s}$ shrink to points, corresponding to the original conical singularities $x_i$, and hence the manifold $\tilde M_s$ approaches the original $\G$~manifold $M$ with isolated conical singularities. Therefore compact $\G$~manifolds with isolated conical singularities can be thought of as possible `boundary points' in the moduli spaces of compact smooth $\G$~manifolds.

\section{Lockhart--McOwen analysis on manifolds with ends} \label{analsec}

In this section we summarize some aspects of the theory of Lockhart--McOwen analysis on non-compact manifolds. We will only need parts of the theory for weighted Sobolev spaces. It applies equally well to to weighted H\"older spaces. This theory originally appeared in Lockhart--McOwen~\cite{LM} and Lockhart~\cite{Lock}. A very detailed exposition can also be found in Marshall~\cite{M}, and a summary in the context of manifolds with ICS was first presented in Joyce~\cite{JSL1}. Note that these techniques have also been applied by Nordsr\"om~\cite{Nord} to study asymptotically \emph{cylindrical} $\G$~manifolds.

\subsection{Lockhart--McOwen analysis on manifolds with ICS} \label{lockhartCSsec}

Let $M$ be a compact $\G$~manifold with isolated conical singularities, as in Definition~\ref{CSdefn}. In order to be able to define sensible ``weighted'' Banach spaces on the non-compact smooth manifold $M'$, we need to introduce the concept of a radius function.
\begin{defn} \label{radiusfunctiondefn}
A \emph{radius function} $\varrho$ on $M'$ is a smooth function on $M'$ that satisfies the following conditions. On the compact subset $K = M' \backslash \sqcup_{i=1}^n S_i$, we define $\varrho \equiv 1$. Let $x$ be a point in the neighbourhood $f_i( (0, \frac{1}{2}\e) \times \Sigma_i)$ of the $i^{\text{th}}$ singularity $x_i$. Then $f_i^{-1} (x) = (r, \sigma)$ for some $r \in (0, \frac{1}{2}\e)$. We define $\varrho(x) = r$ for such a point. Finally, in the regions $f_i( (\frac{1}{2}\e, \e) \times \Sigma_i)$, the function $\varrho$ is defined by interpolating smoothly between its definitions near the singularities and its definition in the compact subset $K$, in an increasing fashion. It is clear that we can always construct such a function $\varrho$. Essentially, $\varrho$ is bounded from below by a positive constant $\frac{1}{2} \e$ away from the singularities, and near each singularity $\varrho$ basically measures the distance to the singularity.
\end{defn}
Let $\bl = (\lambda_1, \ldots, \lambda_n)$ be an $n$-tuple of real numbers. We can add such $n$-tuples and multiply them by real numbers using the vector space structure of $\R^n$. We also define $\bl + j = (\lambda_1 + j, \ldots, \lambda_n + j)$ for any $j \in \R$. Now define $\varrho^{\bl}$ to equal $\varrho^{\lambda_i}$ on $f_i( (0, \e) \times \Sigma_i)$ and to equal $1$ on $K$. Then $\varrho^{\bl}$ is a smooth function on $M'$ which equals
$r^{\lambda_i}$ on the neighbourhood $f_i( (0, \frac{1}{2}\e) \times \Sigma_i)$ of $x_i$. 

\begin{defn} \label{CSSobolevdefn}
Let $p > 1$, $l \geq 0$, and $\bl \in \R^n$. We define the \emph{weighted Sobolev space} $L^p_{l, \bl} (\Lambda^k (T^* M'))$ of $k$-forms on $M'$ as follows. Consider the space $C^{\infty}_{\text{cs}}(\Lambda^k(T^*M'))$ of smooth compactly supported $k$-forms on $M'$. For such forms the quantity
\begin{equation} \label{CSSobolevdefneq}
{||\omega||}_{L^p_{l,\bl}} \, = \, {\left( \sum_{j=0}^l \int_{M'} {| \varrho^{- \bl + j} \nabm^j \omega|}^p_{\gm} \varrho^{-7} \volm \right)}^{\frac{1}{p}}
\end{equation}
is clearly finite, and is a norm. We define the Banach space $L^p_{l, \bl} (\Lambda^k (T^* M'))$ to be the completion of $C^{\infty}_{\text{cs}}(\Lambda^k(T^*M'))$ with respect to this norm.
\end{defn}

\begin{rmk} \label{CSSobolevdefnrmk}
There are several observations to be made about this definition.
\begin{enumerate}[i)]
\item As a topological vector space, $L^p_{l, \bl} (\Lambda^k (T^* M'))$ is independent of the choice of radius function $\varrho$, and any two such choices lead to equivalent norms.
\item It is clear that $L^p_{l, \bl} (\Lambda^k (T^* M')) \subseteq L^p_{l, \bl'} (\Lambda^k (T^* M'))$
if $\lambda_i > \lambda_i'$ for all $i = 1, \ldots, n$.
\item The space $L^2_{l, \bl} (\Lambda^k (T^* M'))$ is a \emph{Hilbert space}, with inner product coming from the polarization of the norm in~\eqref{CSSobolevdefneq}.
\item An element $\omega$ in $L^p_{l, \bl} (\Lambda^k (T^* M'))$ can be intuitively thought of as a $k$-form which is $l$ times weakly differentiable, and such that near each $x_i$, the tensor $\nabm^j \omega$ is growing at most like $r^{\lambda_i - j}$. In fact if $|\omega|_{\gm} = O(r^{\lambda_i})$ near $x_i$, then $\omega \in L^p_{0, \bl - e} (\Lambda^k (T^* M'))$ for any $e > 0$.
\item Because of the factor of $\varrho^{-7}$ in~\eqref{CSSobolevdefneq}, for $l=0$, the space $L^p_{0, -\boldsymbol{\frac{7}{p}}} (\Lambda^k (T^* M'))$ is the usual $L^p (\Lambda^k (T^* M'))$ space,  and in particular we have
\begin{equation} \label{L2equiveq}
L^2_{0, - \boldsymbol{\frac{7}{2}}} (\Lambda^k (T^* M')) \, = \, L^2 (\Lambda^k (T^* M')).
\end{equation}
Here it is understood that $\boldsymbol{\frac{7}{p}}$ denotes the `constant' $n$-tuple $(\frac{7}{p}, \ldots, \frac{7}{p})$.
\end{enumerate}
\end{rmk}

\begin{prop} \label{dualspaceprop}
Let $q$ satisfy $\frac{1}{p} + \frac{1}{q} = 1$. There is a Banach space isomorphism
\begin{equation*}
{\left( L^p_{0, \bl} (\Lambda^k (T^* M')) \right)}^* \, \cong \, L^q_{0, -\bl - 7} (\Lambda^k (T^* M')),
\end{equation*}
given by the $L^2$ inner product pairing. That is, if $\alpha \in L^p_{0, \bl} (\Lambda^k (T^* M'))$ and $\beta \in L^q_{0, -\bl - 7} (\Lambda^k (T^* M'))$, then
\begin{equation*}
| {\langle \alpha , \beta \rangle}_{L^2}| \, \leq \, C || \alpha ||_{L^p_{0, \bl}} || \beta ||_{L^q_{0, -\bl - 7}}. 
\end{equation*}
\end{prop}

Now let $\Lambda^* (T^*M') = \sum_{k=0}^7 \Lambda^k (T^*M')$. Define $L^p_{l, \bl} (\Lambda^* (T^* M')) = \sum_{k=0}^7 L^p_{l, \bl} (\Lambda^k (T^* M'))$.
We will be interested in the following two differential operators:
\begin{equation} \label{diracmapdefneq}
(d + \dsm)^p_{l + 1, \bl} \, : \, L^p_{l + 1, \bl} (\Lambda^* (T^* M')) \, \to \, 
L^p_{l, \bl - 1} (\Lambda^* (T^* M')),
\end{equation}
and
\begin{equation} \label{lapmapdefneq}
(\lapm)^p_{l + 2, \bl} \, : \, L^p_{l + 2, \bl} (\Lambda^k (T^* M')) \, \to \, 
L^p_{l, \bl - 2} (\Lambda^k (T^* M')).
\end{equation}
They are defined by extending the operators $d + \dsm$ and $\lapm$ from smooth compactly supported forms to the Sobolev spaces.

\begin{defn} \label{criticalratesdefn}
Let $C$ be a cone. The set $\mathcal{D}_{d + \dsc}$ of \emph{critical rates} of the operator $d + \dsc$ on $\Lambda^* (T^*C)$ is defined as follows:
\begin{equation} \label{criticalratesdefneq}
\mathcal{D}_{d + \dsc} = \left\{ \begin{array}{l} \lambda \in \R; \, \, \exists \text{ a non-zero } \omega = \sum_{k=0}^7 \omega_k \in \Lambda^*(T^*C), \\ \text{ homogeneous of order $\lambda$, such that } (d + \dsc)(\omega) = 0 \end{array} \right\}.
\end{equation}
Similarly the set  $\mathcal{D}_{\lapc}$ of critical rates of $\lapc$ on $\Lambda^k (T^*C)$ is defined as
\begin{equation} \label{criticalratesdefneq2}
\mathcal{D}_{\lapc} = \left\{ \begin{array}{l} \lambda \in \R; \, \, \exists \text{ a non-zero } \omega \in \Lambda^k(T^*C), \\ \text{ homogeneous of order $\lambda$, such that } \lapc \omega = 0 \end{array} \right\}.
\end{equation}
Both $\mathcal D_{d + \dsc}$ and $\mathcal D_{\lapc}$ are countable, discrete subsets of $\R$.
\end{defn}

Our manifold $M$ with isolated conical singularities has $n$ singular points, with cones $C_1, \ldots, C_n$. Recall that a map between Banach spaces is \emph{Fredholm} if it has closed image, finite dimensional kernel, and finite dimensional cokernel. The relevance of the critical rates is that they are related to the rates $\bl$ for which the operators $(d + \dsm)^p_{l + 1, \bl}$ and $(\lapm)^p_{l+2, \bl}$ of~\eqref{diracmapdefneq} and~\eqref{lapmapdefneq} are Fredholm, by the following theorem.

\begin{thm} \label{fredholmthm}
The map $(d + \dsm)^p_{l + 1, \bl} \, : \, L^p_{l + 1, \bl} (\Lambda^* (T^* M')) \, \to \, L^p_{l, \bl - 1} (\Lambda^* (T^* M'))$ is Fredholm if and only if $\lambda_i \notin \mathcal D_{d + \dsci}$ for all $i = 1, \ldots, n$. Similarly the map $(\lapm)^p_{l + 2, \bl} \, : \, L^p_{l + 2, \bl} (\Lambda^k (T^* M')) \, \to \, L^p_{l, \bl - 2} (\Lambda^k (T^* M'))$ is Fredholm if and only if $\lambda_i \notin \mathcal D_{\lapci}$ for all $i = 1, \ldots, n$.
\end{thm}

It is a standard fact that the operators $d + \dsm$ and $\lapm$ are \emph{elliptic}. In fact they are also `uniformly elliptic' in the sense that near each $x_i$, they approach the elliptic operators $d + \dsci$ and $\lapci$ on $C_i$, respectively. This property allows us to prove an \emph{elliptic regularity} statement for these operators, which is given by the following theorem.

\begin{thm} \label{ellipticregthm}
Suppose that $\omega$ and $\upsilon$ are both locally integrable sections of $\Lambda^* (T^* M')$, and that $\omega$ is a weak solution of the equation $(d + \dsm)(\omega) = \upsilon$. If $\omega \in 
L^p_{0, \bl} (\Lambda^* (T^* M'))$ and $\upsilon \in L^p_{l, \bl - 1} (\Lambda^* (T^* M'))$, then $\omega \in L^p_{l + 1, \bl} (\Lambda^* (T^* M'))$, and $\omega$ is a strong solution of $(d + \dsm)(\omega) = \upsilon$. Furthermore, we have
\begin{equation} \label{ellipticregeq}
{||\omega||}_{L^p_{l+1,\bl}} \, \leq \, C \left( {||(d + \dsm)(\omega)||}_{L^p_{l,\bl -1}} + {||\omega||}_{L^p_{0,\bl}} \right)
\end{equation}
for some constant $C > 0$ independent of $\omega$. That is, $\omega$ has at least one more derivative's worth of regularity than $\upsilon = (d + \dsm)(\omega)$.

Similarly suppose $\omega$ and $\upsilon$ are locally integrable sections of $\Lambda^k (T^* M')$, and $\omega$ is a weak solution of the equation $\lapm \omega = \upsilon$. If $\omega \in 
L^p_{0, \bl} (\Lambda^k (T^* M'))$ and $\upsilon \in L^p_{l, \bl - 2} (\Lambda^k (T^* M'))$, then $\omega \in L^p_{l + 2, \bl} (\Lambda^k (T^* M'))$, and $\omega$ is a strong solution of $\lapm \omega = \upsilon$. Furthermore, we have
\begin{equation} \label{ellipticregeq2}
{||\omega||}_{L^p_{l+2,\bl}} \, \leq \, C \left( {||\lapm \omega||}_{L^p_{l,\bl -2}} + {||\omega||}_{L^p_{0,\bl}} \right)
\end{equation}
for some $C > 0$ independent of $\omega$.
\end{thm}

We also need the following very important result.
\begin{thm} \label{kernelthm}
The \emph{kernel} of $(d + \dsm)^p_{l + 1, \bl}$ is independent of $p > 1$ and independent of $l$. Hence we can denote it unambigiously as $\ker(d + \dsm)_{\bl}$. This kernel is also invariant as we change the rates $\bl$, as long as we do not hit any critical rates. That is, if $\bl = (\lambda_1, \ldots, \lambda_n)$ and $\bl' = (\lambda'_1, \ldots, \lambda'_n)$, with the interval $[\lambda_i, \lambda'_i]$ contained in $\R \backslash \mathcal D_{d+\dsci}$ for each $i$, then
\begin{equation*}
\ker ( d + \dsm)_{\bl'}  \, = \, \ker ( d + \dsm)_{\bl} .
\end{equation*}
Similarly the kernel of $(\lapm)^p_{l + 2, \bl}$ is independent of $p > 1$ and independent of $l$. Hence we can denote it unambigiously as $\ker(\lapm)_{\bl}$. This kernel is also invariant as we change the rates $\bl$, as long as we do not hit any critical rates. That is, if $\bl = (\lambda_1, \ldots, \lambda_n)$ and $\bl' = (\lambda'_1, \ldots, \lambda'_n)$, with the interval $[\lambda_i, \lambda'_i]$ contained in $\R \backslash \mathcal D_{\lapci}$ for each $i$, then
\begin{equation*}
\ker ( \lapm)_{\bl'}  \, = \, \ker ( \lapm)_{\bl} .
\end{equation*}
\end{thm}

\begin{rmk} \label{sobolevrmk}
There is a `Sobolev embedding theorem' in the context of manifolds with ICS, which (for large enough $p$ and $l$) embeds the Sobolev space $L^p_{l, \bl}$ into an appropriate \emph{H\"older space} $C^{k, \alpha}_{\bl'}$ having $k$ continuous derivatives, where $k$, $\alpha$, and $\bl'$ depend on $p$, $l$, and $\bl$. It follows from this theorem and the elliptic regularity of Theorem~\ref{ellipticregthm} that elements in the kernel of $d + \dsm$ or $\lapm$ are smooth, and the independence of the kernels on $l$ follows from this. In particular, it now follows that
\begin{equation}
\begin{aligned} \label{sobolevordereq}
& \omega \in \ker ( d + \dsm)_{\bl} \, \, \Longrightarrow \, \, | f_i^*(\omega) |_{f_i^*(\gm)} \, \leq \, C r^{\lambda_i} \quad \text{ on } \left( 0, \frac{1}{2} \e \right) \times \Sigma_i, \\
& \omega \in \ker ( \lapm)_{\bl} \, \, \Longrightarrow \, \, | f_i^*(\omega) |_{f_i^*(\gm)} \, \leq \, C r^{\lambda_i} \quad \text{ on } \left( 0, \frac{1}{2} \e \right) \times \Sigma_i.
\end{aligned}
\end{equation}
The independence of the kernels on $p > 1$ is more complicated, and can be found in Lockhart-McOwen~\cite[\S 7, \S 8]{LM}. We will not need the embedding theorem explicitly in this paper.
\end{rmk}

\begin{lemma} \label{IBPlemma}
Let $p > 1$. Suppose that $\omega \in L^p_{l+2, \bl} (\Lambda^k (T^* M'))$ and that $\lapm \omega = 0$. If $\lambda_i > -\frac{5}{2}$ for all $i = 1, \ldots, n$, then $\dsm \omega = 0$ and $d \omega = 0$.
\end{lemma}
\begin{proof}
By Theorem~\ref{kernelthm}, if $\omega \in \ker (\lapm)^p_{l+2, \bl}$, then $\omega \in \ker (\lapm)^2_{l+2, \bl}$. Now the elliptic regularity of Theorem~\ref{ellipticregthm} tells us that $\omega \in L^2_{l+2, \bl} (\Lambda^k (T^* M'))$, and hence $d \omega$ and $\dsm \omega$ are both elements of $L^2_{l+1, \bl - 1} (\Lambda^* (T^* M'))$. Since each $\lambda_i > - \frac{5}{2}$, parts ii) and v) of Remark~\ref{CSSobolevdefnrmk} show that $d\omega$ and $\dsm \omega$ are in $L^2_{l+1, -\boldsymbol{\frac{7}{2}}} (\Lambda^* (T^* M')) \subseteq L^2 (\Lambda^* (T^* M'))$. Therefore the following computation is valid:
\begin{equation*}
0 \, = \, {\langle \lapm \omega, \omega \rangle}_{L^2} =  {\langle d \dsm \omega, \omega \rangle}_{L^2} +  {\langle \dsm d \omega, \omega \rangle}_{L^2} = {| \dsm \omega|}^2_{L^2} + {| d \omega|}^2_{L^2},
\end{equation*}
and so we have $\dsm \omega = 0$ and $d\omega = 0$.
\end{proof}

\begin{cor} \label{IBPcor}
Suppose that $\omega_k \in L^2_{l + 2, \bl} (\Lambda^k (T^* M'))$, and that $\lapm \omega_k = 0$. Then we have:
\begin{align} \label{k07lapcceq}
& \text{For } k = 0, 7, \, \text{ if each } \, \lambda_i > -5, \, \text{ then } \dsm \omega = 0 \text{ and } d\omega = 0. &  \\ \label{k16lapcceq}
& \text{For } k = 1, 6, \, \text{ if each } \, \lambda_i > -4, \, \text{ then } \dsm \omega = 0 \text{ and } d\omega = 0. &  \\  \label{k25lapcceq}
& \text{For } k = 2, 5, \, \text{ if each } \, \lambda_i > -3, \, \text{ then } \dsm \omega = 0 \text{ and } d\omega = 0. &
\end{align}
\end{cor}
\begin{proof}
By Proposition~\ref{homolapexcludeprop} we see that in all three cases, there are no critical rates until at least $\lambda_i = -2$. So using Theorem~\ref{kernelthm}, in all cases we can say that $\omega_k \in \ker (\lapm)_{- \boldsymbol{ \frac{5}{2}} + e}$ for some small $e > 0$. The claims now all follow from Lemma~\ref{IBPlemma}.
\end{proof}

The following proposition (in a general setting) originally appeared in Lockhart~\cite[Example 0.16]{Lock} and a version in the setting of manifolds with ICS is stated in Lotay~\cite[Theorem 6.5.2]{Lth}.
\begin{prop} \label{LMcohomologyprop}
Let $\mathcal H^k_{L^2}$ denote the subspace of $L^2 (\Lambda^k (T^* M'))$ consisting of closed and coclosed $k$-forms. Then for $0 \leq k \leq 3$, the map
\begin{align*}
[ \cdot] : & \mathcal H^k_{L^2} \to H^k(M', \R) \\ & \omega \mapsto [\omega]
\end{align*}
is an isomorphism. Here $[\omega]$ denotes the cohomology class of the closed form $\omega$.
\end{prop}

The rest of the results in this section also all have analogues for $\lapm$, but we will only require them for the operator $d + \dsm$. First we will need to consider the adjoint of the map
\begin{equation*}
(d + \dsm)^p_{l + 1, \bl} \, : \, L^p_{l + 1, \bl} (\Lambda^* (T^* M')) \, \to \, L^p_{l, \bl - 1} (\Lambda^* (T^* M')).
\end{equation*} 
By Proposition~\ref{dualspaceprop}, the adjoint is a map
\begin{equation} \label{adjointmapeq}
(d + \dsm)^q_{m+1, -\bl - 6} \, \, : \, \, L^q_{m+1, -\bl - 6} ( \Lambda^* (T^* M' )) \to L^q_{m , -\bl - 7} ( \Lambda^* (T^* M' )),
\end{equation}
where $\frac{1}{p} + \frac{1}{q} = 1$, and $l, m \geq 0$. Here we are being slightly sloppy, in the following sense. Technically, we really have ${(L^p_{l, \bl})}^* = L^q_{-l, - \bl - 7}$, but we would like to avoid having to consider the meaning of $L^p_{l, \bl}$ for $l < 0$. However, we will only ever be interested in the \emph{kernel} of the adjoint on spaces of the form $L^q_{m+1, \bl}$, which by Theorem~\ref{kernelthm} are independent of $m$, so it is safe to assume that $m \geq 0$.

The next result is the version of the `Fredholm Alternative' for manifolds with ICS.
\begin{thm} \label{fredholmalternativethm}
Suppose that $\lambda_i$ is not in $\mathcal D_{d + \dsci}$ for all $i = 1, \ldots, n$, so that by Theorem~\ref{fredholmthm}, the map $(d + \dsm)^p_{l + 1, \bl} \, : \, L^p_{l + 1, \bl} (\Lambda^* (T^* M')) \, \to \, L^p_{l, \bl - 1} (\Lambda^* (T^* M'))$ is Fredholm, and also uniformly elliptic. Then the \emph{image} $\im (d + \dsm)^p_{l + 1, \bl}$ of this map is given by
\begin{equation*}
\im (d + \dsm)^p_{l + 1, \bl} \, = \, \left\{ \upsilon \in L^p_{l, \bl-1} (\Lambda^* (T^* M')) \, ; \, \langle \upsilon, \omega \rangle_{L^2} = 0 \, \text{ for all } \omega \in \ker (d + \dsm)_{-\bl - 6} \right\}.
\end{equation*}
That is, the equation $(d + \dsm)(\omega) = \upsilon$ is solvable for $\omega \in  L^p_{l + 1, \bl} (\Lambda^* (T^* M'))$ if and only if the right hand side $\upsilon$ is orthogonal (with respect to the $L^2$ inner product) to every element of $\ker (d + \dsm)_{-\bl - 6}$, the kernel of the adjoint map.
\end{thm}

We will need to understand how, near each singular point $x_i$, we can asymptotically expand elements of the kernel of $(d + \dsm)^p_{l + 1, \bl}$ in terms of certain special solutions of $(d + \dsci)(\omega) = 0$.
\begin{defn} \label{homowithlogsdefn}
Let $C$ be a cone. For $\lambda \in \R$, we define the space $\mathcal K (\lambda)_{d + \dsc}$ to be
\begin{equation} \label{homowithlogseq}
\mathcal K(\lambda)_{d + \dsc} \, : = \, \left\{ \begin{array}{l} \omega = \sum_{k=0}^7 \sum_{j=0}^m (\log(r))^j \omega_{k,j} ; \, \text{ such that } (d + \dsc)(\omega) = 0, \\ \text{ and each } \omega_{k,j} \text{ is a homogeneous $k$-form of order $\lambda$} \end{array} \right\}.
\end{equation}
That is, $\mathcal K(\lambda)_{d + \dsc}$ is the set mixed-degree forms on $C$ in the kernel of $d + \dsc$, which are polynomials in $\log(r)$ with coefficients being homogeneous forms of order $\lambda$.
\end{defn}

We will use $\ker_i  (d + \dsm)_{\bl}$ to denote the restriction of the kernel of $(d + \dsm)^p_{l + 1, \bl}$ to  the subset $f_i ( (0, \frac{1}{2}\e) \times \Sigma_i)$ of the $i^{\text{th}}$ end of $M'$ (recall that by Theorem~\ref{kernelthm} this kernel is independent of $p > 1$ and $l \geq 0$.) The following crucial result appears in Lockhart--McOwen~\cite[\S 5]{LM}.
\begin{prop} \label{asymptoticexpansionprop}
Let $\boldsymbol{\beta}, \boldsymbol{\beta}'$ be two $n$-tuples of rates. Fix $i$ and suppose that $\beta_i, \beta_i' \in \R \backslash \mathcal D_{d + \dsci}$, and that $\beta_i < \beta_i'$. Let $\alpha_1 < \alpha_2 < \ldots \alpha_N$ be all the critical rates in $\mathcal D_{d + \dsci}$ between $\beta_i$ and $\beta_i'$. If $\omega \in \ker_i  (d + \dsm)_{\boldsymbol{\beta}}$, then there exist $\upsilon_j \in \mathcal K(\alpha_j)_{d + \dsci}$, for $j = 1, \ldots, N$, and an $\omega'$ defined near $x_i$ with $|\omega'|_{\gci} = O(r^{\beta_i + \mu_i})$ such that
\begin{equation*}
\omega - \sum_{j=1}^N (f_i^{-1})^* (\upsilon_j) - \omega' \, \in \ker_i  (d + \dsm)_{\boldsymbol{\beta}'}.
\end{equation*}
That is, when restricted to the subset $f_i ( (0, \frac{1}{2}\e) \times \Sigma_i)$ of the $i^{\text{th}}$ end of $M'$, an element in the kernel of $d + \dsm$ with non-critical rate $\beta_i$ admits an expansion in terms of elements in $\mathcal K(\alpha_j)_{d + \dsci}$ for each critical rate $\alpha_j$ between $\beta_i$ and $\beta_i'$, plus a remainder term which is of order $O(r^{\beta_i + \mu_i})$ and another term which is in the kernel of $d + \dsm$ with non-critical rate $\beta_i' > \beta_i$.
\end{prop}
\begin{rmk} \label{asymptoticexpansionproprmk}
The $O(r^{\beta_i + \mu_i})$ term arises from comparing a solution to $(d + \dsm)(\omega) = 0$ to a solution of $(d + \dsci)(\omega) = 0$, using the relation~\eqref{CSdefneq2} between the metrics $\gci$ and $\gm$ near $x_i$.
\end{rmk}

We will use Proposition~\ref{asymptoticexpansionprop} to prove the following.
\begin{cor} \label{asymptoticexpansioncor}
Let $\bl = (\lambda, \ldots, \lambda)$ be a constant $n$-tuple. Choose $\delta > 0$ small enough so that in the closed interval $[\lambda - \delta, \lambda + \delta]$, at most only $\lambda$ itself is a critical rate for each end. That is, $[\lambda - \delta, \lambda + \delta] \cap \mathcal D_{d + \dsci} \subseteq \{\lambda\}$ for each $i$. This is possible since each $\mathcal D_i$ is a discrete subset of $\R$. Further assume that $\delta > 0$ is small enough so that $\delta < \frac{1}{2}\mu_i$ for each $i$, where $\mu_i$ are the rates of the singularities of the manifold $M$ with ICS. If $\omega \in \Lambda^* (T^* M')$ is in the kernel of $(d + \dsm)^p_{l+1, \lambda - \delta}$, then
\begin{equation} \label{asymptoticexpansioncoreq}
| f_i^* ( \omega ) -  \upsilon_i |_{\gci} \, = \,  O( r^{\lambda + \delta}) \quad \text{ on } \left(0, \frac{1}{2}\e \right) \times \Sigma_i, 
\end{equation}
for some $\upsilon_i$ in $\mathcal K(\lambda)_{d + \dsci}$, which may be zero (if $\lambda \notin \mathcal D_{d + \dsci}$.)
\end{cor}
\begin{proof}
We apply Proposition~\ref{asymptoticexpansionprop}, with $\beta_i = \lambda - \delta$ and $\beta_i' = \lambda + \delta$. Because of the hypotheses on $\delta$, for some $\upsilon_i \in \mathcal K(\lambda)_{d + \dsci}$, on $f_i ((0, \frac{1}{2}\e) \times \Sigma_i)$ we have that $\omega - (f_i^{-1})^*(\upsilon_i) - \omega'$ lies in $\ker_i(d + \dsm)_{\lambda + \delta}$, where $|\omega'| = O(r^{\lambda - \delta + \mu_i})$. Hence by~\eqref{sobolevordereq}, we have
\begin{equation*}
|f_i^*(\omega) - \upsilon_i|_{f_i^*(\gm)} \, \leq \, C r^{\lambda + \delta} + C r^{\lambda - \delta + \mu_i}\quad \text{ on } \left(0, \frac{1}{2}\e \right) \times \Sigma_i.
\end{equation*}
Since by hypothesis we have $\delta < \frac{1}{2} \mu_i$, and $r < 1$, the second term on the right is smaller than, and can be absorbed by, the first term. 
\end{proof}

There are many more results that are known in this setting. For example, the \emph{index} of of $d + \dsm$ will stay constant (like the kernel) unless we cross a critical rate, and when we do, the change in the index can be explicitly computed using the dimensions of the spaces $\mathcal K (\lambda)_{d + \dsci}$. We will not need these results here. The interested reader is referred to Marshall~\cite{M} for a comprehensive treatment.

\subsection{Lockhart--McOwen analysis on AC manifolds} \label{lockhartACsec}

Let $N$ be an asymptotically conical $\G$~manifold, as in Definition~\ref{ACdefn}. All of the analytic results about manifolds with ICS also hold for AC manifolds. The only differences are that there is only one end this time, and since $r \to \infty$ instead of $r \to 0$ at the end, all the inequalities involving rates are reversed. We will require only a few of these results in the AC setting, which we now summarize.

\begin{defn} \label{radiusfunctiondefn2}
A \emph{radius function} $\varrho$ on $N$ is a smooth function on $N$ that satisfies the following conditions. On the compact subset $L$ of $N$, we define $\varrho \equiv 1$. Let $x$ be a point in $h( (2 R, \infty) \times \Sigma)$. Then $h^{-1} (x) = (r, \sigma)$ for some $r \in (2R, \infty)$. We define $\varrho(x) = r$ for such a point. Finally, in the region $h( (R, 2R) \times \Sigma)$, the function $\varrho$ is defined by interpolating smoothly between its definition near infinity and its definition in the compact subset $L$, in a decreasing fashion.
\end{defn}

\begin{defn} \label{ACSobolevdefn}
Let $p > 1$, $l \geq 0$, and $\lambda \in \R$. We define the \emph{weighted Sobolev space} $L^p_{l, \lambda} (\Lambda^k (T^* N))$ of $k$-forms on $N$ as follows. Consider the space $C^{\infty}_{\text{cs}}(\Lambda^k(T^*N))$ of smooth compactly supported $k$-forms on $N$. For such forms the quantity
\begin{equation} \label{ACSobolevdefneq}
{||\omega||}_{L^p_{l,\lambda}} \, = \, {\left( \sum_{j=0}^l \int_{N} {| \varrho^{- \lambda + j} \nabn^j \omega|}^p_{\gn} \varrho^{-7} \voln \right)}^{\frac{1}{p}}
\end{equation}
is a norm, and we let $L^p_{l, \lambda} (\Lambda^k (T^* N))$ be the completion of $C^{\infty}_{\text{cs}}(\Lambda^k(T^*N))$ with respect to this norm.
\end{defn}

\begin{rmk} \label{ACSobolevdefnrmk}
There  are a few differences in the AC setting.
\begin{enumerate}[i)]
\item This time we have $L^p_{l, \lambda} (\Lambda^k (T^* N)) \subseteq L^p_{l, \lambda'} (\Lambda^k (T^* N))$ if $\lambda < \lambda'$. (The inequality is in the opposite direction to the ICS case.)
\item An element $\omega$ in $L^p_{l, \lambda} (\Lambda^k (T^* N))$ can be thought of as a $k$-form which is $l$ times weakly differentiable, and such that near infinity, the tensor $\nabn^j \omega$ is growing at most like $r^{\lambda - j}$. In fact if $|\omega|_{\gn} = O(r^{\lambda})$ near infinity, then $\omega \in L^p_{0, \lambda + e} (\Lambda^k (T^* N))$ for any $e > 0$.
\item The space $L^p_{0, -\frac{7}{p}} (\Lambda^k (T^* N))$ is the usual $L^p (\Lambda^k (T^* N))$ space,  and in particular we have
\begin{equation} \label{L2equiveq2}
L^2_{0, - \frac{7}{2}} (\Lambda^k (T^* N)) \, = \, L^2 (\Lambda^k (T^* N)).
\end{equation}
\end{enumerate}
\end{rmk}

\begin{prop} \label{dualspaceprop2}
Let $q$ satisfy $\frac{1}{p} + \frac{1}{q} = 1$. There is a Banach space isomorphism
\begin{equation*}
{\left( L^p_{0, \lambda} (\Lambda^k (T^* N)) \right)}^* \, \cong \, L^q_{0, -\lambda - 7} (\Lambda^k (T^* N)),
\end{equation*}
given by the $L^2$ inner product pairing.
\end{prop}

For AC manifolds we will only be interested in the operator:
\begin{equation} \label{diracmapdefneq2}
(d + \dsn)^p_{l + 1, \lambda} \, : \, L^p_{l + 1, \lambda} (\Lambda^* (T^* N)) \, \to \, 
L^p_{l, \lambda - 1} (\Lambda^* (T^* N)).
\end{equation}
The AC analogues of Theorems~\ref{fredholmthm},~\ref{ellipticregthm}, and~\ref{kernelthm} are the following.
\begin{thm} \label{fredholmthm2}
The map $(d + \dsn)^p_{l + 1, \lambda} \, : \, L^p_{l + 1, \lambda} (\Lambda^* (T^* N)) \, \to \, L^p_{l, \lambda - 1} (\Lambda^* (T^* N))$ is Fredholm if and only if $\lambda \notin \mathcal D_{d + \dsc}$, where the set of critical rates $\mathcal D_{d + \dsc}$ is as given in Definition~\ref{criticalratesdefn}.
\end{thm}
\begin{thm} \label{ellipticregthm2}
Suppose that $\omega$ and $\upsilon$ are both locally integrable sections of $\Lambda^* (T^* N)$, and that $\omega$ is a weak solution of the equation $(d + \dsn)(\omega) = \upsilon$. If $\omega \in 
L^p_{0, \lambda} (\Lambda^* (T^* N))$ and $\upsilon \in L^p_{l, \lambda - 1} (\Lambda^* (T^* N))$, then $\omega \in L^p_{l + 1, \lambda} (\Lambda^* (T^* N))$, and $\omega$ is a strong solution of $(d + \dsn)(\omega) = \upsilon$. Furthermore, we have
\begin{equation} \label{ellipticregeq3}
{||\omega||}_{L^p_{l+1,\lambda}} \, \leq \, C \left( {||(d + \dsn)(\omega)||}_{L^p_{l,\lambda -1}} + {||\omega||}_{L^p_{0,\lambda}} \right)
\end{equation}
for some constant $C > 0$ independent of $\omega$.
\end{thm}
\begin{thm} \label{kernelthm2}
The \emph{kernel} of $(d + \dsn)^p_{l + 1, \lambda}$ is independent of $p > 1$ and independent of $l$. Hence we can denote it unambigiously as $\ker(d + \dsn)_{\lambda}$. This kernel is also invariant as we change the rate $\lambda$, as long as we do not hit any critical rates. That is, if the interval $[\lambda, \lambda']$ is contained in $\R \backslash \mathcal D_{d+\dsc}$, then
\begin{equation*}
\ker ( d + \dsn)_{\lambda'}  \, = \, \ker ( d + \dsn)_{\lambda} .
\end{equation*}
\end{thm}
\begin{rmk} \label{sobolevrmk2}
As in Remark~\ref{sobolevrmk}, it follows that
\begin{equation} \label{sobolevordereq2}
\omega \in \ker ( d + \dsn)_{\lambda} \, \, \Longrightarrow \, \, | h^*(\omega) |_{h^*(\gn)} \, \leq \, C r^{\lambda} \quad \text{ on } \left( 2R, \infty \right) \times \Sigma.
\end{equation}
\end{rmk}

We will need to consider the adjoint of the map
\begin{equation*}
(d + \dsn)^p_{l + 1, \lambda} \, : \, L^p_{l + 1, \lambda} (\Lambda^* (T^* N)) \, \to \, L^p_{l, \lambda - 1} (\Lambda^* (T^* N)).
\end{equation*} 
By Proposition~\ref{dualspaceprop2}, the adjoint is a map
\begin{equation} \label{adjointmapeq2}
(d + \dsn)^q_{m+1, -\lambda - 6} \, \, : \, \, L^q_{m+1, -\lambda - 6} ( \Lambda^* (T^* N)) \to L^q_{m , -\lambda - 7} ( \Lambda^* (T^* N )),
\end{equation}
where $\frac{1}{p} + \frac{1}{q} = 1$, and $l, m \geq 0$. Here we make the same comments about $l, m \geq 0$ as we did following equation~\eqref{adjointmapeq}.

Next we have the `Fredholm Alternative' for AC manifolds.
\begin{thm} \label{fredholmalternativethm2}
Suppose that $\lambda$ is not in $\mathcal D_{d + \dsc}$, so that by Theorem~\ref{fredholmthm2}, the map $(d + \dsn)^p_{l + 1, \lambda} \, : \, L^p_{l + 1, \lambda} (\Lambda^* (T^* N)) \, \to \, L^p_{l, \lambda - 1} (\Lambda^* (T^* N))$ is Fredholm, and also uniformly elliptic. Then the \emph{image} $\im (d + \dsn)^p_{l + 1, \lambda}$ of this map is given by
\begin{equation*}
\im (d + \dsn)^p_{l + 1, \lambda} \, = \, \left\{ \upsilon \in L^p_{l, \lambda-1} (\Lambda^* (T^* N)) \, ; \, \langle \upsilon, \omega \rangle_{L^2} = 0 \, \text{ for all } \omega \in \ker (d + \dsn)_{-\lambda - 6} \right\}.
\end{equation*}
\end{thm}

Let $\mathcal K_{d + \dsc}$ be as given in Definition~\ref{homowithlogsdefn}. We will use $\ker_{\infty}  (d + \dsn)_{\lambda}$ to denote the restriction of the kernel of $(d + \dsn)^p_{l + 1, \lambda}$ to the subset $h ( (2R, \infty) \times \Sigma)$ of the asymptotic end of $N$. The following is the AC analogue of Proposition~\ref{asymptoticexpansionprop}.
\begin{prop} \label{asymptoticexpansionprop2}
Suppose that $\beta, \beta' \in \R \backslash \mathcal D_{d + \dsc}$, and that $\beta > \beta'$. Let $\alpha_1 < \alpha_2 < \ldots \alpha_N$ be all the critical rates in $\mathcal D_{d + \dsc}$ between $\beta$ and $\beta'$. If $\omega \in \ker_{\infty}  (d + \dsn)_{\beta}$, then there exist $\upsilon_j \in \mathcal K(\alpha_j)_{d + \dsc}$, for $j = 1, \ldots, N$, and an $\omega'$ defined near infinity with $|\omega'|_{\gc} = O(r^{\beta + \nu})$ such that
\begin{equation*}
\omega - \sum_{j=1}^N (h^{-1})^* (\upsilon_j) - \omega' \, \in \ker_{\infty}  (d + \dsn)_{\beta'}.
\end{equation*}
That is, when restricted to the the subset $h ( (2R, \infty) \times \Sigma)$ of the asymptotic end of $N$, an element in the kernel of $d + \dsn$ with non-critical rate $\beta$ admits an expansion in terms of elements in $\mathcal K(\alpha_j)_{d + \dsc}$ for each critical rate $\alpha_j$ between $\beta$ and $\beta'$, plus a remainder term which is of order $O(r^{\beta + \nu})$ and another term which is in the kernel of $d + \dsn$ with non-critical rate $\beta' < \beta$.
\end{prop}
\begin{rmk} \label{asymptoticexpansionproprmk2}
The $O(r^{\beta + \nu})$ term arises from comparing a solution to $(d + \dsn)(\omega) = 0$ to a solution of $(d + \dsc)(\omega) = 0$, using the relation~\eqref{ACdefneq2} between the metrics $\gc$ and $\gn$ near infinity.
\end{rmk}

From Proposition~\ref{asymptoticexpansionprop2} we can prove the following, exactly as in the proof of Corollary~\ref{asymptoticexpansioncor}.
\begin{cor} \label{asymptoticexpansioncor2}
Suppose that $\lambda_1 < \lambda_2$. Choose $\delta > 0$ small enough so that in the closed interval $[\lambda_1 - \delta, \lambda_2 + \delta]$, the only critical rates for the asymptotic end lie in the closed interval $[\lambda_1, \lambda_2]$. That is, $[\lambda_1 - \delta, \lambda_2 + \delta] \cap \mathcal D_{d + \dsc} \subseteq [\lambda_1, \lambda_2]$. Further assume that $\delta > 0$ is small enough so that $\delta < \frac{1}{2}(\lambda_1 - \lambda_2 - \nu)$, where $\nu$ is the rate of the AC manifold $N$. (This will only be possible when $\nu < \lambda_1 - \lambda_2$.) If $\omega \in \Lambda^* (T^* N)$ is in $\ker_{\infty}(d + \dsn)_{\lambda_2 + \delta}$, then
\begin{equation} \label{asymptoticexpansioncoreq2}
| h^* ( \omega ) -  \sum_{j=1}^N \upsilon_j |_{\gc} \, = \,  O( r^{\lambda_1 - \delta}) \quad \text{ on } \left(2R, \infty \right) \times \Sigma, 
\end{equation}
for some $\upsilon_j$ in $\mathcal K(\alpha_j)_{d + \dsc}$, where $\alpha_1, \ldots, \alpha_N$ are all the critical rates for $d + \dsc$ in $[\lambda_1, \lambda_2]$.
\end{cor}

\section{Solution of the obstruction problem} \label{obstructionsec}

In this section we prove Theorem~\ref{obsthm}, which gives the existence of $\eta$ and $\xi$ on $M'$ with the required properties to perform the desingularization, and which was used in Section~\ref{formsconstructionsec}. We restate the theorem here for the convenience of the reader.

\noindent {\bf Theorem~\ref{obsthm}}
\emph{
Let $M$ be a compact $\G$~manifold with isolated conical singularities. For each $i$, let $\xi_i$ and $\eta_i$ be $3$-forms and $4$-forms, homogeneous of orders $-3$ and $-4$, respectively, and closed and coclosed on each cone $C_i$, with respect to $\gci$. Suppose that
\begin{align} \label{obsconditionseq1again}
& \oplus_{i=1}^n [\xi_i] \in \oplus_{i=1}^n H^3(\Sigma_i, \R) \text{ lies in the image of } \Upsilon^3 : H^3(M', \R) \to \oplus_{i=1}^n H^3(\Sigma_i, \R), \\  \label{obsconditionseq2again} & \oplus_{i=1}^n [\eta_i] \in \oplus_{i=1}^n H^4(\Sigma_i, \R) \text{ lies in the image of } \Upsilon^4 : H^4(M', \R) \to \oplus_{i=1}^n H^4(\Sigma_i, \R).
\end{align}
where the maps $\Upsilon^k$ are given in Definition~\ref{topdefn}. Then for $\delta > 0$ sufficiently small, there exists a smooth $3$-form $\xi$ and a smooth $4$-form $\eta$ on $M'$ such that
\begin{align} \nonumber
& d \xi = 0, \qquad d^*_{\gm} \xi = 0, \qquad d \eta = 0, \\
\label{obseq1again} & | \nabci^j \! ( f_i^* (\xi) - \xi_i ) |_{\gci} = \, \, O (r^{-3 + \delta - j})
\qquad \forall j \geq 0, \\ \label{obseq2again} & | \nabci^j \! ( f_i^* (\eta) - \eta_i ) |_{\gci} = \, \, O (r^{-4 + \delta - j}) \qquad \forall j \geq 0.
\end{align}
Notice that we do not say that $\eta$ is coclosed. Furthermore, the $3$-form $\xi$ is in $\Lambda^3_{27}$ with respect to the $\G$~structure $\phm$.}

Notice that if~\eqref{obseq1again} and~\eqref{obseq2again} hold for $j = 0$, then since $\xi$, $\eta$, $\xi_i$, and $\eta_i$ are all closed and $\delta >0$, Lemma~\ref{exactformslemma} shows that the cohomology classes $[f_i^* (\xi) - \xi_i]$ and $[f_i^* (\eta) - \eta_i]$ vanish in $H^3(\Sigma_i, \R)$ and $H^4(\Sigma_i, \R)$, respectively. But by the definition of the maps $\Upsilon^k$ in Definition~\ref{topdefn}, this precisely says that $\Upsilon^3([\xi]) = \oplus_{i=1}^n [\xi_i]$ and $\Upsilon^4([\eta]) = \oplus_{i=1}^n [\eta_i]$. Hence the conditions in~\eqref{obsconditionseq1again} and~\eqref{obsconditionseq2again} are \emph{necessary} for closed forms $\xi$ and $\eta$ to exist satisfying~\eqref{obseq1again} and~\eqref{obseq2again}.  These are therefore \emph{global topological conditions} for the desingularization to be possible, which relate the different singular points. We will see in Sections~\ref{obssubsection} and~\ref{obssubsection2} that the conditions in~\eqref{obsconditionseq1again} and~\eqref{obsconditionseq2again} are also \emph{sufficient} to construct a $\xi$ and $\eta$ satisfying all the conclusions of Theorem~\ref{obsthm}.

\subsection{Construction of the $3$-form correction $\xi$} \label{obssubsection}

Let $\xi'$ be any smooth $3$-form on $M'$ satisfying
\begin{equation} \label{xiprimeeq}
| \nabci^j \! ( f_i^* (\xi') - \xi_i ) |_{\gci} = \, \, O (r^{-3 + \delta - j}) \qquad \forall j \geq 0.
\end{equation}
This can clearly be done by defining $f_i^*(\xi')$ to be equal to $\xi_i$ in a neighbourhood $(0, \frac{1}{2}\e) \times \Sigma_i$ of $x_i$, defining $\xi'$ to be zero in the compact core $K$ of $M'$, and by smoothly interpolating between these definitions on the annuli $[\frac{1}{2}\e, \e] \times \Sigma_i$. 
The goal of this section is to modify $\xi'$ to produce a $3$-form $\xi$ with the properties given in Theorem~\ref{obsthm}.

Let $\Lambda^{\text{odd}} (T^* M') )$ and $\Lambda^{\text{even}} (T^* M') )$ denote the space of odd and even degree forms on $M'$, respectively. Since the operator $d + \dsm$ interchanges the odd and even forms, all the results about $d + \dsm$ from Section~\ref{lockhartCSsec} remain true when we consider $d + \dsm$ as mapping just from odd to even forms or conversely. We need to consider the map
\begin{equation} \label{threeformdiracmapeq}
P : = (d + \dsm)^{\text{odd}, p}_{l+1, -\boldsymbol{3} + \delta} \, \, : \, \, L^p_{l+1, -\boldsymbol{3} + \delta} ( \Lambda^{\text{odd}} (T^* M') ) \to L^p_{l , -\boldsymbol{4} + \delta} ( \Lambda^{\text{even}} (T^* M') ).
\end{equation}
By~\eqref{adjointmapeq}, the adjoint $P^*$ of $P$ is the map
\begin{equation} \label{threeformdiracmapdualeq}
P^* : = (d + \dsm)^{\text{even}, q}_{m+1, -\boldsymbol{3} - \delta} \, \, : \, \, L^q_{m+1, -\boldsymbol{3} - \delta} ( \Lambda^{\text{even}} (T^* M') ) \to L^q_{m , -\boldsymbol{4} - \delta} ( \Lambda^{\text{odd}} (T^* M') ).
\end{equation}
First, we have to describe the kernel of $P^*$ explicitly.
\begin{prop} \label{adjointkernelprop}
If $\upsilon = \sum_{k=0}^3 \upsilon_{2k} \in L^q_{m+1, -\boldsymbol{3} - \delta} ( \Lambda^{\text{even}} (T^* M') )$ satisfies $P^* (\upsilon) = 0$, then on the neighbourhood $(0, \frac{1}{2} \e) \times \Sigma_i$ of $x_i$, we have
\begin{equation} \label{adjointkerneleq}
f_i^*(\upsilon) = dr \wedge \alpha_{4,i} + O(r^{-3 + \delta}),
\end{equation}
where each $\alpha_{4,i}$ is a harmonic $3$-form on the link $\Sigma_i$, respectively. Furthermore, each component $\upsilon_{2k}$ of $\upsilon$ is closed and coclosed with respect to the metric $\gm$ on $M'$.
\end{prop}
\begin{proof}
The fact that $f_i^*(\upsilon)$ takes the form shown in~\eqref{adjointkerneleq} follows from Corollary~\ref{asymptoticexpansioncor} (with $\lambda = -3$), and Proposition~\ref{dirachomokernellogsprop}.
It remains to show that each $\upsilon_{2k}$ is closed and coclosed.  The equation $(d + \dsm)(\upsilon) = 0$ breaks up into
\begin{equation} \label{adjointkernelproptempeq}
d \upsilon_0 + \dsm \upsilon_2 = 0, \qquad d \upsilon_2 + \dsm \upsilon_4 = 0, \qquad d \upsilon_4 + \dsm \upsilon_6 = 0, \qquad d \upsilon_6  = 0.
\end{equation}
Also, $(d + \dsm)(\upsilon) = 0$ implies that $\lapm \upsilon_{2k} = 0$ for all $k = 0, \ldots, 3$. Since $- 3 - \delta > -4$, we see that $\upsilon_0$ and $\upsilon_6$ are both closed and coclosed, from~\eqref{k07lapcceq} and~\eqref{k16lapcceq}, respectively. Therefore~\eqref{adjointkernelproptempeq} simplifies to
\begin{equation} \label{adjointkernelproptempeq2}
\dsm \upsilon_2 = 0, \qquad d \upsilon_2 + \dsm\upsilon_4 = 0, \qquad d \upsilon_4 = 0.
\end{equation}
Thus we will be done if we can show that $d \upsilon_2 = 0$. But $d \upsilon_2$ is a closed $3$-form, and by the middle equation in~\eqref{adjointkernelproptempeq2}, it is also coclosed. Now by Proposition~\ref{LMcohomologyprop} for $k=3$, the map $d \upsilon_2 \mapsto [d \upsilon_2] \in H^3(M', \R)$ is injective. But $[d \upsilon_2] = 0$, so $d \upsilon_2 = 0$ and the proof is complete.
\end{proof}

\begin{thm} \label{solvabilitythm}
Let $\xi'$ be a smooth $3$-form on $M'$, satisfying~\eqref{xiprimeeq}, and further suppose that the condition~\eqref{obsconditionseq1again} is satisfied. Then there exists $\omega = \sum_{k=0}^3 \omega_{2k + 1} \in  L^p_{l+1, -\boldsymbol{3} + \delta} (\Lambda^{\text{odd}} (T^* M' ))$ satisfying the equation
\begin{equation} \label{threeformsolvabilityeq}
(d + \dsm)(\omega) = - \dsm \xi' - d \xi'.
\end{equation}
Furthermore, using this solution, if we define $\tilde \xi = \xi' + \omega_3$, then $\tilde \xi \in L^p_{l+1, -\boldsymbol{3}}(\Lambda^3 (T^* M'))$ and satisfies $d\tilde \xi = 0$ and $\dsm \tilde \xi = 0$.
\end{thm}
\begin{proof}
By Theorem~\ref{fredholmalternativethm}, we need to check that the right hand side of the equation is orthogonal (with respect to the $L^2$ inner product) to the kernel of the adjoint map. Let $\upsilon = \sum_{k=0}^3 \upsilon_{2k}$ be in the kernel of $P^*$. Then
\begin{equation} \label{solvabilitythmtempeq2}
{\langle - \dsm \xi' - d \xi', \upsilon_0 + \upsilon_2 + \upsilon_4 + \upsilon_6 \rangle}_{L^2} \, = \, \int_{M'} {\langle - \dsm \xi', \upsilon_2 \rangle}_{\gm} \volm +  \int_{M'} {\langle - d \xi', \upsilon_4 \rangle}_{\gm} \volm.
\end{equation}
On $3$-forms we have $\dsm = - \stm d \stm$, and $\stm$ is an isometry with respect to $\gm$, so therefore we can write ${\langle - \dsm \xi', \upsilon_2 \rangle}_{\gm} = {\langle d \stm \xi', \stm \upsilon_2 \rangle}_{\gm}$. Hence equation~\eqref{solvabilitythmtempeq2} simplifies to
\begin{equation} \label{solvabilitythmtempeq3}
\begin{aligned} 
{\langle - \dsm \xi' - d \xi', \upsilon \rangle}_{L^2} \,& = \, \int_{M'} ( d \stm \xi' \wedge \upsilon_2 ) - \int_{M'} (d \xi' \wedge \stm \upsilon_4) \\  & = \, \int_{M'} d (\stm \xi' \wedge \upsilon_2) - \int_{M'} d (\xi' \wedge \stm \upsilon_4),
\end{aligned}
\end{equation}
where we have used the fact that $d \upsilon_2 = 0$ and $d \stm \upsilon_4 = 0$ from Proposition~\ref{adjointkernelprop}. Let us define $M'_r = \{ x \in M'; \varrho(x) \geq r \}$, where $\varrho$ is the radius function on $M'$ from Definition~\ref{radiusfunctiondefn}. Then $M'_r$ is a manifold with boundary, and $\partial M'_r = \sqcup_{i=1}^n f_i (\{r\} \times \Sigma_i)$. Then using Stokes' Theorem, equation~\eqref{solvabilitythmtempeq3} becomes
\begin{equation}  \label{solvabilitythmtempeq4}
\begin{aligned}
{\langle - \dsm \xi' - d \xi', \upsilon \rangle}_{L^2} \, = & \, \lim_{r \to 0} \, \, \sum_{i=1}^n \int_{\{r\} \times \Sigma_i} f_i^*(\stm \xi') \wedge f_i^*(\upsilon_2) \\ & \, \, {} - \lim_{r \to 0} \, \,\sum_{i = 1}^n \int_{\{r\} \times \Sigma_i} f_i^*(\xi') \wedge f_i^*(\stm \upsilon_4).
\end{aligned}
\end{equation}
By~\eqref{xiprimeeq} and the fact that $| \xi_i |_{\gci} = O(r^{-3})$, we have $|f_i^*(\xi')|_{\gci} = O(r^{-3})$.
The proof of~\eqref{Etempeq} also works for $\xi'$, since only the property~\eqref{xiprimeeq} is used. Therefore we also have $|f_i^*(\stm \xi')|_{\gci} = O(r^{-3})$. Proposition~\ref{adjointkernelprop} tells us that $|f_i^* (\upsilon_2)|_{\gci} = O(r^{-3 + \delta})$, and $|f_i^* (\upsilon_4) - dr \wedge \alpha_{4,i}|_{\gci} = O(r^{-3 + \delta})$ on $\{r\} \times \Sigma_i$, where $\alpha_{4,i}$ is a harmonic $3$-form on $\Sigma_i$. Now we can mimic the proof of~\eqref{Etempeq} using $\upsilon_4$ and $dr \wedge \alpha_{4,i}$ in place of $\xi$ and $\xi_i$, respectively, to obtain that $|f_i^* (\stm \upsilon_4) - \stci( dr \wedge \alpha_{4,i})|_{\gci} = O(r^{-3 + \delta})$. Using~\eqref{conestareq} we see that $\stci (dr \wedge \alpha_{4,i}) = \stsi \alpha_{4,i}$. In summary, we have:
\begin{equation} \label{solvabilityestimateseq}
\begin{aligned}
|f_i^*(\stm \xi')|_{\gci} & = O(r^{-3}), \qquad \qquad &  |f_i^* (\upsilon_2)|_{\gci} & = O(r^{-3 + \delta}), & \\  |f_i^*(\xi') - \xi_i|_{\gci} & = O(r^{-3+ \delta}), \qquad \qquad & |f_i^* (\stm \upsilon_4) - \stsi \alpha_{4,i} |_{\gci} & = O(r^{-3 + \delta}), & \\  |f_i^*(\xi')|_{\gci} & = O(r^{-3}), \qquad \qquad & |\stsi \alpha_{4,i} |_{\gci} & = O(r^{-3}).
\end{aligned}
\end{equation}
If we write $f_i^*(\xi') = \xi_i + ( f_i^*(\xi') - \xi_i)$ and $f_i^*(\stm \upsilon_4) = \stsi \alpha_{4,i} + ( f_i^*(\stm \upsilon_4) - \stsi \alpha_{4,i} )$, we can then apply the estimates in~\eqref{solvabilityestimateseq} to equation~\eqref{solvabilitythmtempeq4}, to obtain
\begin{align} \nonumber
|{\langle - \dsm \xi' - d \xi', \upsilon \rangle}_{L^2}| \, & \leq \, \lim_{r \to 0} \, \, \sum_{i=1}^n \int_{\{r\} \times \Sigma_i} C r^{-6 + \delta} \volsi + \left| \lim_{r \to 0} \, \,\sum_{i = 1}^n \int_{\{r\} \times \Sigma_i} \xi_i \wedge \stsi \alpha_{4,i} \right|. \\  \label{solvabilitythmtempeq5} & \leq \, \lim_{r \to 0} \, \,C r^{\delta} + \left| \lim_{r \to 0} \, \,\sum_{i = 1}^n \int_{\{r\} \times \Sigma_i} \xi_i \wedge \stsi \alpha_{4,i} \right|.
\end{align}
using the fact that the volume of $\{r\} \times \Sigma_i$ is equal to $C r^6$, where $C$ is the volume of $\{1\} \times \Sigma_i$. The first limit in~\eqref{solvabilitythmtempeq5} is zero, so to conclude the solvability of equation~\eqref{threeformsolvabilityeq} it remains to prove that
\begin{equation} \label{solvabilityconditiontempeq}
\lim_{r \to 0} \, \,\sum_{i = 1}^n \int_{\{r\} \times \Sigma_i} \xi_i \wedge \stsi \alpha_{4,i} \, \, = \, \, 0.
\end{equation}
This is where we will need to use the hypothesis that the condition~\eqref{obsconditionseq1again} is satisfied. This condition says that there exists a smooth \emph{closed} $3$-form $\xi''$ on $M'$ such that $f_i^* (\xi'')$ differs from $\xi_i$ on $(0, \e) \times \Sigma_i$ by an exact piece. That is,
\begin{equation} \label{solvabilityconditiontempeq2}
\xi_i = f_i^*(\xi'') + d V_i \qquad \text{ on } (0, \e) \times \Sigma_i
\end{equation}
for some smooth $2$-forms $V_i$ defined on $(0, \e) \times \Sigma_i$. We also know that $\stsi \alpha_{4,i}$ is closed (since $\alpha_{4,i}$ is harmonic on $\Sigma_i$), and also that $f_i^*(\upsilon_4)$ is closed by Proposition~\ref{adjointkernelprop}. Therefore the fourth equation in~\eqref{solvabilityestimateseq} and Lemma~\ref{exactformslemma} together imply that $f_i^* (\stm \upsilon_4)$ differs from $\stsi \alpha_{4,i}$ on $(0, \e) \times \Sigma_i$ by an exact piece. That is,
\begin{equation} \label{solvabilityconditiontempeq3}
\stsi \alpha_{4,i} = f_i^*(\stm \upsilon_4) + d W_i \qquad \text{ on } (0, \e) \times \Sigma_i
\end{equation}
for some smooth $2$-forms $W_i$ defined on $(0, \e) \times \Sigma_i$. Using~\eqref{solvabilityconditiontempeq2} and~\eqref{solvabilityconditiontempeq3}, the left hand side of equation~\eqref{solvabilityconditiontempeq} now becomes
\begin{equation*}
\lim_{r \to 0} \left( \, \,\sum_{i = 1}^n \int_{\{r\} \times \Sigma_i} f_i^*(\xi'') \wedge f_i^* (\stm \upsilon_4) 
+ \sum_{i = 1}^n \int_{\{r\} \times \Sigma_i} (dV_i \wedge \stsi \alpha_{4,i} + \xi_i \wedge dW_i + dV_i \wedge dW_i ) \right).
\end{equation*}
Since $\xi_i$ and $\stsi \alpha_{4,i}$ are closed, the integrands in the second sum of integrals above are all \emph{exact}, and hence these integrals all vanish by Stokes' Theorem, because $\{r\} \times \Sigma_i$ is compact and without boundary. Therefore the left hand side of equation~\eqref{solvabilityconditiontempeq} has been simplified to
\begin{equation} \label{solvabilityconditiontempeq6}
\begin{aligned}
\lim_{r \to 0} \, \, \sum_{i = 1}^n \int_{f_i(\{r\} \times \Sigma_i)} \xi'' \wedge \stm \upsilon_4 \, & = \, \lim_{r \to 0} \, \, \int_{\partial M'_r} \xi'' \wedge \stm \upsilon_4 \\ & = \, \lim_{r \to 0} \, \, \int_{M'_r} d(\xi'' \wedge \stm \upsilon_4) \, = \, \int_{M'} d(\xi'' \wedge \stm \upsilon_4),
\end{aligned}
\end{equation}
using Stokes' Theorem again. But this integral vanishes since both $\xi''$ and $\stm \upsilon_4$ are closed, by Proposition~\ref{adjointkernelprop} and the hypothesis~\eqref{obsconditionseq1again} which ensures the existence of such a closed $\xi''$. Hence the right hand side of~\eqref{threeformsolvabilityeq} is $L^2$-orthogonal to the kernel of $P^*$, and therefore there exists a solution $\omega = \sum_{k=0}^3 \omega_{2k+1}$ to equation~\eqref{threeformsolvabilityeq}.
It still remains to show that $\tilde \xi = \xi' + \omega_3$ is closed and coclosed.
The equation $(d + \dsm)(\omega) = -\dsm \xi' - d\xi'$ breaks up into
\begin{equation} \label{solvabilitythmtempeq}
\dsm \omega_1 = 0, \qquad d \omega_1 + \dsm \omega_3 = - \dsm \xi', \qquad d \omega_3 + \dsm \omega_5 = - d \xi', \qquad d \omega_5 + \dsm \omega_7  = 0.
\end{equation}
Now $(d + \dsm)(\omega) = 0$ implies that $\lapm \omega_{2k+1} = 0$ for all $k = 0, \ldots, 3$. Since $- 3 + \delta > -3$, we see that $\omega_1$, $\omega_5$, and $\omega_7$ are all closed and coclosed, from~\eqref{k16lapcceq},~\eqref{k25lapcceq}, and~\eqref{k07lapcceq}, respectively. Therefore~\eqref{solvabilitythmtempeq} simplifies to
\begin{equation*}
\dsm( \omega_3 + \xi') = 0, \qquad d( \omega_3 + \xi') = 0.
\end{equation*}
Also, $\xi' = \xi_i$ near $x_i$, and $|\xi_i|_{\gci} = O(r^{-3})$. Therefore we have
\begin{equation} \label{xitildeeq}
| \nabci^j \! ( f_i^* (\tilde \xi) - \xi_i ) |_{\gci} = \, \, O (r^{-3 + \delta - j}) \qquad \forall j \geq 0.
\end{equation}
and hence $\tilde \xi = \omega_3 + \xi'$ is in $L^p_{l+1, -\boldsymbol{3}}(\Lambda^3 (T^* M'))$ and is closed and coclosed.
\end{proof}

We have constructed a smooth $3$-form $\tilde \xi$ on $M'$ which is closed and coclosed, and satisfies~\eqref{obseq1again}. The final step is to modify $\tilde \xi$ to a $3$-form $\xi$ which lies in $\wtht$ with respect to $\phm$.
\begin{prop} \label{projectionprop}
Let $\tilde \xi$ be as given in Theorem~\ref{solvabilitythm}. Denote by $\tilde \xi = \tilde \xi_1 + \tilde \xi_7 + \tilde \xi_{27}$ the decomposition of $\tilde \xi$ into components of the subspaces $\Omega^3 = \wtho \oplus \wths \oplus \wtht$ determined by $\phm$. Then the $3$-form $\xi = \tilde \xi_{27}$ is closed and coclosed and satisfies~\eqref{obseq1again}.
\end{prop}
\begin{proof}
The decomposition of $\Omega^3$ is orthogonal with respect to $\gm$ and preserved by the covariant derivative (since $\phm$ is torsion-free), so each $\tilde \xi_1$, $\tilde \xi_7$, and $\tilde \xi_{27}$ is in $L^p_{l + 2, -\mathbf{3} + \delta} (\Lambda^3 (T^* M'))$. Since $d \tilde \xi = 0$ and $\dsm \tilde \xi = 0$, we have $\lapm \tilde \xi = 0$. Again because $\phm$ is torsion-free, the Laplacian commutes with the projections onto the $\G$~invariant subspaces, so $\lapm \tilde \xi_1 = 0$. But $\tilde \xi_1 = f \phm$ for some function $f$ in $L^p_{l + 2, -\mathbf{3} + \delta}( \Lambda^0 (T^* M'))$, and the first equation in~\eqref{temp27eq1} says that $\lapm f = 0$. Hence by~\eqref{k07lapcceq}
we see that $f$ is closed (hence constant), and therefore $\tilde \xi_1 = f \phm$ is closed and coclosed.

Similarly $\lapm \tilde \xi_7 = 0$, and since $\tilde \xi_7 = \stm (\alpha \wedge \phm)$ for some $1$-form $\alpha$ in $L^p_{l+2, -\mathbf{3} + \delta}( \Lambda^1 (T^* M'))$, the first equation in~\eqref{temp27eq2} says that $\lapm \alpha = 0$. Proposition~\ref{homolapexcludeprop} says there are no critical rates for the Laplacian on $1$-forms in the interval $(-4, -1)$, so by Theorem~\ref{kernelthm}, we can say that $\alpha$ is in $\ker (\lapm)_{-\mathbf{1} - e}$ for some small $e > 0$, and then Theorem~\ref{ellipticregthm} tells us that $\alpha \in L^p_{l+2, - \mathbf{1} - e} (\Lambda^1 (T^* M'))$. Therefore $\nabm{} \alpha$ lies in $L^p_{l+1, - \mathbf{2} - e} ( T^*M' \otimes T^* M' )$, which is in $L^2$ because $- 2 - e > - \frac{7}{2}$. (Here we are using an $L^p_{l, \bl}$ norm on sections of the tensor bundle $T^* M' \otimes T^* M'$, and the analogous statement to v) from Remark~\ref{CSSobolevdefnrmk}.) On a manifold with vanishing Ricci curvature, the usual Bochner--Weitzenb\"ock formula reduces to $\Delta = \nab{}^* \nab{}$ on $1$-forms. Therefore we can compute that:
\begin{equation*}
0 \, = \, {\langle \lapm \alpha, \alpha \rangle}_{L^2} =  {\langle \nabm{}^{\! \! \! *} \, \nabm{} \alpha, \alpha \rangle}_{L^2} = {| \nabm{} \alpha|}^2_{L^2},
\end{equation*}
and thus we see that $\nabm{} \alpha = 0$. But a manifold with holonomy exactly equal to $\G$ has no non-zero parallel $1$-forms (Bryant--Salamon~\cite[Theorem 2]{BS}), so $\alpha = 0$, and hence $\tilde \xi_7 = 0$.

Thus we have $\tilde \xi_{27} = \tilde \xi - \tilde \xi_1$, and since $\tilde \xi$ and $\tilde \xi_1$ are both closed and coclosed, we see that $\xi = \tilde \xi_{27}$ is also closed and coclosed. We still need to show that $\xi$ satisfies~\eqref{obseq1again}. Let $\pi_{27}^M$ denote the projection onto the $\wtht$ subspace of $\Omega^3 (T^* M')$ with respect to $\phm$, and $\pi_{27}^{C_i}$ denote the projection onto the $\wtht$ subspace of $\Omega^3 (T^* C_i)$ with respect to $\phci$. Because of~\eqref{CSdefneq}, the $\G$~structures $f_i^*(\phm)$ and $\phci$ on $(0, \e) \times \Sigma_i$ agree up to order $O(r^{\mu_i})$, and thus we have
\begin{equation*}
| f_i^*( \pi_{27}^M \omega) - \pi_{27}^{C_i} f_i^*(\omega) |_{\gci} \, = \, | f_i^*(\omega) |_{\gci} O(r^{\mu_i})
\end{equation*}
for any $3$-form $\omega$ on $M'$. Since $\xi = \pi_{27}^M \tilde \xi$, the above equation becomes
\begin{equation*}
| f_i^*(\xi) - \pi_{27}^{C_i} f_i^*(\tilde \xi) |_{\gci} \, = \, | f_i^*(\tilde \xi) |_{\gci} O(r^{\mu_i}) \, = \, O(r^{-3 + \mu_i}).
\end{equation*}
But then we have
\begin{align*}
| f_i^*(\xi) - \xi_i |_{\gci} \, & \leq \, | f_i^*(\xi) - \pi_{27}^{C_i} f_i^*(\tilde \xi) |_{\gci}  + | \pi_{27}^{C_i} f_i^*(\tilde \xi)  - \xi_i|_{\gci} \\ & = \, O(r^{-3 + \mu_i}) +  | \pi_{27}^{C_i} (f_i^*(\tilde \xi)  - \xi_i )|_{\gci} \\ & \leq \, 
O(r^{-3 + \mu_i}) + O(r^{-3 + \delta}) = O(r^{-3 + \delta}),
\end{align*}
where we have used the fact that $ \pi_{27}^{C_i} \xi_i = \xi_i$ from Proposition~\ref{cones27prop}, equation~\eqref{xitildeeq}, and the fact that $\delta < \mu_i$. Similarly we can show that
\begin{equation*}
| \nabci^j \! ( f_i^* (\xi) - \xi_i ) |_{\gci} = \, \, O (r^{-3 + \delta - j}) \qquad \forall j \geq 0,
\end{equation*}
and the proof is complete.
\end{proof}

We have succeeded in constructing the $3$-form $\xi$ of Theorem~\ref{obsthm} needed to correct both $\phm$ and $\psm$ for the glueing of Section~\ref{torsionfreesec} to work. We still need to construct the $4$-form correction $\eta$.

\subsection{Construction of the $4$-form correction $\eta$} \label{obssubsection2}

We can try to follow the same procedure as in Section~\ref{obssubsection} to construct the $4$-form correction $\eta.$ However, in this case one encounters some difficulties. The main problem is that the analogue of Proposition~\ref{dirachomokernellogsprop} fails. The relevant rate in this case is $-2$, and it turns out that there \emph{can} be $\log(r)$ terms in the kernel of $d + \dc$ with order $-2$ on odd forms. In fact it turns out that the $\log$ terms arise only for the $3$-form components of the kernel. Because of this, the proof of an exact $4$-form analogue to Theorem~\ref{solvabilitythm} would break down, but one can carry out this procedure partially to produce a $4$-form $\eta$ which is merely closed, but {\em not} coclosed. The arguments become much more involved, however. And then the analgous argument to Proposition~\ref{projectionprop} for projection onto the $\wfot$ component would also break down.

However, we saw in Section~\ref{torsionfreesec} that it was unnecessary for $\eta$ to be coclosed or to be in $\wfot$, just that it be closed. Since that and condition~\eqref{obseq2again} are our only requirements on $\eta$, there is a much simpler argument that avoids all the technical Lockhart--McOwen analysis of Section~\ref{obssubsection}.

\begin{prop} \label{etaprop}
Suppose that condition~\eqref{obsconditionseq2again} is satisfied. Then there exists a smooth \emph{closed} $4$-form $\eta$ on $M'$ satisfying
\begin{equation*}
 | \nabci^j \! ( f_i^* (\eta) - \eta_i ) |_{\gci} = \, \, O (r^{-4 + \delta - j}) \qquad \forall j \geq 0,
\end{equation*}
on $(0, \e) \times \Sigma_i$, for any $\delta > 0$.
\end{prop}
\begin{proof}
As in the discussion before equation~\eqref{solvabilityconditiontempeq2}, the fact that condition~\eqref{obsconditionseq2again} is satisfied says that there exists a smooth \emph{closed} $4$-form $\eta''$ on $M'$ such that $f_i^* (\eta'')$ differs from $\eta_i$ on $(0, \e) \times \Sigma_i$ by an exact piece. That is,
\begin{equation} \label{etaproptempeq}
\eta_i = f_i^*(\eta'') + d U_i \qquad \text{ on } (0, \e) \times \Sigma_i
\end{equation}
for some smooth $3$-forms $U_i$ defined on $(0, \e) \times \Sigma_i$. Let $w: (0, \infty) \to \R$ be any smooth decreasing function such that
\begin{equation*}
w(r) = \begin{cases} 1 & \text{for } 0 < r \leq \frac{4}{6}\e, \\0 & \text{for } \frac{5}{6}\e \leq r < \infty. \end{cases}
\end{equation*}
Now we define the smooth $4$-form $\eta$ on $M'$ as follows:
\begin{equation} \label{etapropdefneq}
\eta = \begin{cases} \eta'' & \text{on the compact core $K$ of $M'$}, \\ \eta'' + (f_i^{-1})^*(d ( wU_i)) & \text{on } f_i ( [\frac{1}{2}\e, \e] \times \Sigma_i), \\ (f_i^{-1})^* (\eta_i) & \text{on } f_i( (0, \frac{1}{2}\e] \times \Sigma_i). \end{cases}
\end{equation}
It is clear that $\eta$ is closed, and by equation~\eqref{etaproptempeq} and the definition of $w$ it is also smooth and well-defined. Also, on $f_i( (0, \frac{1}{2} \e) \times \Sigma_i)$, we have $f_i^*(\eta) = \eta_i$, so
\begin{equation*}
| \nabci^j \! ( f_i^* (\eta) - \eta_i ) |_{\gci} = \, \, O (r^{-4 + \delta - j}) \qquad \forall j \geq 0,
\end{equation*}
for any $\delta > 0$, since it is identically equal to zero for $r < \frac{1}{2} \e$.
\end{proof}

We have succeeded in constructing the $4$-form $\eta$ of Theorem~\ref{obsthm} needed to correct $\psm$ for the glueing of Section~\ref{torsionfreesec} to work. Thus the proof of Theorem~\ref{obsthm} is now complete.

\section{Asymptotic expansion of $\G$~structures on AC $\G$~manifolds} \label{ACgaugefixsec}

In this section we prove Theorem~\ref{ACasymptoticexpansionthm}, which says that under an appropriate \emph{gauge-fixing condition} we can obtain a nice asymptotic expansion of the $\G$~structure $(\phn, \psn)$ on an AC $\G$~manifold $N$. This theorem was used in Section~\ref{formsconstructionsec} for the glueing construction. We restate the theorem here for the convenience of the reader, and drop the $i$ subscripts to simplify notation.

\noindent {\bf Theorem~\ref{ACasymptoticexpansionthm}}
\emph{
Suppose that $N$ is an asymptotically conical $\G$~manifold with rate $\nu \leq -3$, and that $h$ satisfies the gauge-fixing condition given in Definition~\ref{gaugefixdefn}. Then on the subset $(2R, \infty) \times \Sigma$ of the cone $C$ we can write
\begin{align} \label{threeformasymptoticexpansioneqagain}
h^*(\phn) \, & = \, \phc + \xi + d\zeta,
\\ \label{fourformasymptoticexpansioneqagain}
h^*(\psn) \, & = \, \psc + \eta - \stc \xi + d\theta.
\end{align}
where $\xi$ is a harmonic $3$-form, homogeneous of order $-3$, and in $\wtht$ with respect to $\phc$, $\eta$ is a harmonic $4$-form, homogeneous of order $-4$, and $\zeta$ and $\theta$ are $2$-forms and $3$-forms on $(2R, \infty) \times \Sigma$, respectively, satisfying
\begin{equation} \label{ACinterpeq3newagain}
| \nabc^j \zeta |_{\gc} = \, \, O(r^{\nu' + 1 - j}), \qquad \quad | \nabc^j \theta |_{\gc} = \, \, O(r^{\nu' + 1 - j}), \qquad \forall j \geq 0,
\end{equation}
where $\nu' = - 4$. Furthermore, $[\xi] = \Phi(N)$ and $[\eta] = \Psi(N)$, where $\Phi(N)$ and $\Psi(N)$ are the cohomological invariants of the AC $\G$~manifold $N$ from Definition~\ref{ACinvariantsdefn}.
}

We will begin by showing that the gauge-fixing condition leads to an elliptic equation. Let $N$ be an asymptotically conical $\G$~manifold with rate $\nu \leq -3$. Then we have that
\begin{equation*}
|h^*(\phn) - \phc|_{\gc} \, = \, O (r^{\nu}),
\end{equation*}
by equation~\eqref{ACdefneq}. Since $\nu$ can equal $-3$, the invariant $\Phi(N)$ in $H^3(\Sigma, \R)$ of Definition~\ref{ACinvariantsdefn} can be non-zero. By Proposition~\ref{formsrepresentprop}, we can represent the class $\Phi(N)$ by a $3$-form $\xi$ on $(R, \infty) \times \Sigma$ which is homogeneous of order $-3$ and harmonic with respect to the cone metric. Furthermore, Proposition~\ref{cones27prop} says that $\xi$ is in $\Lambda^3_{27}$ with respect to $\phc$. Now $h^*(\phn) - \phc - \xi$ is exact, so we can write
\begin{equation*}
h^*(\phn) - \phc = \xi + d\zeta
\end{equation*}
for some $2$-form $\zeta$ on $(R, \infty) \times \Sigma$, with $|d\zeta|_{\gc} = O(r^{-3})$.

\begin{lemma} \label{ACgaugefixlemma}
Suppose that $N$ is an AC $\G$~manifold with rate $\nu \leq -3$, and that $h$ satisfies the gauge-fixing condition of Definition~\ref{gaugefixdefn}. Then we have
\begin{equation} \label{ACgaugefixlemmaeq}
|(d + \dsc) (h^*(\psn) - \psc)|_{\gc} \, = \, O(r^{-7}) \qquad \text{ on $(R, \infty) \times \Sigma$}.
\end{equation}
\end{lemma}
\begin{proof}
We have $h^*(\phn) = \phc + \xi + d\zeta$, where by hypothesis $\xi + d\zeta \in \wtht$ with respect to $\phc$. By Lemma~\ref{quadlemma}, we see that
\begin{align*}
h^*(\psn) = \Theta(h^*(\phn)) \, & = \, \psc - \stc (\xi + d\zeta) + F_{\phc} (\xi + d\zeta) \\ & = \, \psc - \stc \xi - \dsc (\stc \zeta) + F_{\phc} (\xi + d\zeta),
\end{align*}
where we have used the fact that $\stc d = \dsc \stc$ on $2$-forms. From equation~\eqref{conedleq}, we know that $\xi$ is coclosed, since $\xi$ is a closed and coclosed form on the link $\Sigma$. Therefore if we take $\dsc$ of both sides of the above equation, we get $\dsc (h^*(\psn)) = \dsc (F_{\phc} (\xi + d\zeta))$, and from~\eqref{quadeq2} and the fact that $|\xi + d\zeta|_{\gc} = O(r^{-3})$, we see that
\begin{equation*}
| \dsc (h^*(\psn)) |_{\gc} \, = \, |  \dsc (F_{\phc} (\xi + d\zeta)) |_{\gc} \, \leq \, C r^{-7}.
\end{equation*}
Equation~\eqref{ACgaugefixlemmaeq} now follows from the fact that $(d + \dsc) (h^*(\psn) - \psc) = \dsc(h^*(\psn))$.
\end{proof}
\begin{cor} \label{ACgaugefixcor}
Suppose that $N$ is an AC $\G$~manifold with rate $\nu \leq -3$, and that $h$ satisfies the gauge-fixing condition. Then $\psn - (h^{-1})^*(\psc)$ satisfies the equation
\begin{equation} \label{gaugefixedtosolveeq}
(d + \dsn) (\psn - (h^{-1})^*(\psc)) \, = \, \dsn \upsilon \qquad \text{ on $h((R, \infty) \times \Sigma)$},
\end{equation}
for some \emph{exact} $4$-form $\upsilon$ defined on the subset $h((R, \infty) \times \Sigma)$ of $N$, with $|\dsn \upsilon|_{\gn} = O(r^{-7})$.
\end{cor}
\begin{proof}
We begin by observing that
\begin{equation*}
| (d + \dsn) (\psn - (h^{-1})^*(\psc)) |_{\gn} \, = \, | (d + \dshn) (h^*(\psn) - \psc) |_{h^*(\gn)}
\end{equation*}
where $\dshn$ is the coderivative with respect to the metric $h^*(\gn)$ on $(R, \infty) \times \Sigma$. Now by~\eqref{ACdefneq2}, the tensor $h^*(\gn) - \gc$ is uniformly bounded with respect to the $\gc$ metric, on $(R, \infty) \times \Sigma$. Thus $| \omega |_{h^*(\gn)} \leq C | \omega |_{\gc}$ for any tensor $\omega$, and therefore
\begin{equation*}
| (d + \dsn) (\psn - (h^{-1})^*(\psc)) |_{\gn} \, \leq \, C| (d + \dshn) (h^*(\psn) - \psc) |_{\gc}.
\end{equation*}
But we also have
\begin{equation*}
| (d + \dshn) (h^*(\psn) - \psc) |_{\gc} \, \leq \, | (d + \dsc) (h^*(\psn) \! - \psc) |_{\gc} + | (\dshn - \dsc) (h^*(\psn) - \psc) |_{\gc}.
\end{equation*}
The first term on the right hand side is $O(r^{-7})$ by Lemma~\ref{ACgaugefixlemma}. For the second term, we use~\eqref{ACdefneq2} again to see that
\begin{equation*}
 | (\dshn - \dsc) (h^*(\psn) - \psc) |_{\gc} \, \leq \, C r^{\nu - 1} | h^*(\psn) - \psc |_{\gc} \, = \, O(r^{-7}),
\end{equation*}
since $| h^*(\psn) - \psc |_{\gc} = O(r^{-3})$ and $\nu \leq -3$. Putting these all together gives
\begin{equation*}
| (d + \dsn) (\psn - (h^{-1})^*(\psc)) |_{\gn} \, \leq \, C r^{-7}.
\end{equation*}
It is clear that $(d + \dsn) (\psn - (h^{-1})^*(\psc)) = \dsn \upsilon$ with $\upsilon = (h^{-1})^*(\psc)$, and we know that $\upsilon$ is exact since $\psc$ is exact by Proposition~\ref{coneformsexactprop}.
\end{proof}

In order to obtain a nice asymptotic expansion of $h^*(\phn)$ and $h^*(\psn)$, we will use Corollary~\ref{asymptoticexpansioncor2}, but first we require one more preliminary result.
\begin{lemma} \label{ACgaugefixlastpreliminarylemma}
There exists a closed $4$-form $\omega_4$ in $L^p_{1, -5}(\Lambda^4(T^* N))$ such that
\begin{equation} \label{ACgaugefixlastpreliminarylemmaeq}
(d + \dsn) (\psn - (h^{-1})^*(\psc) - \omega_4) \, = \, 0
\end{equation}
on the subset $h( (2R, \infty) \times \Sigma)$ of $N$.
\end{lemma}
\begin{proof}
Let $\vartheta$ be a smooth bounded function on $N$ such that
\begin{equation} \label{varthetaeq}
\vartheta(x) = \begin{cases} 0 & \text{for } x \in L, \\ 1 & \text{for } x \in h( (2R, \infty) \times \Sigma).
\end{cases}
\end{equation}
By Corollary~\ref{ACgaugefixcor}, we know that $\upsilon = d \alpha$ for some $3$-form $\alpha$ on $h( (R, \infty) \times \Sigma)$. Now the form $\dsn d (\vartheta \alpha)$ is a well defined $3$-form on all of $N$ which is coexact. It is identically equal to zero on the compact subset $L$ of $N$, and equals $\dsn \upsilon = \dsn d \alpha$ on the subset $h( (2R, \infty) \times \Sigma)$ of $N$. By item ii) of Remark~\ref{ACSobolevdefnrmk}, we know that $\dsn(\vartheta \upsilon)$ is in $L^p_{0, - 7 + e}(\Lambda^{\text{odd}}(T^* N))$ for any $e > 0$. Consider the operator
\begin{equation} \label{ACgaugefixlastpreliminarylemmatempeq}
(d + \dsn)^{\text{even},p}_{1, \lambda} \, : \, L^p_{1, \lambda} (\Lambda^{\text{even}} (T^* N)) \, \to \, L^p_{0, \lambda - 1} (\Lambda^{\text{odd}} (T^* N)).
\end{equation}
We claim that for any $\lambda > -6 + e$ which is \emph{not} a critical rate, we can solve the equation
\begin{equation} \label{gaugefixmustsolveeq}
(d + \dsn)^p_{1, \lambda} (\omega) = \dsn d(\vartheta \alpha),
\end{equation}
where $\omega = \sum_{k=0}^3 \omega_{2k}$ is an even-degree mixed form. Note that this makes sense because since $\dsn d(\vartheta \alpha)$ is in $L^p_{0, - 7 + e}(\Lambda^{\text{odd}}(T^* N))$, it follows from item i) of Remark~\ref{ACSobolevdefnrmk} that $\dsn d(\vartheta \alpha)$ is in $L^p_{0, \lambda -1}(\Lambda^{\text{odd}}(T^* N))$ for any $\lambda > -6 + e$. To know that~\eqref{gaugefixmustsolveeq} is solvable for such $\lambda$, we need to use Theorem~\ref{fredholmalternativethm2}. It will be solvable if and only if $\lambda$ is not critical and $\dsn d(\vartheta \alpha)$ is $L^2$-orthogonal to the kernel of the adjoint of~\eqref{ACgaugefixlastpreliminarylemmatempeq}, which by~\eqref{adjointmapeq2} is
\begin{equation*}
(d + \dsn)^{\text{odd},q}_{m+1, -\lambda - 6} \, \, : \, \, L^q_{m+1, -\lambda - 6} ( \Lambda^{\text{odd}} (T^* N)) \to L^q_{m , -\lambda - 7} ( \Lambda^{\text{even}} (T^* N )).
\end{equation*}
Suppose $\beta = \sum_{k=0}^3 \beta_{2k+1}$ is in the kernel of $(d + \dsn)^q_{m+1, -\lambda - 6}$. Then in particular $d \beta_3 + \dsn \beta_5 = 0$. We have
\begin{align*}
{\langle \beta , \dsn d (\vartheta \alpha) \rangle}_{L^2} \, & = \, {\langle \beta_3 , \dsn d (\vartheta \alpha) \rangle}_{L^2} \, = \, {\langle d \beta_3 , d (\vartheta \alpha) \rangle}_{L^2} \\ & = \, - {\langle \dsn \beta_5 , d (\vartheta \alpha) \rangle}_{L^2} \, = \, - {\langle \beta_5 , d d (\vartheta \alpha) \rangle}_{L^2} \, = \, 0,
\end{align*}
and thus equation~\eqref{gaugefixmustsolveeq} has a solution $\omega = \sum_{k=0}^3 \omega_{2k}$ in 
$L^p_{1, \lambda} (\Lambda^{\text{even}} (T^* N))$. This $\omega$ satisfies
\begin{equation} \label{ACgaugefixlastpreliminarylemmatempeq2}
d \omega_0 + \dsn \omega_2 = 0, \qquad d \omega_2 + \dsn \omega_4 = \dsn d (\vartheta \alpha), \qquad d \omega_4 + \dsn  \omega_6 = 0, \qquad d \omega_6  = 0.
\end{equation}
Every term in~\eqref{ACgaugefixlastpreliminarylemmatempeq2} lies in $L^p_{0, \lambda -1}(\Lambda^{\text{odd}}(T^* N))$. Therefore if $\lambda - 1 < - \frac{7}{2}$, we can use~\eqref{L2equiveq2} to integrate by parts and conclude that
\begin{equation*}
d \omega_0 = \dsn \omega_2 = 0, \qquad d \omega_2 = \dsn \omega_4 - \dsn d (\vartheta \alpha) = 0, \qquad d \omega_4 = \dsn  \omega_6 = 0, \qquad d \omega_6  = 0.
\end{equation*}
Hence in particular $\omega_4$ is closed and $(d + \dsn)(\omega_4) = \dsn d (\vartheta \alpha)$ if $\lambda \in (- 6 + e, 1 - \frac{7}{2})$. If we choose $\lambda$ to be any non-critical rate less than $-5$, we have $\omega_4 \in L^p_{1, -5}(\Lambda^4(T^* N))$. Then by~\eqref{gaugefixedtosolveeq} and~\eqref{varthetaeq} we see that
\begin{equation*}
(d + \dsn) (\psn - (h^{-1})^*(\psc) - \omega_4) \, = \, 0
\end{equation*}
on the region $h( (2R, \infty) \times \Sigma)$.
\end{proof}
\begin{rmk} \label{omega4smoothrmk}
The elliptic regularity result in Theorem~\ref{ellipticregthm2}, combined with the smoothness of $\psc$ and $\psn$, now tell us that $\omega_4$ is smooth. Therefore we have that
\begin{equation} \label{omega4smoothrmkeq}
|\nabn^j \omega_4|_{\gn} \, = \, O(r^{-5 - j}) \qquad \text{for all $j \geq 0$.}
\end{equation}
\end{rmk}

We can finally prove Theorem~\ref{ACasymptoticexpansionthm}.

\begin{proof}[Proof of Theorem~\ref{ACasymptoticexpansionthm}.]
From Lemma~\ref{ACgaugefixlastpreliminarylemma} we know that $(\psn - (h^{-1})^*(\psc) - \omega_4)$ is in $\ker_{\infty}(d + \dsn)_{-3 + \delta}$ for any $\delta > 0$. Let $\lambda_1 = -4$ and $\lambda_2 = -3$. Since $\nu \leq -3$, we have $\nu <  \lambda_1 - \lambda_2 = -1$, and thus we can apply Corollary~\ref{asymptoticexpansioncor2} to obtain
\begin{equation} \label{ACasymptoticexpansionthmtempeq1}
h^*(\psn) - \psc - h^*(\omega_4) \, = \, \sum_{j=1}^N \upsilon_j + \upsilon',
\end{equation}
where $\upsilon_j$ is in $\mathcal K(\alpha_j)_{d + \dsc}$ for $\alpha_j$ a critical rate in the interval $[-4, -3]$, and $| \upsilon'|_{\gc} = O(r^{-4 - \delta})$. Now \emph{a priori}, the $\upsilon_j$'s are even-degree mixed forms. However, since the left hand side of~\eqref{ACasymptoticexpansionthmtempeq1} consists of only $4$-forms, we know that the $\upsilon_j$'s and $\upsilon'$ are also pure $4$-forms. Thus each $\upsilon_j$ is of the form $\upsilon_j = \sum_{i=1}^m (\log(r))^i \upsilon_{j,i}$ where $\upsilon_{j,i}$ is a homogeneous $4$-form of order $\alpha_j$ on the cone. From the calculation in the proof of Proposition~\ref{dirachomokernellogsprop}, it follows that the leading coefficient $\upsilon_{j,m}$ is a $4$-form, homogeneous of order $-4$, and closed and coclosed. But then Corollary~\ref{homodiracexcludecor} says that $\upsilon_{j,m} = 0$ (and hence $\upsilon_j=0$) for any $\alpha_j \in (-4,-3)$. So the only critical rates which can occur in~\eqref{ACasymptoticexpansionthmtempeq1} are $\alpha_1 = -4$ and $\alpha_2 = -3$. Proposition~\ref{dirachomokernellogsprop2} and Corollary~\ref{homocccor} tell us that $\alpha_1 = \eta$, a homogeneous $4$-form of order $-4$, harmonic with respect to the cone metric. Also, Propositions~\ref{dirachomokernellogsprop} and~\ref{dirachomokernelprop} tell us that $\alpha_2 = dr \wedge (\sts \xi)$ for some homogeneous $3$-form $\sts \xi$ of order $-3$, harmonic with respect to the cone metric. (We choose to write is at $\sts \xi$ for convenience.) Equation~\eqref{conestareq} says we can write $dr \wedge (\sts \xi) = - \stc \xi$. Now $\xi$ is homogeneous of order $-3$, and harmonic, so $\xi$ lies in $\wtht$ with respect to $\phc$, by Proposition~\ref{cones27prop}. Therefore we have shown that~\eqref{ACasymptoticexpansionthmtempeq1} becomes
\begin{equation*}
h^*(\psn) \, = \, \psc + \eta - \stc \xi + \upsilon' + h^*(\omega_4),
\end{equation*}
where $\upsilon' + h^*(\omega_4)$ is closed. We also know that $|\nabc^j(\upsilon' + h^*(\omega_4))|_{\gc} = O(r^{-4 - \delta - j})$ as $r \to \infty$, since $|\nabc^j \upsilon'|_{\gc} = O(r^{-4-\delta - j})$ and $|\nabc^j h^*(\omega_4)|_{\gc} = O(r^{-5-j})$, which follows immediately from~\eqref{omega4smoothrmkeq} and~\eqref{ACdefneq2}. Hence we can apply Lemma~\ref{exactformslemma} to say that $\upsilon' + h^*(\omega_4) = d\theta$, for some $3$-form $\theta$ on $(2R, \infty) \times \Sigma$ satisfying $|\nabc^j \theta |_{\gc} = \, \, O(r^{\nu' + 1 - j})$, for any $\nu' \geq -4 - \delta$. This proves equation~\eqref{fourformasymptoticexpansioneqagain} and half of~\eqref{ACinterpeq3newagain}.

Now we apply the operator $\Theta^{-1}$ to equation~\eqref{fourformasymptoticexpansioneqagain}. We obtain
\begin{align*}
h^*(\phn) \, = \, \Theta^{-1}(h^*(\psn)) \, & = \, \Theta^{-1}(\psc - \stc \xi + \eta + d\theta) \\ & = \, \phc + \xi + J_{\phc}(\eta + d\theta) + G_{\phc}(\eta + d\theta)
\end{align*}
using equations~\eqref{quadeq3} and~\eqref{Jdefneq} and the fact that $\xi$ is in $\wtht$ with respect to $\phc$. Now the remainder term $J_{\phc}(\eta + d\theta) + G_{\phc}(\eta + d\theta) = h^*(\phn) - \phc - \xi$ is clearly smooth and closed. Hence~\eqref{Jesteq} and~\eqref{quadeq4}, along with the estimates on $\eta$ and $d\theta$ above tell us that
\begin{equation*}
|\nabc^j(J_{\phc}(\eta + d\theta) + G_{\phc}(\eta + d\theta))|_{\gc} \, =\, O(r^{-4 - j}) \qquad \forall j \geq 0,
\end{equation*}
and therefore we can again apply Lemma~\ref{exactformslemma} to conclude that $h^*(\phn) - \phc - \xi = d\zeta$ for some $2$-form $\zeta$ on $(2R, \infty) \times \Sigma$ satisfying $|\nabc^j \zeta |_{\gc} = \, \, O(r^{\nu' + 1 - j})$, for any $\nu' \geq -4$. This proves equation~\eqref{threeformasymptoticexpansioneqagain} and the other half of~\eqref{ACinterpeq3newagain}. It is clear that $[\xi] = [h^*(\phn)] = \Phi(N)$, and $[\eta] = [h^*(\psn)] = \Psi(N)$ follows from the fact that $-\stc \xi = dr \wedge \sts \xi = d( r \sts \xi)$ is exact.
\end{proof}

\begin{rmk} \label{perversermk}
To prove Theorem~\ref{ACasymptoticexpansionthm}, which is really just a result about $\G$~cones, we transferred the problem to an analytic question on the AC manifold to use the machinery of Section~\ref{lockhartACsec}, and then transferred back to the cone. In principle, we could have avoided that by working directly on the cone itself, but then we would have needed to introduce Lockhart--McOwen type results for non-compact manifolds which had asymptotic ends as well as isolated conical singularities (which can be done), but we preferred to avoid doing so.
\end{rmk}

\begin{rmk} \label{nogaugefixrmk}
If we did not impose the gauge-fixing condition, the $O(r^{-7})$ expression in~\eqref{ACgaugefixlemmaeq} would only be $O(r^{-4})$. Then the $4$-form $\omega_4$ from Lemma~\ref{ACgaugefixlastpreliminarylemma} would be \emph{at best} $O(r^{-3 + e})$, which would be too large to ignore and the proof of Theorem~\ref{ACasymptoticexpansionthm} would break down. Therefore the gauge-fixing condition is \emph{necessary} to be able to prove results like~\eqref{threeformasymptoticexpansioneqagain} and~\eqref{fourformasymptoticexpansioneqagain}.
\end{rmk}

\end{document}